\documentclass[12pt]{amsart}
\usepackage{amsmath,amssymb,amsthm}
\usepackage{color}
\usepackage[mathscr]{euscript}
\usepackage{lineno}
\usepackage[arrow,curve,frame,color]{xy}
\usepackage{hyperref}
\usepackage{graphicx}
\usepackage[labelformat=empty]{caption}
\usepackage{accents}
\numberwithin{equation}{section}
\theoremstyle{plain}
\definecolor{orange}{rgb}{1.0,0.3,0}
\newtheorem{theorem}[equation]{Theorem}
\newtheorem{thm}[equation]{Theorem}
\newtheorem{proposition}[equation]{Proposition}
\newtheorem{lemma}[equation]{Lemma}
\newtheorem{corollary}[equation]{Corollary}

\theoremstyle{remark}
\newtheorem{remark}[equation]{Remark}

\theoremstyle{definition}
\newtheorem{definition}[equation]{Definition}
\newtheorem{defn}[equation]{Definition}

\newtheorem{question}[equation]{Question}
\newtheorem*{question*}{Question}

\newtheorem{example}[equation]{Example}

\newcommand{\C}{{\mathcal C}}

\newcommand{\h}{{\mathcal H}}

\renewcommand{\L}{{\mathcal L}}

\newcommand{\N}{\mathbb N}

\newcommand{\R}{\mathbb R}
\newcommand{\calr}{{\mathcal R}}

\newcommand{\Z}{\mathbb Z}

\newcommand{\acts}{\curvearrowright}

\newcommand{\diam}{\operatorname{diam}}
\newcommand{\dist}{\operatorname{dist}}

\newcommand{\gen}{\operatorname{gen}}

\newcommand{\id}{\operatorname{id}}

\newcommand{\im}{\operatorname{Im}}
\newcommand{\Int}{\operatorname{Int}}

\newcommand{\length}{\operatorname{length}}

\newcommand{\Lip}{\operatorname{Lip}}
\newcommand{\LIP}{\operatorname{LIP}}

\newcommand{\Star}{\operatorname{St}}

\newcommand{\al}{\alpha}

\newcommand{\D}{\partial}

\newcommand{\De}{\Delta}

\newcommand{\eps}{\varepsilon}
\newcommand{\ga}{\gamma}
\newcommand{\Ga}{\Gamma}

\newcommand{\la}{\lambda}

\newcommand{\lra}{\longrightarrow}

\newcommand{\on}{\:\mbox{\rule{0.1ex}{1.2ex}\rule{1.1ex}{0.1ex}}\:}
\newcommand{\ra}{\rightarrow}
\newcommand{\restr}{\mbox{\Large \(|\)\normalsize}}

\newcommand{\si}{\sigma}

\newcommand{\ul}{\underline}

\def\Xint#1{\mathchoice
   {\XXint\displaystyle\textstyle{#1}}%
   {\XXint\textstyle\scriptstyle{#1}}%
   {\XXint\scriptstyle\scriptscriptstyle{#1}}%
   {\XXint\scriptscriptstyle\scriptscriptstyle{#1}}%
   \!\int}
\def\XXint#1#2#3{{\setbox0=\hbox{$#1{#2#3}{\int}$}
     \vcenter{\hbox{$#2#3$}}\kern-.5\wd0}}

\def\av{\Xint-}

\def\ccgroup{{\mathbb G}}
\def\horalg{{\mathfrak{h}}}
\DeclareMathOperator\sgn{sgn}
\DeclareMathOperator\spt{spt} 
\def\pcreatedst#1#2{\expandafter\def\csname #1dst\endcsname##1,##2.{#2(##1,##2)}
\expandafter\def\csname #1dstp\endcsname##1.{{#2}_{##1}}
\expandafter\def\csname #1dname\endcsname{#2}} 
\def\pcreatenrm#1#2{\expandafter\def\csname
  #1nrm\endcsname##1.{\left\|##1\right\|_{#2}} \expandafter\def\csname
  #1nrmname\endcsname{\left\|\,\cdot\,\right\|_{#2}}} 
\pcreatedst{h}{d_{\text{\normalfont H}}} 
\pcreatedst{hz}{d_{Z,\text{\normalfont H}}} 
\pcreatedst{r}{\varrho} 
\pcreatedst{z}{d_Z} 
\pcreatedst{w}{d_W} 
\pcreatedst{inv}{d_{Y_\infty}}
\pcreatedst{invkone}{d_{Y_{k+1}}}
\pcreatedst{invk}{d_{Y_{k}}}
\pcreatedst{x}{d_X} 
\pcreatedst{y}{d_Y} 
\pcreatedst{F}{\varrho_F} 
\pcreatedst{car}{d_{\ccgroup}} 
\pcreatenrm{dx}{\dset,\rdname}
\pcreatenrm{hor}{\horalg}
\pcreatenrm{r}{\rdname}
\pcreatenrm{gen}{}
\pcreatenrm{al}{\alpha{}}
\pcreatenrm{be}{\beta{}}
\pcreatenrm{ta}{\tau{}}
\pcreatenrm{om}{\omega{}}
\pcreatenrm{Om}{\Omega{}}
\pcreatenrm{rLIP}{{\rdname,\text{\normalfont LIP}}}
\pcreatenrm{rLip}{{\rdname,\text{\normalfont Lip}}}
\pcreatenrm{tom}{\tilde\Omega{}}
\pcreatenrm{T}{T}
\pcreatenrm{t}{TX}
\pcreatenrm{ct}{T^*X}
\pcreatenrm{tz}{TZ}
\pcreatenrm{f}{F}
\pcreatenrm{ltwo}{l^2}
\pcreatenrm{empty}{}
\pcreatenrm{al}{\alpha}
\def\albrep#1.{{\mathcal A}_{#1}} 
\def\mpush#1.{{#1}_{\sharp}} 

\def\tnorm#1.{{{\left\|#1\right\|_{\Lip^*}}}}
\def\ptnorm#1,#2.{{{\left\|#1\right\|_{#2,\Lip^*}}}}
\def\cotnorm#1.{{{\left\|#1\right\|_{\Lip}}}}
\def\lebmeas#1.{\setbox1=\hbox{$#1$\unskip}{\mathcal L}^{\ifdim\wd1>0pt
    #1 \else 1 \fi}} 
\def\glip#1.{{\bf L}(#1)} 

\def\lipalg#1.{{\rm Lip}_{\text{\normalfont b}}(#1)} 
\def\lipfun#1.{{\rm Lip}(#1)} 
\def\lipalgb#1,#2.{{\rm Lip}_{\text{\normalfont b},#2}(#1)}
\def\ulipalg#1.{{\rm Lip}_1(#1)} 
\def\ball#1,#2.{B(#1,#2)} 
\def\clball#1,#2.{\bar B(#1,#2)} 
\def\dst#1,#2.{d(#1,#2)} 
\def\dstp#1.{d_{#1}} 
\def\hmeas#1.{\mathscr{H}^{\setbox0=\hbox{$#1\unskip$}\ifdim\wd0=0pt 1
    \else #1\fi}} 
\def\rhmeas#1.{\mathscr{H}^{\setbox0=\hbox{$#1\unskip$}\ifdim\wd0=0pt 1
    \else #1\fi}_{\rdname}} 
\def\natural{{\mathbb N}} 
\let\zahlen=\integers
\def\real{{\mathbb R}} 
\def\cbb#1.{\text{\normalfont CBB}({\setbox0=\hbox{$#1\unskip$}\ifdim\wd0=0pt 0
    \else #1\fi})}              
\def\cba#1.{\text{\normalfont CBA}({\setbox0=\hbox{$#1\unskip$}\ifdim\wd0=0pt 0
    \else #1\fi})}              
\def\cat#1.{\text{\normalfont CAT}({\setbox0=\hbox{$#1\unskip$}\ifdim\wd0=0pt 0
    \else #1\fi})}              
\def\dset{D_X}
\pcreatenrm{dxdx}{\dset,\xdname}

\def\dirdst#1,#2,#3.{\setbox1=\hbox{$#1$\unskip}\setbox2=\hbox{$#2$\unskip}\setbox3=\hbox{$#3$\unskip}
{\hat d_{\ifdim\wd3>0pt
    #3\else\infty\fi}}{\ifdim\wd1>0pt\left(#1,{\ifdim\wd2>0pt #2 \else\cdot\fi}\right)\fi}}
\def\dirprj#1,#2,#3.{\setbox1=\hbox{$#1$\unskip}\setbox2=\hbox{$#2$\unskip}\setbox3=\hbox{$#3$\unskip}
    {\ifdim\wd3>0pt (\fi}{\hat\pi_{\ifdim\wd1>0pt #1 \else
        \fi}^{\ifdim\wd2>0pt #2\else\infty\fi}} {\ifdim\wd3>0pt
      )^{-1}\fi}}
  \def\ndirprj#1,#2,#3.{\setbox1=\hbox{$#1$\unskip}\setbox2=\hbox{$#2$\unskip}\setbox3=\hbox{$#3$\unskip}
    {\ifdim\wd3>0pt (\fi}{\pi_{\ifdim\wd1>0pt #1 \else
        0\fi}^{\ifdim\wd2>0pt #2\else\infty\fi}} {\ifdim\wd3>0pt
      )^{-1}\fi}}
    \def\nbdirprj#1,#2,#3.{\setbox1=\hbox{$#1$\unskip}\setbox2=\hbox{$#2$\unskip}\setbox3=\hbox{$#3$\unskip}
    {\ifdim\wd3>0pt (\fi}{\bar\pi_{\ifdim\wd1>0pt #1 \else
        0\fi}^{\ifdim\wd2>0pt #2\else\infty\fi}} {\ifdim\wd3>0pt
      )^{-1}\fi}}
  \def\zahlen{{\mathbb Z}} 
   
    \def\zset#1,#2.{J(#1,#2)}
  \def\infzset#1,#2.{\underline J(#1,#2)}
\def\hmeas#1.{\mathscr{H}^{\setbox0=\hbox{$#1\unskip$}\ifdim\wd0=0pt Q
    \else #1\fi}} 

\begin{document}

\title[PI spaces with analytic dimension 1]{PI spaces with analytic dimension $1$ and arbitrary topological dimension}
\author{Bruce Kleiner}
\author{Andrea Schioppa}
\thanks{The first author was supported by a Simons Fellowship and NSF grant
DMS-1405899.
The second author was supported by the ``ETH Zurich Postdoctoral Fellowship Program and the Marie Curie Actions
for People COFUND Program''}
\date{\today}
\maketitle

\begin{abstract}
For every $n$, we construct
 a  metric measure space that is doubling,  satisfies a 
Poincare inequality in the sense of Heinonen-Koskela, has   topological
dimension $n$, and has a measurable tangent bundle of dimension $1$.
\end{abstract}

\tableofcontents

\section{Introduction}
Since they were introduced in \cite{heinonen_koskela}, 
PI spaces (metric measure
spaces that are doubling and satisfy a Poincar\'e inequality) have been
investigated extensively, leading to  progress in many 
directions.
In spite of this, there remains a gap between the structural 
constraints imposed by the existing
theory  and the properties exhibited by
known examples.  On the one hand existing examples come from a 
variety of different sources:
\begin{enumerate}
\item
Sub-Riemannian manifolds.
\item Limits of sequences of Riemannian manifolds with a lower 
bound on the Ricci curvature \cite{cheeger_colding_ricci3}.
\item Metric measure spaces satisfying 
  synthetic Ricci curvature conditions \cite{rajala}.
\item  Certain Ahlfors-regular topological
manifolds \cite{semmes}.
\item Examples with ``small'' singular sets 
\cite{laakso_a_infty,semmes_no_good_parametrizations}.
\item Boundaries of hyperbolic groups \cite{bourpaj}.
\item Quotient constructions \cite{laakso}.
\item Inverse limit constructions \cite{piex}.
\item Spaces obtained from the above by taking products, 
gluing \cite[Thm 6.15]{heinonen_koskela},  or passing to 
nice subsets  \cite{mackay_tyson_wildrick}. 
\end{enumerate}
On the other hand, in some respects this list is somewhat limited.  For instance,
the examples in  (2), (4), (5) are rectifiable, and  many in (3) are known to be
rectifiable \cite{mondino_naber,gigli,gigli_mondino_rajala}, while those
in (6), (7) and (8) admit a common description as  limits of inverse systems  
and have very similar properties.  
Moreover, if $(X,\mu)$ denotes one of  the above examples,
then the infinitesimal structure of $(X,\mu)$ has a special
form, in the sense that when one blows-up  $(X,\mu)$ at $\mu$-a.e. point, one
gets a metric measure space that
is bilipschitz equivalent to the product of a Carnot group with an example as in (8).

Our purpose in this paper is to  
construct PI spaces that have different characteristics from previously
known examples.
Before stating our theorem, we recall that any PI space $(X,\mu)$
has a measurable (co)tangent bundle \cite{cheeger}; we will refer to its
dimension as the {\em analytic dimension} of $(X,\mu)$.

\begin{theorem}
\label{thm_top_dim_n_anal_dim_1}
For every $n$, there is a complete self-similar PI space $(X_\infty,\mu_\infty)$ with
analytic dimension $1$ and topological dimension $n$.
Furthermore, 
for some $\al\in (0,1)$:
\begin{enumerate}
\item  
There is a surjective
David-Semmes regular map $\hat\pi^\infty:(\R^n,d_\al)\ra X_\infty$,
where $d_\al$ is the partial snowflake metric on $\R^n$ given by 
$$d_\al(p,p')=|p_1-p_1'|+\sum_{i=2}^n|p_i-p_i'|^\al\,.$$
In particular, 
letting $\L^n$ denote Lebesgue measure,
 for $Q=1+(n-1)\al^{-1}$,
the pushforward measure 
  $\mu_\infty=\hat\pi^\infty_{\#}(\lebmeas n.)$
is comparable to $Q$-dimensional Hausdorff measure $\h^Q$ on $X_\infty$, and 
$X_\infty$ is Ahlfors $Q$-regular.
\item $X_\infty$ has topological and Assouad-Nagata dimension $n$
(see Section \ref{sec_assouad_nagata}).
\item $(X_\infty,\h^Q)$ satisfies a $(1,1)$-Poincar\'e inequality
   (see
  Section~\ref{sec_poincare_inequality}).
\item 
\label{item_analytic_dimension_1}
$(X_\infty,\h^Q)$ has  analytic dimension $1$: there is a single Lipschitz
function $x_\infty:X_\infty\ra\R$ that is a chart on all of $X_\infty$, i.e. it defines
the measurable differentiable structure for $(X_\infty,\h^Q)$
 (see Section~\ref{sec:anal_dim}).
\item 
\label{item_horizontal_a_rep}
Let $\hat\Ga$ be the family of horizontal lines in 
$\R^n$, equipped with the obvious measure.
Then the pushforward of $\hat\Ga$ under the map 
$\hat\pi^\infty:\R^n\ra X_\infty$ gives a
universal Alberti representation in the sense of
\cite{bate} for $(X_\infty,\mu_\infty)$  (see Section~\ref{sec:anal_dim}).
\item 
\label{item_ditto_for_limits}
If $\{p_k\}\subset X_\infty$, $\la_k\subset (0,\infty)$ are arbitrary 
sequences, and $(Z,z)$ is a pointed Gromov-Hausdorff limit of the 
sequence 
$\{(\la_kX_\infty,p_k)\}$ of pointed rescalings of $X_\infty$, then (2)-(4) hold for
$(Z,z)$.   Moreover, there is a collection of at most $N=N(n)$
David-Semmes regular
maps $(\R^n,d_\al)\ra X_\infty$ whose images cover $X_\infty$. 
\end{enumerate}
\end{theorem}

For comparison, we note that all
the previously known examples with analytic dimension $1$ 
have topological dimension $1$ (see \cite{bourpaj,laakso,piex}). 

We refer the reader to Section \ref{sec_overview}
for an overview of the proof of Theorem \ref{thm_top_dim_n_anal_dim_1},
and to Section \ref{sec_generalizations} for some generalizations.

We now pose some questions 
concerning the relation between the topological  and the analytical structure
of PI spaces.

The  examples  in Theorem \ref{thm_top_dim_n_anal_dim_1}
have small analytic dimension and arbitrarily large topological dimension. 
 One may ask if the topological
dimension can be small while the analytic dimension is large.  This is not
an interesting question, though:  there are compact subsets
$X\subset [0,1]^n$ with positive Lebesgue measure such that the 
metric measure space $(X,\L^n)$ is a PI space 
with analytic dimension $n$ 
and topological dimension $1$ (see \cite{mackay_tyson_wildrick}
for the $n=2$ case).  
Nonetheless, such examples are rectifiable and look on the small scale like $\R^n$
itself, in the sense that typical blow-ups are copies of 
$\R^n$.  This motivates the following revised version of the above question:
\begin{question}
\label{ques_analytic_dim_n}
Pick $n\geq 2$.  Is there a  PI space of analytic dimension $n$
and   Assouad-Nagata dimension $1$?
\end{question}
\noindent
The
Assouad-Nagata dimension is a  notion  that is metric based, 
scale invariant,   well behaved
with respect to Gromov-Hausdorff limits, and is bounded below by the 
topological dimension (see Section \ref{sec_assouad_nagata});    
in particular, the rectifiable examples mentioned above 
have Assouad-Nagata dimension $n$. 
Rather than using Assouad-Nagata dimension in Question \ref{ques_analytic_dim_n}, 
one could require instead 
that every blow-up (i.e. weak tangent) of $X$ has topological dimension $1$.

The spaces in Theorem \ref{thm_top_dim_n_anal_dim_1} have complicated local topology.
One may wonder if there are examples with similar properties that are topological
manifolds:
\begin{question}
Pick $n\geq 2$.  Suppose  $(X,\mu)$ is a PI space homeomorphic to $\R^n$, such
that 
any pointed Gromov-Hausdorff limit of any sequence of
pointed rescalings of $X$ is also homeomorphic to $\R^n$.
What are the possibilities for the analytic dimension of $(X,\mu)$?  The 
Hausdorff dimension?
For instance,
is there  such a PI space homeomorphic to $\R^3$ (or even $\R^2$)
with analytic dimension $1$?
Or one which is Ahlfors $Q$-regular for large $Q$?
\end{question}

\subsection*{Organization of the paper}
In Section \ref{sec_overview} we define the metric space $X_\infty$ in the 
$n=2$ case and then
discuss some of the key points in
the proof of Theorem \ref{thm_top_dim_n_anal_dim_1}.
The $n=2$ case of Theorem \ref{thm_top_dim_n_anal_dim_1} is proven in  
Sections \ref{sec_cell_structure}-\ref{sec_n_equals_2_case_concluded}, and the case
of general $n$ is treated in Section \ref{sec_n_dim_case}.
In Section \ref{sec_generalizations} we consider a more general class of 
examples of direct systems that have many of the same features.

\section{Overview of the proof of Theorem \ref{thm_top_dim_n_anal_dim_1}}
\label{sec_overview}

In this section we define $X_\infty$ and related objects in the $n=2$ case.  We also 
have an informal discussion of  the
proof of Theorem \ref{thm_top_dim_n_anal_dim_1}.  The proof itself appears
in  Sections 
\ref{sec_cell_structure}-\ref{sec_n_dim_case}.
\par {\bf Standing assumptions:}
The objects and notation introduced in this section will be retained up through Section 
\ref{sec_n_equals_2_case_concluded}.

\subsection*{A combinatorial description of  partial snowflake metrics on 
$\R^2$} Pick integers $m,m_v$ with $2\leq m<m_v$.  

For $j\in\Z$, let $Y_j$ be the cell complex associated with the tiling
of $\R^2$ by the translates of the rectangle $[0,m^{-j}]\times [0,m_v^{-j}]$. Thus 
the translation group
$(m^{-j}\Z)\times(m_v^{-j}\Z)$ acts by cellular isomorphisms on $Y_j$.
Given $k\geq 0$, we may view $Y_{j+k}$ as a $k$-fold iterated 
subdivision of $Y_j$, where 
at each iteration the $2$-cells are
subdivided $m$ times in the horizontal direction
and $m_v$ times in the vertical direction.
Let $\Phi:\R^2\ra\R^2$ be the linear transformation 
$\Phi((x,y))=(m^{-1}x,m_v^{-1}y)$.  Then $\Phi^k:\R^2\ra\R^2$ induces a cellular 
isomorphism $Y_j\ra Y_{j+k}$ for all $j,k\in\Z$.

We now define a metric on $\R^2$ based on the combinatorial structure of
the $Y_j$'s.  Let
 $\hat d_\infty^Y$ be the largest pseudodistance on $\R^2$ with the property
that for every $j$, each cell of $Y_j$ has $\hat d_\infty^Y$-diameter at most
$m^{-j}$.  One readily checks (see Lemma \ref{lem_hat_d_y_partial_snowflake}) that
$\hat d_\infty^Y$ is comparable to the partial snowflake metric 
$$
d_\al((p_1,p_2),(p_1',p_2'))=|p_1-p_1'|+|p_2-p_2'|^\al\,,
$$
where $\al=\frac{\log m}{\log m_v}$.

\subsection*{The definition of $X_\infty$}
We will define the space $X_\infty$ as a quotient of $(\R^2,\hat d_\infty^Y)$, 
where the quotient is generated by certain
identifications that respect the $x$-coordinate.    
One may compare this with Laakso's construction
of PI spaces as quotients of the product $[0,1]\times C$, where $C$ is a Cantor
set \cite{laakso}.    
Henceforth we will call
a  $1$-cell of $Y_j$  {\bf horizontal} (respectively {\bf vertical}) if it is a 
translate of
$[0,m^{-j}]\times \{0\}$ (respectively $\{0\}\times[0,m_v^{-j}]$).

Choose a large integer $L$ (e.g. $L=100$), and set $m=4$ and $m_v=3L$.  

For all $k,\ell\in \Z$, $i\in \{1,2,3\}$, we define the following pair of
vertical $1$-cells of $Y_1$ (see Figure \ref{fig_modified_gluing}):
 $$
a_{k,l,i}=\left\{\frac{i}{4}+k\right\}\times \left[(3\ell+i-1)m_v^{-1},(3\ell+i)m_v^{-1}\right]\,,
$$
$$ 
a'_{k,l,i}=\left\{\frac{i}{4}+k\right\}\times \left[(3\ell+i)m_v^{-1},(3\ell+i+1)m_v^{-1}\right],
$$
Note that
$a_{k,\ell,i}'$ is the image of $a_{k,\ell,i}$ under the vertical
translation $(x,y)\mapsto (x,y+m_v^{-1})$. 
The collections $\{a_{k,\ell,i}\mid k,\ell\in\Z,\; 1\leq i\leq 3\}$, 
$\{a_{k,\ell,i}'\mid k,\ell\in\Z,\;1\leq i\leq 3\}$ are invariant
under translation by $\Z^2$ and  are contained in the union of vertical lines
$\{(x,y)\in\R^2\mid x\in \frac{1}{4}\Z\setminus \Z\}$.  
\begin{figure}[htb!] 

\begin{center}  
\includegraphics[scale=.4]{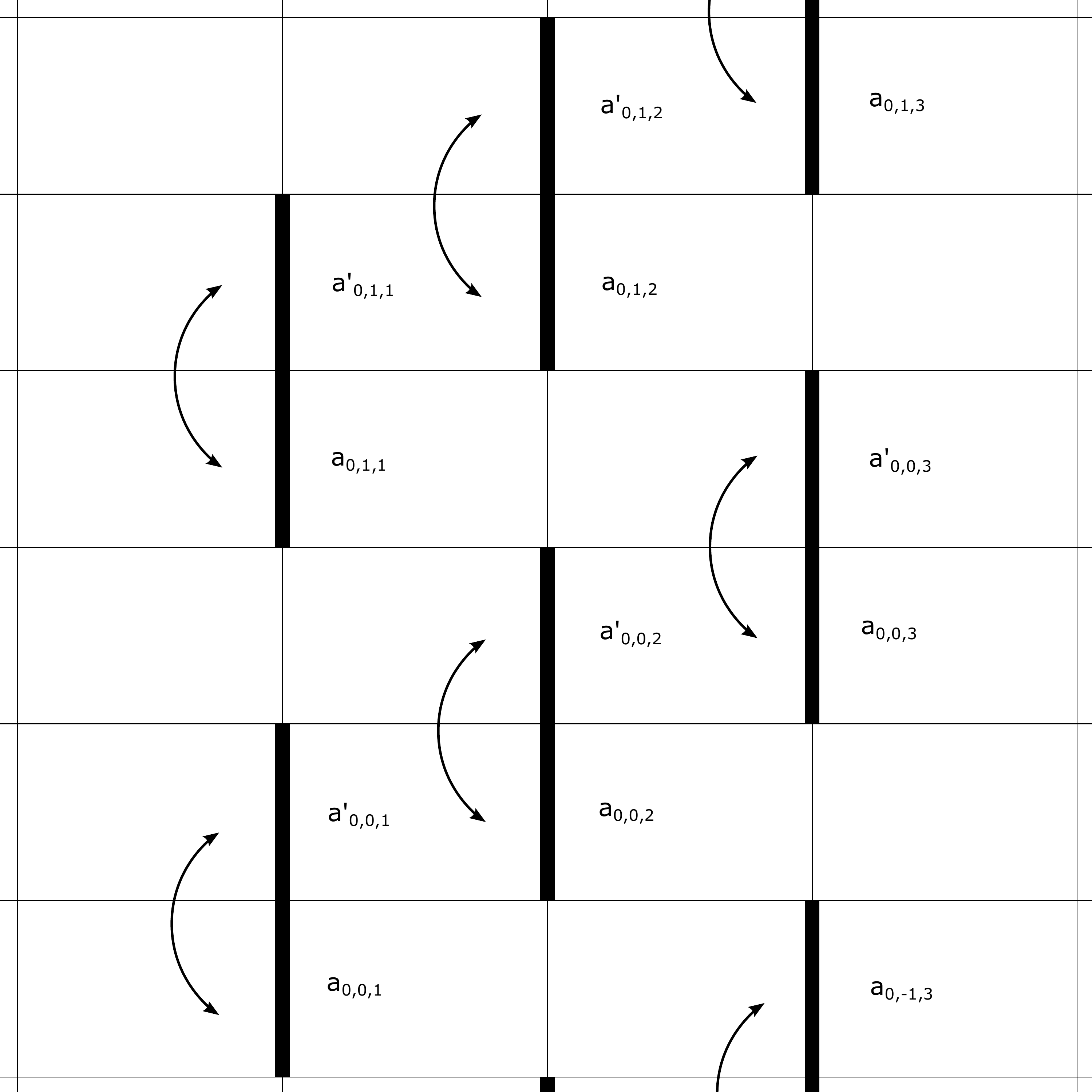} 
\caption{\label{fig_modified_gluing}Figure \ref{fig_modified_gluing}: $L=2$ case}
\end{center}

\end{figure}

Next, we define an equivalence relation
$\calr$ on $\R^2$ by identifying 
$a_{k,\ell,i}$ with $a_{k,\ell,i}'$ by the vertical translation 
$(x,y)\mapsto (x,y+m_v^{-1})$ for all $k,\ell\in \Z$, $1\leq i\leq 3$.
 Note that $\calr$ is invariant under translation by $\Z^2$.  

 Let $\calr_\infty$ be the equivalence relation on $\R^2$
generated by 
the collection of
pushforwards $\Phi^{j}_*\calr$ for all $j\in\Z$.  
We define $X_\infty$ 
to be the
quotient $\R^2/\calr_\infty$, and let $\hat\pi^\infty:\R^2\ra X_\infty$ be the quotient
map.
We metrize $X_\infty$ using the largest pseudodistance $\hat d_\infty^X$ 
on $X_\infty$ such that for every $j\in\Z$ and every
$2$-cell $\hat\si$ of $Y_j$, the  projection $\hat\pi^\infty(\hat\si)\subset X_\infty$
has $\hat d_\infty^X$-diameter at most 
$m^{-j}$.  It is not hard to see that $\hat d_\infty^X$ is the largest pseudodistance
on $X_\infty$ such that 
$\hat\pi^\infty:(\R^2,\hat d_\infty^Y)\ra (X_\infty,\hat d_\infty^X)$ 
is $1$-Lipschitz.  Henceforth we use $\hat d_\infty$ instead of 
$\hat d_\infty^X$ when there is no risk of confusion.

\subsection*{$X_\infty$ as a direct limit of cell complexes}
While the definition of $X_\infty$ as a quotient  $\R^2/\calr_\infty$ 
is transparent, it does not provide a convenient framework for understanding
the structure of $X_\infty$.  Instead, we will
analyze $X_\infty$ by representing it as a direct limit.  

For every $j\in\Z$, let $\calr_j$ be the equivalence relation on 
$\R^2$ generated by $\Phi^i_*\calr$ for all $i\in \Z$ with $i < j$, 
 and let  
$X_j$ be the quotient $\R^2/\calr_{j}$ equipped with quotient topology. 
Since $\calr_{j-1}\subset \calr_j$,
the quotient maps induce a direct system of topological spaces
$$
\ldots\stackrel{\pi_{-1}}{\lra}
X_0\stackrel{\pi_0}{\lra}X_1\stackrel{\pi_1}{\lra}\ldots
\stackrel{\pi_{j-1}}{\lra}X_j\stackrel{\pi_j}{\lra}\ldots
$$
For all $i\in \Z$, $j\in \Z\cup\{\infty\}$
with  $i\leq j$,  we denote the projection map $X_i\ra X_j$  by $\pi_i^j$,
the quotient map $\R^2\ra X_j$ by $\hat\pi^j$, and the
composition $\pi_i^j\circ\hat\pi^i: Y_i\ra X_j$
by $\hat \pi_i^j$. 

We metrize $X_j$ using
the  largest pseudodistance $\hat d_j$ on $X_j$ such that 
for every 
$i\leq j$ and every cell $\hat\si$ of $Y_i$, the image $\hat\pi^j(\hat\si)\subset X_j$
has $\hat d_j$-diameter at most $m^{-i}$.

\begin{remark}
Our
examples were partly inspired by \cite[Sec. 11]{piex}, which
gives  a construction of  PI spaces of
 topological and analytic
dimension $n$.  In fact, $X_\infty$
arose when we attempted to find an ``anisotropic'' variant of the cube
complex examples of \cite[Sec. 11]{piex}.  
However, this leads to a situation where the projection map
$\pi_0^\infty: X_\infty\ra [0,1]^n$ 
is Lipschitz with respect to the 
partial snowflake metric on $[0,1]^n$, and is moreover a light map.
This is incompatible with
the existence of a Poincar\'e inequality in $X_\infty$. 
\end{remark}

\subsection*{Discussion of the proof}
We now give an indication of some of the key points in the proof of 
Theorem \ref{thm_top_dim_n_anal_dim_1}.

The first part of the proof, which appears
in Sections \ref{sec_cell_structure}-\ref{sec_metric_structure},
develops the combinatorial and metric structure of the direct system
$\{X_j\}$.  
Because the equivalence relation $\calr_j$ may be generated by 
identifying pairs of 
vertical $1$-cells of the cell complex $Y_j$ by vertical translations,
the cell structure of $Y_j$ descends to a cell structure on $X_j$.  
This cell complex has controlled combinatorics;
because the distance $\hat d_j$ is defined  combinatorially, this
implies that $(X_j,\hat d_j)$ is doubling at 
the scale $m^{-j}$.  A key estimate 
(Proposition \ref{prop_distance_lower_bound}) 
compares $(X_j,\hat d_j)$ with $(X_\infty,\hat d_\infty)$
is that for every $p,p'\in X_j$ one
has 
\begin{equation}
\label{eqn_compression_overview}
\hat d_j(p,p')-2m^{-j}\leq \hat d_\infty(\hat\pi_j^\infty(p),\hat \pi_j^\infty(p'))
\leq \hat d_j(p,p')\,.
\end{equation}
In particular, this implies that
the sequence $\{X_j\}$ Gromov-Hausdorff converges to $X_\infty$.  

With the foundation laid in Sections \ref{sec_cell_structure}-\ref{sec_metric_structure},
several parts of Theorem \ref{thm_top_dim_n_anal_dim_1} follow fairly
easily:  
\begin{itemize}
\item By a 
short argument
(Lemma \ref{lem_david_semmes_regular_map}), 
one deduces   the David-Semmes regularity of
the projection  $\hat\pi^\infty:(\R^2,\hat d_\infty^Y)\ra X_\infty$, which
gives part (1) of Theorem \ref{thm_top_dim_n_anal_dim_1}. 
\item The restriction of the projection
$\hat\pi^\infty:\R^2\ra X_\infty$ to the boundary 
of the unit square $\D [0,1]^2$ is injective.  By an elementary topological
argument this implies that the topological dimension of $X_\infty$ is
at least $2$ (see Lemma \ref{lem_david_semmes_regular_map}).
\item Using (\ref{eqn_compression_overview}), one shows that there is a ``good cover''
of $X_\infty$ whose inverse image under $\pi_j^\infty:X_j\ra X_\infty$
approximates the decomposition of $X_j$ into open cells.  This proves that the
Assouad-Nagata dimension of $X_\infty$ is at most $2$ (see Theorem
\ref{thm:nagata_bound}).
 \end{itemize}

The remaining assertions of Theorem \ref{thm_top_dim_n_anal_dim_1} have to 
do with the analytical structure of $X_\infty$, and are largely based on the existence of 
good families of curves, whose construction we now describe.

The starting point is the 
observation that the first coordinate
$x:\R^2\ra \R$ descends to a $1$-Lipschitz
function $x_\infty:X_\infty\ra\R$, and hence any horizontal geodesic
segment in $\R^2$ projects to a geodesic in $X_\infty$.  Every
$2$-cell $\hat \si$ of $Y_j$ is a rectangle of width $m^{-j}$ and
height $m_v^{-j}$ so it yields a family 
$\Ga_{\hat\si}$ of geodesic segments 
in $X_\infty$ of length
$m^{-j}$; since $\Ga_{\hat\si}$ has a natural parametrization by 
an interval of length $m_v^{-j}$, it carries a natural measure $\nu_{\hat\si}$.

If 
$\hat\si$, $\hat\si'$ are $2$-cells
of $Y_j$ such that the projections
$\hat\pi^j(\si)$, $\hat\pi^j(\si')$ 
share a vertical $1$-cell of $X_j$, then the 
corresponding measured families of curves
$(\Ga_{\hat\si},\nu_{\hat\si})$, $(\Ga_{\hat\si'},\nu_{\hat\si'})$ may
be concatenated to form a new measured family of curves.
More generally, if 
$\si_1,\ldots,\si_\ell$ is a sequence of $2$-cells of 
$X_j$ that form a horizontal gallery\footnote{We have borrowed
the term ``gallery'' from the theory of Coxeter complexes and Tits 
buildings.} in $X_j$ 
(i.e. $\si_{i-1}$ shares a vertical $1$-cell
with $\si_i$ for all $1<i\leq\ell$) then one may concatenate
the corresponding curve families.
To produce an abundance of such curve families,
a crucial property  is  the  ``horizontal
gallery accessibility'' (Lemma \ref{lem:gall_joins}), 
which is an analog of the gallery diameter bound
of \cite[Sec. 11, condition (3)]{piex}.  This says that
if $\si$, $\si'$ are $2$-cells of $X_j$
with controlled combinatorial distance in $X_j$, then they
may be joined by a horizontal gallery 
of controlled length.  This accessibility property
is due to the choice of the equivalence relation
$\calr$: notice that if $\hat\si$, $\hat\si'$ are $2$-cells of $Y_1$, then
one may form a horizontal gallery between their projections
$\si,\si'\in X_1$ by using the vertical identifications $a_{k,\ell,i}\leftrightarrow
a_{k,\ell,i}'$ that define $\calr$.

The $(1,1)$-Poincar\'e inequality is proved using the standard approach  
via ``pencils'' of curves \cite{semmes}.  The pencils are built by combining
the   measured families of curves  described above, for horizontal galleries of different 
scales, see  Section \ref{sec_poincare_inequality}.

To prove that the analytic dimension of $X_\infty$ is $1$, we
define a  
``horizontal derivative'' 
$D_\infty u$ for any Lipschitz function $u:X_\infty\ra\R$. 
 By using  horizontal galleries again,
we show that at an approximate continuity point $p$ of 
 $D_\infty u$, we have
$$
u(q)-u(p)=(D_\infty u(p))(x_\infty(q)-x_\infty(p))+o(d(q,p))\,,
$$
see Section \ref{sec:anal_dim}.

\begin{remark}
Rather than metrizing $X_\infty$ using $\hat d_\infty$, 
an alternate approach is to define a metric on $X_\infty$ using a distinguished
set of paths.  For instance, following \cite{laakso}
 one could define, for all $p,p'\in X_\infty$, the distance $d(p,p')$ to be the infimal 
$1$-dimensional Hausdorff measure of $(\hat\pi^\infty)^{-1}(\ga)$, 
where
$\ga$ is a path from $p$ to $p'$.  This leads to an essentially equivalent 
analysis, with the details organized somewhat differently.
We found the 
approach using $\hat d_\infty$ more transparent.
\end{remark}

\subsection*{Notational conventions}
\label{subsec:conventions}

 In the following $a\approx b$ will indicate that $a$ and $b$ are
comparable up to a uniformly bounded multiplicative factor $C$, and
sometimes we will also write $a\approx_C b$ to highlight $C$. We will
similarly use expressions like $a\lesssim b$ and $a\gtrsim b$.

\section{The cellular structure of the direct system $\{X_j\}$}
\label{sec_cell_structure}
In this section we examine different aspects of the combinatorial
structure of the $X_j$'s.
We remind the reader that we
will retain the notation from Section 2 through Section 
\ref{sec_n_equals_2_case_concluded}. 
The following lemma lists the main properties of $\calr$ and $\calr_j$ that 
will be used later.
\par 
\begin{lemma}[Properties of $\calr$]
\label{lem_properties_of_calr_j}

\mbox{}
\begin{enumerate}
\item 
\label{item_nontrivial_cosets}
Nontrivial cosets of 
$\Phi^{j-1}_*\calr$ belong to the union of vertical
lines  $\{(x,y)\in\R^2\mid x\in m^{-j}\Z\setminus m^{-(j-1)}\Z\}$, and any two
points in the same coset lie in a vertical edge path of $Y_j$ of combinatorial
length at most $2$.
\item 
\label{item_coset_description}
Cosets of $\calr_j$ are contained in orbits of the action
$m^{-j}\Z\times m_v^{-j}\Z\acts\R^2$ by translations.
\item 
\label{item_pair_identifications_descend} 
For all $j\in\Z$, let $\bar\calr_{j+1}$ be the equivalence relation on $X_j$
generated by the pushforward $(\hat\pi_j\circ\Phi^{j})_*(\calr)$.
Then $\bar\calr_{j+1}$ is an equivalence relation on $X_j$, 
and the cosets of $\bar\calr_{j+1}$ are the fibers of $\pi_j:X_j\ra X_{j+1}$.
\item 
\label{item_entire_cells_identified}
For $i=0,1$, let $\hat\si_i$ be a cell of $Y_j$, and 
$\hat p_i$ be an interior point of $\hat\si_i$.  If 
$\hat p_0\sim_{\calr_j}\hat p_1$, then there is a translation 
$t\in m^{-j}\Z\times m_v^{-j}\Z$ such that $t(\hat \si_0)=\hat \si_1$,
and $\hat q_0\sim_{\calr_j}t(\hat q_0)$ for all $\hat q_0\in\hat\si_0$.
\end{enumerate}
\end{lemma}
\begin{proof}
(\ref{item_nontrivial_cosets}) and
(\ref{item_coset_description}) follow immediately from the definition of
$\calr$ and $\Phi$.

(\ref{item_pair_identifications_descend}).  By (\ref{item_nontrivial_cosets})
the nontrivial cosets of $\Phi^{j-1}\calr$ intersect only trivial cosets of
$\calr_{j-1}$.  This implies that $\bar\calr_{j+1}$ is an equivalence relation
on $X_j$, and that its cosets are the fibers of $\pi_j:X_j\ra X_{j+1}$.

(\ref{item_entire_cells_identified}).  Since $\calr_j$ is generated by
$\Phi^i_*\calr$ for $i<j$, it is generated by pairwise identifications of
cells of $Y_j$ by translations.  This implies 
(\ref{item_entire_cells_identified}).

\end{proof}

\begin{lemma}[Cell structure of $X_j$]
\label{lem_cell_structure}
\mbox{}
\begin{enumerate}
\item 
\label{item_open_cells_cw_structure}
The collection
of  images of the
open cells of $Y_j$ under the projection map $\hat\pi^j:Y_j\ra X_j$ defines
a CW complex structure on $X_j$.
\item For every cell $\hat\si$ of $Y_j$, the restriction of $\hat\pi^j$ to 
$\hat\si$ is a characteristic map for the open cell $\hat\pi^j(\Int(\hat\si))
\subset X_j$.
\item 
\label{item_face_identification}
If $\hat \si_0$, $\hat \si_1$ are cells of $Y_j$ then 
$\hat\pi^j(\Int(\hat\si_0))\cap \hat\pi^j(\Int(\hat\si_1))\neq\emptyset$
 if and only if $\calr_j$ identifies $\hat\si_0$ with $\hat\si_1$ by translation.
In particular,
If $\hat\si$ is a cell of $Y_j$, then $\hat\pi^j\restr_\si:\si\ra X_j$
identifies certain pairs of faces of $\hat\si$ by translation. 
\item 
\label{item_x_j_k_cw_structure}
The cells of $Y_{j+k}$ project under $\hat\pi^j:\R^2\ra X_j$ to define
a cell complex $X_j^{(k)}$, which is the $k$-fold iterated
subdivision of $X_j$.
\end{enumerate}
\end{lemma}
\begin{proof}
(\ref{item_open_cells_cw_structure})-(\ref{item_face_identification}).
We first show that $X_j$ is Hausdorff.  

Pick $p\in X_j$.  
By Lemma \ref{lem_properties_of_calr_j}(\ref{item_coset_description}) and 
the fact that $Y_j$ is invariant under $m^{-j}\Z\times m_v^{-j}\Z$,
there is an
$r>0$ such that any cell $\hat\si$ of $Y_j$ that intersects
$N_r((\hat\pi^j)^{-1}(p))$
must intersect $(\hat\pi^j)^{-1}(p)$. Here $N_r(S)=\{q\in \R^2\mid d_{\R^2}(q,S)\leq r\}$ 
denotes the Euclidean metric $r$-neighborhood.

{\em Claim. $N_r((\hat\pi^j)^{-1}(p))$ is a union of cosets of $\calr_j$.}

Suppose $\hat q'\in \R^2$ and 
$\hat q'\sim_{\calr_j}\hat q$ for some $\hat q\in N_r((\hat\pi^j)^{-1}(p))$.
By the choice of $r$, if $\Int(\hat\si)$ is the open cell of $Y_j$ 
containing $\hat q$, then $\hat\si$ contains some point $\hat p\in (\hat\pi^j)^{-1}(p)$.
By Lemma \ref{lem_properties_of_calr_j}(\ref{item_entire_cells_identified}) there
is a translation $t\in m^{-j}\Z\times m_v^{-j}\Z$ such that  $t(\hat q)=\hat q'$
and $\hat p\sim_{\calr_j}t(\hat p)$.  But then $\hat q'\in B(t(\hat p),r)
\subset N_r((\hat\pi^j)^{-1}(p))$, proving the claim.

For every pair of distinct points $p_0,p_1\in X_j$, the sets
$(\hat\pi^j)^{-1}(p_0)$, $(\hat\pi^j)^{-1}(p_1)$ are disjoint
and lie in orbits of $m^{-j}\Z\times m_v^{-j}\Z$, and hence
have positive distance from each other.  By the claim, if
$r>0$ is sufficiently
small, then $\hat\pi^j(N_r((\hat\pi^j)^{-1}(p_0)))$, and
$\hat\pi^j(N_r((\hat\pi^j)^{-1}(p_1)))$ are disjoint open subsets
of $X_j$.  This 
proves that $X_j$ is Hausdorff.

Let $T^2_j$ denote the quotient space
$\R^2/(m^{-j}\Z\times m_v^{-j}\Z)$, and  $\pi_{T^2_j}:\R^2\ra T^2_j$ be the 
quotient map.  The open cells of $Y_j$ project under $\pi_{T^2_j}$
to open cells of $T^2_j$, inducing the standard cell structure on 
$T^2_j$.  
Note that
by Lemma \ref{lem_properties_of_calr_j}(\ref{item_coset_description}) there
is a well-defined continuous map $X_j\ra T^2_j$.  

Now consider a cell $\hat\si$ of $Y_j$.  The composition 
$Y_j\stackrel{\hat\pi^j}{\ra}X_j\ra T^2_j$ maps the interior of $\hat\si$
homeomorphically onto an open cell of $T^2_j$; it follows that $\hat\pi^j:Y_j\ra X_j$
maps the interior of
$\hat\si$ homeomorphically onto its image, which is therefore an open cell.
Thus the
restriction of $\hat\pi^j$ to $\hat\si$
of $Y_j$ is a characteristic map for the open cell $\hat\pi^j(\Int(\hat\si))$.

If  $\hat\si_0$, $\hat\si_1$ are cells of $Y_j$ such that 
$\hat\pi^j(\Int(\hat\si_0))\cap \hat\pi^j(\Int(\hat\si_1))\neq\emptyset$,
then by 
Lemma \ref{lem_properties_of_calr_j}(\ref{item_entire_cells_identified})
we have 
 $\hat\pi^j(\Int(\hat\si_0))= \hat\pi^j(\Int(\hat\si_1))$.
This shows that the collection of images of open cells of $Y_j$ is
a  decomposition of $X_j$ into disjoint open cells.  

For any cell $\hat\si$ of $Y_j$, the closure of $\hat\pi^j(\Int(\hat\si))$ is
just $\hat\pi^j(\hat\si)$, and is therefore contained in the union of 
the images of the open cells contained in $\hat\si$.  The closure of any open
cell of $X_j$ intersects only finitely many open cells.

 Finally, if $C\subset X_j$, then $C$ is closed if and only if $(\hat\pi^j)^{-1}(C)$ is  
closed, which is equivalent to $(\hat\pi^j)^{-1}(C)\cap \hat\si$ being closed
for every cell $\hat\si$ of $Y_j$, which happens if and only if $C\cap \hat\pi^j(\hat\si)$ is closed.

Thus we have verified 
(\ref{item_open_cells_cw_structure})-(\ref{item_face_identification}).

The proof of (\ref{item_x_j_k_cw_structure}) is similar to the proof of
(\ref{item_open_cells_cw_structure}).

\end{proof}

Henceforth the notation $X_j$ and $X_j^{(k)}$ will refer to the cell complex
structure established in Lemma \ref{lem_cell_structure}.
A $1$-cell of 
$X_j^{(k)}$ is {\bf horizontal} ({\bf vertical}) if it is the image 
of a horizontal (vertical) cell of $Y_{j+k}$ under the projection map
$\hat\pi^j:\R^2\ra X_j$.
\par We now analyze the behavior of the projection maps with respect to
the combinatorial structure.
\begin{lemma}
\label{lem_pi_j_structure}
\begin{enumerate}
\item 
\label{item_distance_2}
If $p,p'\in X_j$ are distinct points 
and $\pi_j(p)=\pi_j(p')$, then there is a vertical edge path 
in the $1$-skeleton of $X_j^{(1)}$ that contains $p$, $p'$
and has combinatorial
length at most $2$.
\item
\label{item_si_si'_intersect} 
If $\si$, $\si'$ are $2$-cells of $X_j$, 
and  $\pi_j(\si)\cap \pi_j(\si')\neq \emptyset$, then
$\si\cap\si'\neq\emptyset$.
\item 
\label{item_limits_intersect}
Suppose $0\leq j\leq j'<\infty$ and
$\hat\si$, $\hat\si'$ are  $2$-cells of $Y_j$ and $Y_{j'}$ respectively.  If
$\hat\pi^\infty(\hat\si)\cap \pi^\infty(\hat\si')\neq\emptyset$, then 
$\hat\pi^{j'}(\hat\si)\cap \hat\pi^{j'}(\hat\si')\neq\emptyset$.
\end{enumerate}
\end{lemma}

\begin{proof}
(\ref{item_distance_2}).  This follows from (\ref{item_nontrivial_cosets})
and (\ref{item_pair_identifications_descend}) of 
Lemma \ref{lem_properties_of_calr_j}.

(\ref{item_si_si'_intersect}). Suppose that $p\in \si$, $p'\in \si'$, and
$\pi_j(p)=\pi_j(p')$.  By (\ref{item_distance_2}) there is a vertical 
edge path in $X_j^{(1)}$ of combinatorial length at most two joining $p$ to
$p'$.  Since the subdivision of $\si$ in $X_j^{(1)}$ has
combinatorial ``height'' $m_v$, it follows that $\si\cap \si'\neq\emptyset$.

(\ref{item_limits_intersect}). 
By the definition of the direct limit, we have
 $\hat\pi^{\ell}(\hat\si)\cap\hat\pi^{\ell}(\hat\si')\neq\emptyset$
for some $\ell$.  Letting $k_0+1$ be the minimal such $\ell$, suppose we have 
$k_0\geq j'$.
Since $\si$, $\si'$ are finite unions of $2$-cells of $Y_{k_0}$, we
may assume without loss of generality that $j=j'=k_0$.  
Then  $\hat\pi^{k_0}(\hat\si)$,  $\hat\pi^{k_0}(\hat\si')$ are disjoint
$2$-cells of $X_{k_0}$
whose projections to $X_{k_0+1}$ intersect.  This contradicts 
(\ref{item_si_si'_intersect}).  Thus $k_0+1\leq j'$, proving
(\ref{item_limits_intersect}).

\end{proof}
The next lemma shows that the $X_j$'s and the projection maps have
bounded complexity.
\begin{lemma}
\label{lem_bounded_complexity}

\mbox{}
\begin{enumerate}
\item 
\label{item_injective_1_skeleton}
$\pi_j:X_j\ra X_{j+1}$ is injective on the $1$-skeleton of $X_j$.
\item 
\label{item_0_skeleton_no_collapsing}
If $p\in X_j$ is a $0$-cell, then $\pi_j^{-1}(\pi_j(p))=\{p\}$, i.e. $0$-cells
experience no collapsing. 
\item 
\label{item_0_cell_inverse}
For every $p\in X_j$, the point inverse $(\hat\pi^j)^{-1}(p)$ contains at most
$3$ points.  
\item 
\label{item_bounded_link_complexity}
For every $j$, the link of any cell in $X_j$ contains at most $24$ cells.
\item 
\label{item_links_connected}
For every $j$ and every $0$-cell $v\in X_j$, the link of $v$ is connected:
any two cells $\si$, $\si'$ of $X_j$ containing $v$ may joined by a
sequence $\si=\tau_1,\ldots,\tau_\ell=\si'$ of cells of $X_j$ containing
$v$, where
$\tau_{i-1}$ and $\tau_i$ share a $1$-cell.
\item 
\label{item_bounded_cell_complexity}
Every  cell of $X_j$ contains at most $9$ open cells.
\item 
\label{item_ball_complexity}
There is an $N=N(n)$ such that for every $j\in \Z$, every combinatorial 
$n$-ball in $X_j$ contains at most $N$ cells.
\item
\label{item_cell_inverse_bounded}
For every $j\in \Z$ and every $2$-cell $\si$ of 
$X_j$, the inverse image $(\hat\pi^j)^{-1}(\si)$ may be covered by at most 
$27$ $2$-cells of $Y_j$.
\end{enumerate}
\end{lemma}
\begin{proof}
(\ref{item_injective_1_skeleton}) and (\ref{item_0_skeleton_no_collapsing}) 
follow from Lemma \ref{lem_properties_of_calr_j}(\ref{item_pair_identifications_descend}) and 
 Lemma \ref{lem_pi_j_structure}(\ref{item_distance_2}).

(\ref{item_0_cell_inverse}).  By Lemma \ref{lem_properties_of_calr_j}, a nontrivial
coset of $\calr_j$ is a nontrivial coset of $\Phi^i\calr$ for some $i<j$, and 
therefore contains at most $3$ elements.

(\ref{item_bounded_link_complexity}).
It suffices to verify this
 for $0$-cells.  If $v$ is a $0$-cell of $X_j$, then 
$(\hat\pi^j)^{-1}(v)$  contains at most $3$ vertices of $Y_j$ by
(\ref{item_0_cell_inverse}).   Therefore
$v$ is contained in at most $12$ $2$-cells and  $12$ $1$-cells.

(\ref{item_links_connected}).  If  $(\hat\pi^j)^{-1}(v)$ contains only one point 
$\hat v\in\R^2$,
then the link of $v$ is isomorphic to the link of $\hat v$.  Otherwise for some
$i<j$ the set $(\hat\pi^j)^{-1}(v)$ is a nontrivial coset of $\Phi^i\calr$, and
has the form $\{\hat v, \hat v\pm m_v^{-(i+1)}\}$ for some $0$-cell
$\hat v$ of $Y_{i+1}$, and the two vertical $1$-cells emanating from
$\hat v$ are identified by $\Phi^i\calr$ with the vertical $1$-cells
emanating from $\hat v\pm m_v^{-(i+1)}$.
Thus the link of $v$ is homeomorphic to the quotient of three disjoint circles
$S^1, S^1_+, S^1_-$ by 
by identifying $p_\pm\in S^1$ with $q_\pm\in S^1_\pm$; this is connected.

(\ref{item_bounded_cell_complexity}).
This follows from Lemma \ref{lem_cell_structure}(\ref{item_face_identification}).
 
(\ref{item_ball_complexity}).  This follows by induction on $n$, using
(\ref{item_bounded_link_complexity}) and (\ref{item_bounded_cell_complexity}).

(\ref{item_cell_inverse_bounded}).  
Suppose $\tau\subset\si$ is an open cell of dimension $d$.   By 
Lemma \ref{lem_cell_structure}(\ref{item_face_identification}) the inverse image
of $\tau$ under the projection $\hat\pi^j:Y_j\ra X_j$ is the disjoint union  of the
interiors of a collection $\hat\tau_1,\ldots,\hat\tau_\ell$ of $d$-cells of $Y_j$.  
By (\ref{item_0_cell_inverse})  we have $\ell\leq 3$.  Combining this with
(\ref{item_bounded_cell_complexity}) we get that $(\hat\pi^j)^{-1}(\si)$
can be covered by at most  $27$ cells of $Y_j$.

\end{proof}

Let $x,y$ denote the coordinate functions on $\R^2$.   For every $j\in\Z$,
define $\hat y_j:\R^2\ra S^1(m_v^{-j})=\R/m_v^{-j}\Z$ to be the composition of 
$y:\R^2\ra\R$
with the quotient map $\R\lra \R/m_v^{-j}\Z$.  We will metrize
$S^1(r)=\R/r\Z$ with the quotient metric $d_{S^1(r)}$.
\begin{lemma}[The functions $x_j$ and $y_j$]
\label{lem_x_and_y_descend}

\mbox{}
\begin{enumerate}
\item 
For all $j\in \Z\cup \{\infty\}$ the function $x$ descends to a function $x_j:X_j\ra \R$ and 
for all $j\in \Z$ the function $\hat y_j$ descends to a function
$y_j:X_j\lra S^1(m_v^{-j})=\R/m_v^{-j}\Z$.
\item 
\label{item_x_on_2_cell}
If $\si$ is a $2$-cell of $X_j$, then the image of $\si$ under $x_j:X_j\ra \R$
is an interval of length $m^{-j}$ whose endpoints are the images of the vertical
$1$-cells of $\si$.
\item
\label{item_y_j_on_fiber} If $p,p'\in X_j$ and $\pi_j(p)=\pi_j(p')$, then 
$d(y_j(p),y_j(p'))\leq 2\, m_v^{-(j+1)}$.
\item
\label{item_y_j_diameter} 
If $\bar\si$ is a $2$-cell of $X_{j+1}$, then the inverse image 
$\pi_j^{-1}(\bar \si)\subset X_j$
maps under $y_j:X_j\ra \R/m_v^{-j}\Z$ to a set of diameter at most 
$5m_v^{-(j+1)}$.
\end{enumerate}
\end{lemma}
\begin{proof}
(1).
Since the identifications generating $\calr$ are given by vertical translations by 
multiples of $m_v^{-1}$, the cosets
of $\calr$ are contained in the orbits of the translation 
action $\{e\}\times m_v^{-1}\Z\acts \R^2$; likewise the cosets of $\calr_j$ are contained
in the orbits of $\{e\}\times m_v^{-(j)}\Z$.
Therefore $x$ and   
 $\hat y_j$ descend to $X_j=X_0/\calr_{j}$. 
 
(\ref{item_x_on_2_cell}).  We have $\si=\hat\pi^j(\hat\si)$ for some $2$-cell
$\hat\si$ of $Y_j$.  Then $y_j(\si)=y(\si)$, and the assertion is clear.

(\ref{item_y_j_on_fiber}).  If $\bar e$ is a vertical
$1$-cell of $X_j^{(1)}$, then 
$\bar e=\hat\pi^j(e)$ for some vertical $1$-cell $e$ of $Y_{j+1}$; therefore
$y_j(\bar e)$ has diameter $m_v^{-(j+1)}$.  Applying
Lemma \ref{lem_pi_j_structure}(\ref{item_distance_2}) we 
get $d(y_j(p),y_j(p'))\leq 2\,m_v^{-(j+1)}$.

(\ref{item_y_j_diameter}).  Pick $p,p'\in (\pi_j)^{-1}(\bar\si)$.
Choose $\si$ a $2$-cell of $X_j^{(1)}$ such that 
$\pi_j(\si)=\bar\si$.   Then  there exist $q,q'\in \si$
such that $\pi_j(p)=\pi_j(q)$ and $\pi_j(p')=\pi_j(q')$.
Since $\diam(y_j(\si))=m_v^{-(j+1)}$, by 
(\ref{item_y_j_on_fiber}) we get that
 $\diam(\pi_j)^{-1}(\bar\si)\leq 5\,m_v^{-(j+1)}$.

\end{proof}

\section{Metric structure}
\label{sec_metric_structure}

We now analyze the metric $\hat d_\infty$ by relating it to the geometry
of the approximating cell complexes $X_j$.  The main result in this
section is Proposition \ref{prop_distance_lower_bound}.

Let $\hat d_j$ be the largest pseudodistance on $X_j$ such that 
for every $i\leq j$, and every $2$-cell $\hat\si$ of $Y_i$, the diameter
of $\hat\pi^j(\hat\si)$ is $\leq m^{-i}$.

Before proceeding, we introduce some additional terminology.

A {\bf chain} in a set is a sequence $S_1,\ldots,S_\ell$ of subsets such that
$S_{i-1}\cap S_i\neq\emptyset$ for all $1< i\leq \ell$.  

\begin{definition}
\label{def_j_chain}
Let $\hat\si_1,\ldots,\hat\si_\ell$ be a sequence, 
where $\hat\si_i$ is a $2$-cell of  $Y_{j_i}$.   
Then:
\begin{itemize}
\item  For $j\in \Z\cup \{\infty\}$, the sequence
$\{\hat\si_i\}$ is a {\bf $j$-chain joining $p,p'\in X_j$} if the projections 
$\hat\pi_0^j(\hat\si_0),\ldots,\hat\pi^j(\hat\si_\ell)$
form a chain in $X_j$, and $p\in \hat\pi^j(\hat\si_0)$, $p'\in \hat\pi^j(\hat\si_\ell)$. 
By convention, the empty sequence is a chain joining every point to itself.
\item The {\bf length} and 
{\bf generation} of the sequence $\{\hat\si_i\}$ are
$\sum_im^{-j_i}$ and $\max_ij_i$ respectively.
\end{itemize} 
\end{definition}

By Definition \ref{def_j_chain}, for $p,p'\in X_\infty$, we have
$$
\hat d_\infty(p,p')=\inf\{\length(\{\hat\si_i\})\mid 
\{\hat\si_i\}\;\text{is an $\infty$-chain
joining $p,p'$}\}\,,
$$
and for all $p,p'\in X_j$
  \begin{equation*}
    \hat d_j(p,p')=
       \inf \left\{\length(\{\hat\si_i\})\;\;\middle|\;\;
    \begin{aligned}
    \{\hat\si_i\}\;\text{is }&\text{a $j$-chain of generation}\\
    &\text{$\leq j$ joining $p,p'$}
    \end{aligned}
    \right\},
  \end{equation*}

Note  that it follows that $\hat d_j$ restricts to a discrete metric on each $2$-cell
$\si$ of $X_j$.

\begin{lemma}
If $\hat\si_1,\ldots,\hat\si_\ell$ is an $\infty$-chain of generation $g$,  
then it is a $j$-chain for all $j\geq g$. 
\end{lemma}
\begin{proof}
This follows from Lemma \ref{lem_pi_j_structure}(\ref{item_limits_intersect}).
\end{proof} 

\begin{lemma}
\label{lem_chain_connected}
Let $\C=\C_1\cup\C_2$ be a collection of subsets of a set, and $S,S'\in \C_1$.  
Assume that $\C$ contains a chain from $S$ to $S'$, and that every element of $\C_2$
is contained in some element of $\C_1$.  Then $\C_1$ contains a chain from $S$ to $S'$.
\end{lemma}
\begin{proof}
Starting with a chain in $\C$ from $S$ to $S'$, one may inductively reduce the number
of elements from $\C_2$ by replacing each one with an element of $\C_1$ that contains it.
\end{proof}

\begin{definition}
\label{def_gallery}
A {\bf gallery} in $X_j^{(k)}$ is a sequence of $2$-cells 
$\si_1,\ldots,\si_\ell$ of $X_j^{(k)}$
such that $\si_{i-1}\cap\si_i$ contains a $1$-cell
$e_i$ of $X_j^{(k)}$
for every $1<i\leq \ell$; the gallery is {\bf vertical}  
if for all $1<i\leq\ell$ the $1$-cell $e_i$ is horizontal, and {\bf horizontal} 
if $e_i$ is
vertical for all $1<i\leq \ell$. 
\end{definition}

\begin{lemma}
\label{lem_metric_versus_combinatorial_distance}
There exist $N_1=N_1(C)$, $N_2=N_2(C)$
such that for every $j\in\Z$,
every $\hat d_j$-ball $B(p,Cm^{-j})\subset X_j$ is contained in a combinatorial
$N_1$-ball of $X_j$, and is contained in a union of at most $N_2$ $2$-cells
of $X_j$.  
\end{lemma}

\begin{proof}
Fix $C<\infty$ and $p\in X_j$.  

If $\hat d_j(p',p)<Cm^{-j}$,  then there is a $j$-chain 
$\hat\si_1^0,\ldots,\hat\si_{\ell_0}^0$ from $p$ to $p'$ of length
$<Cm^{-j}$.  
Letting $j_i$ be the generation of $\hat\si_i^0$, this yields
$$
m^{-j_i}<Cm^{-j}\,,\quad \ell_0\cdot m^{-j}< Cm^{-j}
$$
so

\begin{equation}
\label{eqn_generation_bound}
j-j_i<\frac{\log C}{\log m}
\end{equation}
and
$\ell_0<C$.
Now (\ref{eqn_generation_bound}) implies
that  $\hat\si_i^0$ contains at most $C_1=C_1(C)$ $2$-cells
of $Y_j$.  The collection $\C$ of all $2$-cells of $Y_j$
contained in $\cup_i\hat\si_i^0$ contains at most $C_1\cdot \ell_0$
$2$-cells.  Also $\C$  contains a $j$-chain $\hat\si_1,\ldots,\hat\si_\ell$
from $p$ to $p'$, where $\ell\leq C_1\ell_0<C_1\cdot C$.  Putting 
$N_1=C_1\cdot C$, we have shown that $p'$ lies in the combinatorial
$N_1$-ball of the open cell containing $p$.  

By Lemma \ref{lem_bounded_complexity}(\ref{item_ball_complexity}) the 
$N_1$-ball contains at most $N_2=N_2(C)$ cells.

\end{proof}

The next two lemmas provide lower bounds on the combinatorial length
of certain types of chains of $2$-cells in $X_{j+1}$.

\begin{lemma}
\label{lem_joining_edges}
Let $\si\subset X_j$ be a $2$-cell with vertical $1$-cells $e_1,e_2$.  
If $\bar \si_1,\ldots,\bar \si_\ell
\subset\pi_j(\si)$ is a chain of $2$-cells of $X_{j+1}$ that joins 
$\pi_j(e_1)$ to $\pi_j(e_2)$, 
then $\ell\geq m$.
\end{lemma}
\begin{proof}
Since $\{\bar \si_i\}$ is a chain in $X_{j+1}$, the images under $x_{j+1}:X_{j+1}\ra [0,1]$
form a chain in $[0,1]$ joining $x_{j+1}(\pi_j(e_1))$ to $x_{j+1}(\pi_j(e_2))$.  But 
by Lemma \ref{lem_x_and_y_descend}(2) the image $x_{j+1}(\bar \si_i)$
is a segment of length $m^{-(j+1)}$, while 
$|x_{j+1}(\pi_j(e_1))-x_{j+1}(\pi_j(e_2))|=m^{-j}$.
Thus 
$$
m^{-j}=\diam(\pi_j(e_1)\cup\pi_j(e_2))
\leq \sum_i\diam(x_{j+1}(\bar \si_i))=\ell\cdot m^{-(j+1)}\,
$$
so $\ell\geq m$.
\end{proof}
\begin{definition}
\label{def_stars}
The {\bf open star (resp. closed star)} of a cell $\si$ in a cell complex is the union of the open (respectively closed)
cells whose closure contains $\si$.
\end{definition}

\begin{lemma}
\label{lem_edge_star}
Let $e$ be a vertical $1$-cell of $X_j$, and $\Star(e,X_j)$ be the closed star of $e$
in $X_j$.    Let $ \bar\si_1,\ldots, \bar\si_\ell$ be a chain of $2$-cells in $X_{j+1}$
such that:
\begin{enumerate}
\item $ \bar\si_1,\ldots,\bar\si_\ell\subset \pi_j(\Star(e,X_j))$.
\item $ \bar\si_1\cap \pi_j(e)\neq\emptyset$.
\item There is a $2$-cell $\tau$ of $X_j$  that
is not contained in $\Star(e,X_j)$, such that
$ \bar\si_\ell$ intersects $\pi_j(\tau)$.
\end{enumerate}
Then one of the following holds:
\begin{enumerate}
\renewcommand{\labelenumi}{(\alph{enumi})}
\item
\label{item_vertical_movement} 
For some $ \bar\si_i$, there is a $p\in \pi_j^{-1}(\bar\si_i)$
such that
$d(y_j(p),0)\leq 2m_v^{-(j+1)}$.  \item There is a $2$-cell $\si$ of $X_j$ that is contained 
in $\Star(e,X_j)$, with vertical $1$-cells $e$, $e_1$
such that some subchain of $\{ \bar\si_i\}$ 
is contained in $\pi_j(\si)$ and joins $\pi_j(e)$ to $\pi_j(e_1)$.
\end{enumerate}
Moreover in case (a) we have
\begin{equation}
\label{eqn_lower_bound_on_ell}
\ell\geq \frac{1}{5}\cdot \left[
m_v^{(j+1)}\dist(y_j(\pi_j^{-1}(\bar\si_1)),0)-2\right]\,.
\end{equation}
\end{lemma}
\begin{proof}
Suppose (a) holds.   Since the sets
$\{y_j(\pi_j^{-1}(\bar\si_k))\}_{1\leq k\leq i}$ each have
diameter $\leq 5m_v^{-(j+1)}$ by 
Lemma \ref{lem_x_and_y_descend}(\ref{item_y_j_diameter}), and they 
form a chain in $S^1(m_v^{-j})$
joining $y_j(\pi_j^{-1}(\bar\si_1))$ to
$B(0,2m_v^{-(j+1)})\subset S^1(m_v^{-j})$,
the triangle inequality gives (\ref{eqn_lower_bound_on_ell}).  

Now suppose (a) does not hold, i.e. for every $\bar\si_i$ and every $p\in \pi_j^{-1}(\bar\si_i)$,
\begin{equation}
\label{eqn_not_a}
d(y_j(p),0)> 2m_v^{-(j+1)}\,.
\end{equation}
  We may assume without loss of generality that $\bar\si_1,\ldots,\bar\si_\ell$ is a 
minimal chain satisfying the hypotheses of the Lemma, so that: 

\begin{enumerate}
\renewcommand{\labelenumi}{(\roman{enumi})}
\item For all $i>1$, the $2$-cell $\bar\si_i$  does not intersect
$\pi_j(e)$. 
\item For all $i<\ell$, the $2$-cell $\bar\si_i$ does not intersect $\pi_j(\tau')$ for
any $2$-cell $\tau'$ of $X_j$ not contained in $\Star(e,X_j)$.
\end{enumerate} 

For every $i$, there is a $2$-cell $\si_i$ of $X_j^{(1)}$ such that
$\pi_j(\si_i)=\bar\si_i$, and $\si_i\subset\Star(e,X_j)$.  The $2$-cell $\si_i$
is contained in a unique $2$-cell $\tau_i$ of $X_j$ that contains $e$.

Let $\tau$ be as in (3), and pick $\bar p\in \bar\si_\ell\cap \pi_j(\tau)$.
Choose $p_\ell\in \si_\ell$, $p\in\tau$ such that $\pi_j(p_\ell)=\pi_j(p)=\bar p$.
By Lemma \ref{lem_pi_j_structure}(\ref{item_distance_2}) there is a vertical edge 
path $\ga$ in $X_j^{(1)}$  of combinatorial length at most $2$ that
contains $p_\ell$ and $p$.  
Because $\tau\not\subset\Star(e,X_j)$, we know that $e\not\subset \tau$, and
so $\tau\neq\tau_\ell$.  Therefore   $\ga$ must intersect $\D\tau_\ell$ and
$\D \tau$.  In view of (\ref{eqn_not_a}) we conclude that $p_\ell$ and $p$
both lie in the interior of a vertical $1$-cell $e_1\subset \tau_\ell\cap \tau$.  
Since $e\not\subset \tau$, it follows that $e_1\neq e$, so $e_1$
is the second vertical $1$-cell of $\D \tau_\ell$.

Let $i_0$ be the minimal $i$ such that the subchain
$\bar\si_i,\ldots,\bar\si_\ell$ is contained in $\tau_\ell$, i.e. such that
$\tau_i=\ldots=\tau_\ell$.   Suppose $i_0>1$.  Choose $p_{i_0-1}\in \si_{i_0-1}$,
$p_{i_0}\in\si_{i_0}$ such that 
$\pi_j(p_{i_0-1})=\pi_j(p_{i_0})\in \bar\si_{i_0-1}\cap\bar\si_{i_0}
\subset\pi_j(\tau_{i_0-1})\cap \pi_j(\tau_\ell)$.  Reasoning as above,
we get that $p_{i_0-1}$ and $p_{i_0}$ belong to the interior 
of a vertical $1$-cell $e'$ of $X_j$, where
$e'\subset\D\tau_{i_0-1}\cap\D \tau_\ell$.  We cannot have $e'=e$ or
$e'=e_1$ by the minimality of the chain $\bar\si_1,\ldots,\bar\si_\ell$.
This is a contradiction.  Hence $i_0=1$ and 
$\bar\si_1,\ldots,\bar\si_\ell\subset\tau_\ell$.
Putting $\si=\tau_\ell$ we have shown that (b) holds.
\end{proof}

\begin{proposition}
\label{prop_distance_lower_bound}
For all $k\in\Z$ and  $p,p'\in X_k$, 
\begin{equation}
\label{eqn_distance_lower_bound}
\hat d_\infty(\pi_k^\infty(p),\pi_k^\infty(p'))\geq 
\hat d_k(p,p')-2m^{-k}\,.
\end{equation}
\end{proposition}
\begin{proof}
Fix $k\in\Z$, $p,p'\in X_k$, and let $\bar p=\pi_k^\infty(p)$, 
$\bar p'=\pi_k^\infty(p')$.  Choose $2$-cells $\hat\si,\hat\si'$ of $Y_k$
such that $p\in\hat\pi^k(\si)$, $p'\in\hat\pi^k(\si')$.

Let $\hat\si_1^0,\ldots,\hat\si_{\ell_0}^0$ be an 
$\infty$-chain joining $\bar p$ to $\bar p'$.  We enlarge this
to a new $\infty$-chain $\hat\si,\hat\si_1^0,\ldots,\hat\si_{\ell_0}^0,\hat\si'$,
which we relabel as $\hat\si_1,\ldots,\hat\si_\ell$, where $\ell=\ell_0+2$.
Then
$\{\hat\si_i\}$ is an $\infty$-chain joining $\bar p$ to $\bar p'$ of length 
$$
\leq \length(\{\hat\si_i^0\})+2m^{-k}\,.
$$
  We will show that
we can reduce the generation of the chain $\{\hat\si_i\}$ to $k$ without increasing
its length, so as to obtain a $k$-chain from $\bar p$ to $\bar p'$ of length
at most $\length(\{\hat\si_i\})$. This implies that 
$\hat d_k(p,p')\leq \length(\{\hat\si_i^0\})+2m^{-k}$.
Taking the infimum over all $\infty$-chains from
$p$ to $p'$ yields (\ref{eqn_distance_lower_bound}).

Let $j+1$ be the generation of $\{\hat\si_i\}$.  If $j+1\leq k$ we are done, so we assume
that $j\geq k$.  Moreover, we may assume that there is no other $\infty$-chain 
$\hat\si_0',\ldots,\hat\si_{\ell'}'$ such that:
\begin{itemize}
\item  
$\hat\si_0'=\hat\si_0$, $\hat\si_{\ell'}'=\hat\si_\ell$.
\item 
$\length(\{\hat\si_i'\}) \leq \length(\{\hat\si_i\})$.
\item 
The  generation of $\{\hat\si_i'\}$ is at most $j+1$.
\item $\{\hat\si_i'\}$ has fewer cells of generation $j+1$ than $\{\hat\si_i\}$.
\end{itemize}
 Since $\{\hat\si_i\}$ is an $\infty$-chain, by 
Lemma \ref{lem_pi_j_structure} the
sequence $\hat\si_1,\ldots,\hat\si_\ell$ is a $(j+1)$-chain.

Before proceeding further, we first indicate the rough idea of the 
argument.  Although it takes values in $S^1(m_v^{-j})$, we view
$y_j:X_j\ra S^1(m_v^{-j})$ as a ``height function''.  Since $m_v$ is 
much larger than $m$, a string of cells in 
$\hat\si_1,\ldots,\hat\si_\ell$ of maximal generation $j+1>k$
cannot move efficiently in the ``vertical'' direction, i.e. it cannot
change $y_j$ efficiently.  Thus any such string is forced to move 
roughly horizontally, and it may then be replaced by cells of lower generation
without increasing  $\length(\{\hat\si_i\})$.  This contradicts the choice
of $\{\hat\si_i\}$.

We now resume the proof of the proposition.

Let $\hat\si_{i_{1}},\ldots,\hat\si_{i_{2}}$ be a maximal string of consecutive
cells from $\{\hat\si_i\}$ of generation $j+1$,  so that both $\hat\si_{i_1-1}$ 
and $\hat\si_{i_2+1}$ have
generation $\leq j$.  
For $1\leq i \leq \ell$, let $\si_i=\hat\pi^j(\hat\si_i)$, 
$\bar\si_i=\hat\pi^{j+1}(\hat\si_i)$.  For $i_1\leq i\leq i_2$
let $\hat\tau_i$ be the unique $2$-cell
of $Y_j$ containing $\hat\si_i$, and put $\tau_i=\hat\pi^j(\hat\tau_i)$.

\bigskip
{\em Step 1.
If $e$ is a vertical
$1$-cell of $X_j$, and $\bar\si_{i_3}\cap \pi_j(e)\neq\emptyset$
for some $i_1\leq i_3\leq i_2$, then 
$\dist(y_j((\pi_j)^{-1}(\bar\si_{i_3})),0)< 11m\cdot m_v^{-(j+1)}$.}

Let $\bar\si_{i_4},\ldots,\bar\si_{i_5}$ 
be a maximal
subsequence of $\bar\si_{i_1},\ldots,\bar\si_{i_2}$ such that:
\begin{itemize}
\item $i_4\leq i_3\leq i_5$, i.e. $\bar\si_{i_3}$ belongs to the subsequence.
\item The cells $\bar\si_{i_4},\ldots,\bar\si_{i_5}$ are all   contained in 
$\pi_j(\Star(e,X_j))$.
\end{itemize}

{\bf Claim.} At least one of the two cells $\bar\si_{i_4},\bar\si_{i_5}$
must intersect $\pi_j(\tau)$, for some $2$-cell $\tau$ of $X_j$
that is not contained in $\Star(e,X_j)$.

\begin{proof}[Proof of claim]
Suppose the claim were false.  

Assume first that $\si_{i_4-1}$ has generation
$j+1$.  By the maximality of $\bar\si_{i_3},\ldots,\bar\si_{i_5}$, 
we have  $\bar\si_{i_4-1}\not\subset \pi_j(\Star(e,X_j))$, so 
$\tau_{i_4-1}\not\subset\Star(e,X_j)$.  Then 
$\emptyset\neq\bar\si_{i_4-1}\cap \bar\si_{i_4}
\subset\pi_j(\tau_{i_4-1})\cap\bar\si_{i_4}$,
proving the claim.   Thus we may assume that 
$\si_{i_4-1}$ has generation
$\leq j$, and likewise for $\si_{i_5+1}$.  

Since $\si_{i_4-1}$ is a union of $2$-cells of $X_j$, for some
$2$-cell $\tau$ of $X_j$, we have $\tau\subset\si_{i_4-1}$ and
$\bar\si_{i_4}\cap \pi_j(\tau)\neq\emptyset$.   
If $\tau\not\subset\Star(e,X_j)$, then the claim follows, so we assume
that $\tau\subset\Star(e,X_j)$, and in particular $e\subset\tau$.
Likewise we may assume that $\si_{i_5+1}$ contains a $2$-cell $\tau'$
of $X_j$ that contains $e$.  Now $\si_{i_4-1}\cap\si_{i_5+1}\neq\emptyset$,
so we may shorten the $j+1$-chain $\{\hat\si_i\}$ by deleting 
$\hat\si_{i_4},\ldots,\hat\si_{i_5}$, which is a contradiction.
Thus the claim holds. 
\end{proof}

We now assume without loss of generality that the claim holds for 
$\bar\si_{i_5}$.   Hence $\bar\si_{i_3},\ldots,\bar\si_{i_5}$
satisfies the hypotheses of Lemma \ref{lem_edge_star}.

First suppose that conclusion (a) of Lemma \ref{lem_edge_star} holds.
Then (\ref{eqn_lower_bound_on_ell}) gives
\begin{equation}
\label{eqn_i3_i5}
i_5-i_3+1\geq \frac{1}{5}\cdot \left[
m_v^{(j+1)}\dist(y_j(\pi_j^{-1}(\bar\si_{i_3})),0)-2\right]\,.
\end{equation}
Because $\bar\si_{i_3},\bar\si_{i_5}\subset\Star(e,X_j)$, we have
$\tau_{i_3},\tau_{i_5}\subset\Star(e,X_j)$, so $e\subset\tau_{i_3}\cap \tau_{i_5}$.
It follows that we obtain a $(j+1)$ chain with fewer cells of generation 
$j+1$ by replacing $\hat\si_{i_3},\ldots,\hat\si_{i_5}$ with the two
cells $\tau_{i_3},\tau_{i_5}$.  This new chain has length 
$$
\length(\{\hat\si_i\}_{i=1}^\ell)
-\length(\hat\si_{i_3},\ldots,\hat\si_{i_5})
+\length(\hat\tau_{i_3},\hat\tau_{i_5})
$$
$$
=\length(\{\hat\si_i\}_{i=1}^\ell)
-m^{-(j+1)}(i_5-i_3+1)-2m^{-j}
$$
$$
=\length(\{\hat\si_i\}_{i=1}^\ell)
-m^{-(j+1)}((i_5-i_3+1)-2m)\,.
$$

By the minimality of $\{\hat\si_i\}_{i=1}^\ell$ we obtain
$$
(i_5-i_3+1)-2m\leq 0
$$
Using (\ref{eqn_i3_i5}) we get 
$$
\dist(y_j(\pi_j^{-1}(\bar\si_{i_3})),0)\leq (10m+2)m_v^{-(j+1)}<11m\cdot m_v^{-(j+1)}\,.
$$

Now suppose Case (b) of Lemma \ref{lem_edge_star} holds, and let
$\si$ be as in (b).  Then the subchain
given by (b) satisfies the hypotheses of Lemma \ref{lem_joining_edges}, and so
it contains at least $m$ cells; we may replace these with $\si$ and obtain
a new $j+1$-chain with length at most $\length(\{\hat\si_i\}$, and fewer
cells of generation $j+1$.  This contradicts the choice of $\{\hat\si_i\}$.

This completes Step 1.

\bigskip
{\em Step 2.  We have 
\begin{equation}
\label{eqn_16_m_y_j_bound}
\dist(y_j(\pi_j^{-1}(\bar\si_{i_3})),0)<16m\cdot m_v^{-(j+1)}
\end{equation}
for all $i_1\leq i_3\leq i_2$.}

Suppose $i_3$ violates (\ref{eqn_16_m_y_j_bound}).  Choose $i_4$ maximal
such that $i_3\leq i_4\leq i_2$ and $\tau_i=\tau_{i_3}$ for all $i_3\leq i\leq i_4$.
Since we obtain a comparison chain by replacing $\hat\si_{i_3},\ldots,\hat\si_{i_4}$
with $\hat\tau_{i_3}$, we get that 
$$
m^{-(j+1)}(i_4-i_3+1)=\length(\hat\si_{i_3},\ldots,\hat\si_{i_4})
<\length(\hat\tau_{i_3})=m^{-j}\,,
$$ 
i.e. $i_5-i_3+1<m$.  Applying Lemma \ref{lem_x_and_y_descend}(\ref{item_y_j_diameter})
and the fact that the sets $y_j(\pi_j^{-1}(\bar\si_i))$ form a chain
in $S^1(m_v^{-j})$, we get
\begin{equation}
\label{eqn_lower_height_bound_bar_si}
\dist(y_j(\pi_j^{-1}(\bar\si_{i_4})),0)\geq 11m\cdot m_v^{-(j+1)}\,.
\end{equation}
First suppose $i_4<i_2$, i.e. that $\gen(\hat\si_{i_4+1})=j+1$.    
Then $\tau_{i_4+1}\neq\tau_{i_4}$ and 
$\pi_j(\tau_{i_4+1})\cap\pi_j(\tau_{i_4})\neq\emptyset$.
By (\ref{eqn_lower_height_bound_bar_si}) it follows that 
$\bar\si_{i_4}$ must intersect  $\pi_j(e)$ for some
vertical $1$-cell $e$ of $\tau_{i_4}$.
But then (\ref{eqn_lower_height_bound_bar_si}) contradicts Step 1.

Now suppose $i_4=i_2$.  Then $\gen(\hat\si_{i_4+1})\leq j$, and
so $\si_{i_4+1}$ contains a $2$-cell $\tau$ of $X_j$ such that
$\bar\si_{i_4}\cap \pi_j(\tau)\neq\emptyset$.  If $\tau=\tau_{i_3}=\tau_{i_4}$,
then we may discard $\hat\si_{i_3},\ldots,\hat\si_{i_4}$, contradicting
the minimality of the chain $\hat\si_1,\ldots,\hat\si_\ell$.
Therefore $\tau\neq\tau_{i_3}$, and reasoning as above we find that
$\bar\si_{i_4}\cap \pi_j(e)\neq\emptyset$ for some vertical $1$-cell of
$\tau_{i_3}$.  Using (\ref{eqn_lower_height_bound_bar_si}), we again get a 
contradiction to Step 1.

\bigskip
{\em Step 3. Shifting $\bar\si_{i_1},\ldots,\bar\si_{i_2}$ toward the
$1$-skeleton of $X_j$.}

Pick $i_1\leq i\leq i_2$.
Since $\si_i\subset\pi_j^{-1}(\bar\si_i)$, by Step 2 and 
Lemma \ref{lem_x_and_y_descend}(\ref{item_y_j_diameter}) we have
$$
\dist(y_j(\si_i),0)\leq \dist(y_j(\pi_j^{-1}(\bar\si_i)),0)
+\diam(y_j(\pi_j^{-1}(\bar\si_i)))
$$
$$
<(16m+5)\cdot m_v^{-(j+1)}\,.
$$
Hence there is vertical gallery
(see Definition \ref{def_gallery})
$\mu_{i1},\ldots,\mu_{i\ell_i}$ in $X_j^{(1)}$ with
$\ell_i<16m+5$ that starts from $\si_i=\mu_{i1}$, such that $\mu_{i\ell_i}$ contains
a horizontal $1$-cell $e_i'$ of $X_j^{(1)}$ that is contained in 
a horizontal $1$-cell $e_i$ of $X_j$, see Figure ???.
Since the gallery is vertical, we have 
$x_j(\si_{i-1})=x_j(e_{i-1}'),\,x_j(\si_i)=x_j(e_i')\subset\R$.

{\bf Claim.}  $e_{i-1}'\cap e_i'\neq\emptyset$ for all $i_1<i\leq i_2$.  

\begin{proof}[Proof of claim]
The images $x_j(\si_{i-1}),\,x_j(\si_i)\subset\R$
are intervals of length $m^{-(j+1)}$ that have an endpoint in common
because $x_j(\si_{i-1})=x_{j+1}(\bar\si_{i-1})$, $x_j(\si_i)=x_{j+1}(\bar\si_i)$, and
$\bar\si_{i-1}\cap\bar\si_i\neq\emptyset$.  We may therefore choose
$0$-cells $v_{i-1},v_i$ of $X_j^{(1)}$ such that $v_{i-1}\in e_{i-1}'$,
$v_i\in e_i'$, and $x_j(v_{i-1})=x_j(v_i)$.  Choose $0$-cells $p_{i-1},p_i$
of $X_j^{(1)}$ such that $p_{i-1}\in \si_{i-1}$, $p_i\in \si_i$, 
$\pi_j(p_{i-1})=\pi_j(p_i)$, $x_j(p_{i-1})=x_j(p_i)=x_j(v_{i-1})=x_j(v_i)$.
We may join $p_{i-1}$ to $p_i$ with a vertical edge path in $X_j^{(1)}$
of combinatorial length at most $2$ by 
Lemma \ref{lem_pi_j_structure}(\ref{item_distance_2}).  Using the vertical
cells in the vertical galleries $\{\mu_{i-1n}\}$, $\{\mu_{in}\}$, we get
vertical edge paths $\ga_{i-1}$, $\ga_i$ in $X_j^{(1)}$ joining
$v_{i-1}$ to $p_{i-1}$ and $v_i$ to $p_i$, respectively, where $\ga_{i-1},\,\ga_i$
have combinatorial length $<16m+5$.  Concatenating $\ga_{i-1},\,\ga,\,\ga_i$
we get a vertical edge path joining $v_{i-1}$ to $v_i$ of combinatorial length
$<32m+12=140$.  Since $m_v=3L\geq 300$, this forces $v_{i-1}=v_i$, proving the
claim.
\end{proof}

\bigskip
{\bf Claim.} $\si_{i_1-1}\cap e_{i_1}'\neq\emptyset$ and 
$e_{i_2}'\cap\si_{i_2+1}\neq\emptyset$.

\begin{proof}[Proof of claim]
There is a $2$-cell $\tau\subset\si_{i_1-1}$ of $X_j$
such that $\pi_j(\tau)\cap\bar\si_{i_1}\neq\emptyset$.
We may find $0$-cells $p,p_{i_1}$ of $X_j^{(1)}$ such that
$p\in \tau$, $p_{i_1}\in \si_{i_1}$, and $\pi_j(p)=\pi_j(p_{i_1})$. 
There is a vertical edge path $\ga$ in $X_j^{(1)}$ from $p$ to $p_{i_1}$ of 
combinatorial length at most $2$, and a vertical edge path $\ga_{i_1}$
in $X_j^{(1)}$ of combinatorial length $<16m$ joining $p_{i_1}$
to an endpoint $v_{i_1}$ of $e_{i_1}'$.  Combining $\ga$ and $\ga_{i_1}$
we get a vertical edge path of length $<16m+2$ in $X_j^{(1)}$
starting at $p\in\tau$ and ending in the $1$-skeleton of $X_j$.  This
implies that $v_{i_1}\in \D\tau$.  Hence $\si_{i_1-1}\cap e_{i_1}'\neq\emptyset$
and the claim holds.
\end{proof}

Combining the two claims, we get that
\begin{equation}
\label{eqn_edge_chain}
\si_1,\ldots,\si_{i_1-1},e_{i_1}',\ldots,e_{i_2}',\si_{i_2+1},\ldots,\si_\ell
\end{equation}
forms a chain in $X_j$.  Therefore, after modifying $\hat\si_i$ for $i_1\leq i\leq i_2$
if necessary, we may assume without loss of generality that $\si_i=\mu_{i\ell_i}$,
and in particular $\si_i$ contains $e_i'$.  

\bigskip
{\em Step 4. The final contradiction.}

Pick $i_3$ such that $i_1\leq i_3\leq i_2$.  

Suppose some interior point of $e_{i_3}$ is contained in $\si_{i_1-1}\cup\si_{i_2+1}$.
Since $\si_{i_1-1}\cup \si_{i_2+1}$ is a subcomplex of $X_j$, this implies
that $e_{i_3}$ is contained in $\si_{i_1-1}\cup\si_{i_2+1}$.  
Therefore applying Lemma \ref{lem_chain_connected} we may delete
the edge $e_{i_3}'$ from the collection
(\ref{eqn_edge_chain}), and it will still contain a chain, and
likewise we may delete $\hat\si_{i_3}$ from the collection
$\hat\si_1,\ldots,\hat\si_\ell$ and it will
still contain a chain;
this contradicts the minimality of $\{\hat\si_i\}$.  

Now suppose $e_{i_3}$ is contained in $e_{i_1}'\cup\ldots \cup e_{i_2}'$.  
Taking
the union of the collection (\ref{eqn_edge_chain}) with $e_{i_3}$ and applying
Lemma \ref{lem_chain_connected}, it follows that we may remove from the
collection 
(\ref{eqn_edge_chain}) each $e_i'$ that is 
contained in $e_{i_3}$, and add $e_{i_3}$, and 
the resulting collection will contain a chain from $\si_1$ to $\si_\ell$.
Therefore we may 
remove from $\hat\si_1,\ldots,\hat\si_\ell$ each $\hat\si_i$ with $i_1\leq i\leq i_2$
such that $e_i'\subset e_{i_3}$, and add the $2$-cell $\hat\tau_{i_3}$, and the
resulting collection will contain  
a  $j+1$-chain that has length at most 
$\length(\{\hat\si_i\})$ and fewer cells of generation $j+1$, contradicting
the definition of $\hat\si_1,\ldots,\hat\si_\ell$.

Therefore $e_{i_3}$ is not contained in 
$\si_{i_1-1}\cup \si_{i_1}\cup\ldots\cup \si_{i_2}\cup\si_{i_2+1}$.  But then we may remove
from the $j+1$-chain $\hat\si_1,\ldots,\hat\si_\ell$
some $2$-cell $\hat\si_i$, 
with $i_1\leq i\leq i_2$ and $e_i'\subset e_{i_3}$  and still 
have a chain.  This contradicts the minimality of $\hat\si_1,\ldots,\hat\si_\ell$.

We conclude that $j+1\leq k$, completing the proof of 
Proposition \ref{prop_distance_lower_bound}.

\end{proof}

The following two corollaries of Proposition \ref{prop_distance_lower_bound}
relate the distance in $X_\infty$ or $X_j$ with the combinatorial distance.

\begin{corollary}
  \label{cor_dinfty_comb_dist}
For every $C$ there is an $N=N(C)$ such that if 
$p,p'\in X_j$ and $\hat d_\infty(\pi_j^\infty(p),\pi_j^\infty(p'))<Cm^{-j}$, then
there is a chain of at most $N$ cells of $X_j$ joining $p$ to $p'$.
\end{corollary}
\begin{proof}
This follows from  Proposition \ref{prop_distance_lower_bound}
and Lemma \ref{lem_metric_versus_combinatorial_distance}.
\end{proof}
\begin{corollary}
  \label{cor:cell_dist}
  For  $p_0,p_1\in X_\infty$ let
  \begin{equation}
    \label{eq:cell_dist_s1}
    J(p_0,p_1)=
    \left\{j\;\;\middle|\;\;
    \begin{aligned}
    (\pi_j^\infty)^{-1}(p_0), (\pi_j^\infty)^{-1}(p_1)\;\text{do not}\\
    \text{intersect adjacent cells of}\;X_j
    \end{aligned}
    \right\},
  \end{equation}
  and let $\ul J(p_0,p_1)=\inf J(p_0,p_1)$, where as usual the infimum of the
  empty set  is $\infty$.  Then
  \begin{equation}
    \label{eq:cell_dist_s2}
    m^{-\ul J(p_0,p_1)}\le  \hat d_\infty(p_0,p_1)\le 2m\cdot m^{-\ul J(p_0,p_1)},
  \end{equation}
  where by convention we let $m^{-\ul J(p_0,p_1)}=0$ when $\ul J(p_0,p_1)=\infty$.
\end{corollary}
\begin{proof}
Note that $\ul J(p_0,p_1)=\inf J(p_0,p_1)=1+\sup (\Z\setminus J(p_0,p_1))$.

If $k\not\in J(p_0,p_1)$, then 
$(\pi_k^\infty)^{-1}(p_0)$ and $(\pi_k^\infty)^{-1}(p_1)$ intersect
  adjacent cells, and this gives $\hat d_\infty(p_0,p_1)\leq 2m^{-k}$ by 
  the definition of $\hat d_\infty$.
Thus
\begin{equation*}
\begin{aligned}
\hat d_\infty(p_0,p_1) &\leq 2\,\inf\{m^{-k}\mid k\not\in J(p_0,p_1)\}\\
&\leq 2m\cdot \inf \{m^{-(k+1)}\mid k\not\in J(p_0,p_1)\}\\
&=2m\cdot \sup\{m^{-k}\mid k\in J(p_0,p_1)\}\\
&=2m\cdot m^{\ul J(p_0,p_1)}\,.
\end{aligned}
\end{equation*}

If $k\in J(p_0,p_1)$ then 
 then any $k$-chain  connecting  
$(\pi_k^\infty)^{-1}(p_0)$ with $(\pi_k^\infty)^{-1}(p_1)$
must contain at least $3$ cells, and hence
  \begin{equation}
    \label{eq:cell_dist_p2}
    \hat d_k 
    ((\pi_k^\infty)^{-1}(p_0),(\pi_k^\infty)^{-1}(p_1))\ge 3m^{-k}.
  \end{equation}
The lower bound in (\ref{eq:cell_dist_s2}) then follows applying Proposition \ref{prop_distance_lower_bound}.
\end{proof}

We remark that the following lemma is not really essential to the discussion.
Even
without knowing that $\hat d_\infty$ is a distance, we could quotient out the
sets of zero diameter and work in the resulting metric space, cf. 
Section \ref{sec_generalizations}.

\begin{lemma}
  \label{lem:cell_sep}
Keeping the notation from Corollary \ref{cor:cell_dist},
the set $J(p_0,p_1)$ is nonempty iff $p_0\neq p_1$; in particular 
$\hat d_\infty$ is a distance function on $X_\infty$.
\end{lemma}
\begin{proof}
Clearly $J(p_0,p_1)\neq\emptyset\implies p_0\neq p_1$.
We will show that if $J(p_0,p_1)$ is
  empty then $p_0=p_1$. 
  
For all $j\in \Z$, since $j\not\in J(p_0,p_1)$, the sets
$(\pi_j^\infty)^{-1}(p_0)$, $(\pi_j^\infty)^{-1}(p_1)$ 
intersect adjacent cells of $X_j$, and therefore the values of $x_j$
on these sets agree to within error $2m^{-j}$.
Since
  $x_\infty(p_i)=x_j((\pi_j^\infty)^{-1}(p_i))$, we get that 
  $|x_\infty(p_0)-x_\infty(p_1)|\leq 2m^{-j}$.  As this holds for all $j$ we conclude that
   $x_\infty(p_0)=x_\infty(p_1)$.

  Let $t=x_\infty(p_0)=x_\infty(p_1)$.

If $t\in \cup_{j\in \Z}\,(m^{-j}\Z)$ select $j_0\in\Z$ such that $t\in m^{-j_0}\Z$,
and otherwise let $j_0\in\Z$ be arbitrary.  Then 
$t\not\in m^{-(j+1)}\Z\setminus  m^{-j}\Z$ for every $j\geq j_0$, and by the definition
of $\calr_j$, it follows that $(\pi_{j_0}^\infty)^{-1}(p_i)$ 
contains   a unique element
$\hat p_i\in X_{j_0}$ for $i=0,1$.

Let $j_1\geq 0$ be arbitrary.

{\em Claim. $\hat p_0$ may be joined to $\hat p_1$ 
by a chain of at most four $2$-cells of $X_{j_0}^{(j_1)}$.}

Pick $j\geq j_0$, to be determined later.

If $t\in m^{-j}\Z$, we define $t_\pm =t\pm m^{-j}$; if $t\not\in m^{-j}\Z$
then for some $k_j\in \Z$ we have
$t\in (k_jm^{-j},(k_j+1)m^{-j})$, and then we let $t_-=k_jm^{-j}$, $t_+=(k_j+1)m^{-j}$.
Let $g_j=\min(\{\ell\mid t_-\in m^{-\ell}\Z\},\{\ell\mid t_+\in m^{-\ell}\Z\})$\,.
Since $|t_\pm-t|\leq m^{-j}$, we will have 
\begin{equation}
\label{eqn_gj_bound}
\min(g_j-j_0,j-j_0)>j_1
\end{equation} 
for all but 
finitely many $j$, so we assume (\ref{eqn_gj_bound}) holds.
  
Choose adjacent cells $\tau_0,\tau_1$ of $X_j$
such that $\tau_i\cap (\pi_j^\infty)^{-1}(p_i)\neq\emptyset$, and
cells $\hat\tau_0,\hat\tau_1$ of $X_{j_0}$ such that $\pi_{j_0}^j(\hat\tau_i)=\tau_i$.
Let $v\in\tau_0\cap\tau_1$ be a $0$-cell  of $X_j$, and choose $\hat v_i\in \hat\tau_i$
such that $\pi_{j_0}^j(\hat v_i)=v$.  

{\em Case 1: $x_j(v)=t$.}  As $\pi_{j_0}^j$ is injective on 
$x_j^{-1}(t)$, we have $\hat v_0=\hat v_1$, so $\hat \tau_0\cap \hat \tau_1\neq\emptyset$.
Thus $\hat p_0,\hat p_1$ may be joined by a chain of at most two $2$-cells of 
$X_{j_0}^{(j-j_0)}$, and hence also by a chain of at most two $2$-cells of
$X_{j_0}^{(j_1)}$, proving the claim.

{\em Case 2: $x_j(v)\neq t$.} 
Then $x_j(v)\in\{t_-,t_+\}$.
Since $\pi_{j_0}^j(\hat v_i)=v$ it follows that $\hat v_0$ may be joined
to $\hat v_0$ by a vertical edge path in $X_{j_0}^{(g_j-j_0)}$ of combinatorial
length at most $2$.  
Hence $\hat p_0$ may be joined to $\hat p_1$ may be joined
by a chain of at most four $2$-cells of $X_{j_0}^{(j_1)}$, and the claim holds in this
case.

Since $j_1$ is arbitrary, the claim  forces
$\hat p_0=\hat p_1$.

  \end{proof}

\section{David-Semmes regularity of the projection $\R^2\ra X_\infty$ and
the lower bound on the topological dimension}
\label{sec_n_equals_2_case}

In this section we  prove part (1) of  Theorem
\ref{thm_top_dim_n_anal_dim_1}, and the lower bound in
part (2),  in the $n=2$ case.

Let $\hat d_\infty^Y$, $\al$, and $d_\al$ be as in Section \ref{sec_overview}.
 We recall that
 $\L^2$ denotes  Lebesgue measure on $\R^2$, and for $j\in\Z\cup \{\infty\}$
we let $\mu_j=(\hat\pi^j)_{\#}\L^2$.

\begin{lemma}
\label{lem_hat_d_y_partial_snowflake}
There is a constant $C$ such that $C^{-1}\,d_\al\leq \hat
d_\infty^Y\leq C\,d_\al$.
Moreover, the $Q$-dimensional Hausdorff measure
is uniformly comparable to a Lebesgue measure. 
\end{lemma}
\begin{proof}
 The metric on $Y=\real^2$ is uniformly
comparable to a product metric $d_{hor}\times d_{vert}$; $d_{hor}$ is
just the standard metric on $\real$; on the other hand, $d_{vert}$ is the largest
pseudodistance which makes each cell of generation $j\in\zahlen$ of
the standard $m_v$-adic subdivision of
$\real$ have
diameter at most $m^{-j}$. 
Specifically, for any $x,x'\in \R$, the distance $d_{vert}(x,x')$ is the 
infimum of the quantities $\sum_i \length(\si_i)$, where $\si_1,\ldots,\si_\ell$
is a chain of $m_v$-adic intervals joining $x$ to $x'$, and
$\length(\si_i)=m^{-j_i}$ when $\si_i$ has generation $j_i$
(cf. Section \ref{sec_metric_structure}); since 
$|x-x'|\leq \sum_i m_v^{-j_i}=\sum_i (\length(\si_i))^{\frac{1}{\al}}$
one readily checks that $d_{vert}$ is uniformly comparable to
the snowflake of the standard metric on $\real$ by the 
power $\al=\frac{\log m}{\log m_v}$.

Finally, as $d_\infty^Y$-balls of radius $r$
 have Lebesgue measure $\approx r^Q$, we conclude that
$\hmeas Q.\approx \lebmeas 2.$
\end{proof}

\begin{lemma}[David-Semmes regularity of $\hat\pi^\infty$]
\label{lem_david_semmes_regular_map}

\mbox{}
\begin{enumerate}
\item
\label{item_regular_map}
The projection map  
$\hat\pi^\infty:(\R^2,\hat d_\infty^Y)\ra (X_\infty,\hat d_\infty)$
is a David-Semmes regular map.   
\item 
\label{item_ar}
$(X_\infty,\hat d_\infty)$ is
Ahlfors $Q$-regular, where $Q=1+\al^{-1}$.
\item 
\label{item_top_dim_2}
$(X_\infty,\hat d_\infty)$ has topological dimension at least $2$.
\end{enumerate}
\end{lemma}
\begin{proof}
(\ref{item_regular_map}).
Suppose $\si$ is a $2$-cell of $Y_j$.  Then by the definition of 
$\hat d_\infty$, the pullback of $(\hat\pi^\infty)^*\hat d_\infty$ is a pseudodistance
on $\R^2$ with respect to which  the diameter of $\si$ is $\leq m^{-j}$.  
From the definition of $\hat d_\infty^Y$ we therefore have
$\hat d_\infty^Y\leq (\hat\pi^\infty)^*\hat d_\infty$, so
$\hat\pi^\infty:(\R^2,\hat d_\infty^Y)\ra (X_\infty,\hat d_\infty)$ is $1$-Lipschitz.

Pick $\bar p\in X_\infty$, $r>0$.

  We want to show that 
$(\hat\pi^\infty)^{-1}(B(\bar p,r))$ may be covered by a controlled number of balls
of radius comparable to $r$ in $(\R^2,\hat d_\infty^Y)$.    
To that end, we choose $j\in\Z$ such that
$m^{-(j+1)}<r\leq m^{-j}$, and pick $p\in (\pi_j^\infty)^{-1}(\bar p)$.

Since $(\pi_j^\infty)^{-1}(B(\bar p,m^{-j}))\subset B(p,3m^{-j})$ by 
Proposition \ref{prop_distance_lower_bound}, we have
$$
(\hat\pi^\infty)^{-1}(B(\bar p,r))\subset (\hat\pi^\infty)^{-1}(B(\bar p,m^{-j}))
$$
$$
=(\hat\pi^j)^{-1}(\pi_j^\infty)^{-1}(B(\bar p,m^{-j}))
\subset (\hat\pi^j)^{-1}(B(p,3m^{-j}))\,.
$$
Note that $B(p,3m^{-j})$ is contained in a controlled number of $2$-cells of $X_j$
by 
Lemma \ref{lem_bounded_complexity}(\ref{item_ball_complexity}), 
and the inverse image of each of these
may be covered by at most  
$27$ 
cells of $Y_j$ by 
Lemma \ref{lem_bounded_complexity}(\ref{item_cell_inverse_bounded}).
Since a cell of $Y_j$ has $\hat d_\infty^Y$-diameter $\leq m^{-j}<m\cdot r$,
we are done.

(\ref{item_ar}).  This follows immediately from (\ref{item_regular_map}) since
regular maps preserve Ahlfors-regularity.

(\ref{item_top_dim_2}).  By (\ref{item_regular_map}) the point inverses 
$(\hat\pi^\infty)^{-1}(\bar p)$ have controlled cardinality; in particular,
they are totally disconnected.  
Therefore $\hat\pi^\infty$ cannot 
decrease the topological dimension (see for example
\cite[Thm.~1.24.4]{engelking_topbook});  therefore, the topological dimension of 
$X_\infty$ is at least that of $\R^2$.

One may give a more concrete proof that $X_\infty$ has topological dimension $\geq 2$ 
along the following lines.  Let $f$ be
the restriction of $\hat\pi^\infty:\R^2\ra X_\infty$ to the boundary of $[0,1]^2$;
note also that $f$ is injective.  To see this, suppose 
that $p,p'\in \D[0,1]^2$ are distinct points with $f(p)=f(p')$.  Then 
$x(p)=x_\infty(f(p))=x_\infty(f(p'))=x(p')$, and by the definition of the 
equivalence relation we must have $x(p)=x(p')\in m^{-j}\Z\setminus m^{-(j-1)}\Z$
for some $j\in\Z$.  In view of the definition of $\calr_\infty$,
by inspection we get a contradiction to the fact that $p,p'\in\D[0,1]^2$.   
Letting $\ga=f(\D[0,1]^2)\subset X_\infty$ 
be the image, we apply the Tietze
extension theorem to  extend $g=f^{-1}:\ga\ra \D[0,1]^2$ to a continuous map
$\hat g:X_\infty\ra \R^2$.  By degree theory,
the origin is a stable value of the composition 
$$
[0,1]^2\stackrel{\hat\pi^\infty}{\lra}X_\infty\stackrel{\hat g}{\lra}\R^2\,,
$$
so $\hat g$ cannot be approximated by a map to
$\R^2\setminus \{0\}$.  Therefore $X_\infty$ has topological
dimension at least $2$.

\end{proof}

\section{The upper bound on the Assouad-Nagata dimension}
\label{sec_assouad_nagata}

The proof of the upper bound on the topological dimension is more subtle
than the proof of the lower bound.  Note that the existence of a map
$\R^2\ra X_\infty$ with finite point inverses is not by itself enough
to imply that $X_\infty$ has topological dimension $2$: 
recall that the Peano curve $[0,1]\ra [0,1]^2$ is a finite-to-one surjective
map, showing that such maps can  increase the topological dimension.

We now recall the definition of the Assouad-Nagata dimension.
Let   $\C$ be a cover  of a metric space $X$.  Then
$\C$ is  {\em $r$-bounded} if 
$\diam(C)\leq r$ for all $C\in \C$, and $\C$ has {\em $r$-multiplicity
at most $k$} 
 if every  ball of radius $r$ intersects at most $k$ elements of 
$\C$.

\begin{definition}\cite{assouad_nagata,lang_schlichenmaier}
The {\em Assouad-Nagata dimension} of a metric space $X$ is the infimum of the 
integers $n\geq 0$ such that for some $c>0$ and every $r>0$,
there is a  cover $\C$ of  $X$ that is $cr$-bounded and has $r$-multiplicity
at most $n+1$.
\end{definition}
Note that the Assouad-Nagata dimension is bounded below by the topological dimension
\cite[Prop. 2.2]{lang_schlichenmaier}.

\begin{thm}
  \label{thm:nagata_bound}
  The Assouad-Nagata dimension of $X_\infty$ is at most $2$.
\end{thm}
\begin{proof}
  We exhibit ``good coverings'' in the sense of
\cite[Prop. 2.5]{lang_schlichenmaier}(4): we will show that for every $k\in\Z$,
there is a cover $\C_0\cup\C_1\cup\C_2$ of $X_\infty$ such that
$\C_i$ is $m^{-k}$-bounded and has $m^{-(k+1)}$-multiplicity $1$
for all $i\in\{0,1,2\}$.  This 
implies that the Assouad-Nagata dimension of $X_\infty$ is at most $ 2$.

  \par  Consider the family $\Omega_2$ of $2$-cells of $X_k$; from
  each $\sigma\in\Omega_2$ produce a subset $c_\sigma$ by taking the closure of
  the subset of $\sigma^{(1)}$ obtained by removing those points
  $p\in\sigma^{(1)}$ which satisfy one of the following:
  \begin{equation}
  \begin{aligned}
    \label{eq:nagata_bound_p1}
    d_\real\left(x(p),\min x(\sigma^{(1)})\right)&\le m^{-(k+1)}/2,\\
    d_\real\left(x(p),\max x(\sigma^{(1)})\right)&\le m^{-(k+1)}/2,\\
    d_{S^1}\left(y_k(p),0\right)&\le 3m_v^{-(k+1)}.
  \end{aligned}
\end{equation}
Note that $\diam\dirprj k,\infty,.(c_\sigma)\le m^{-k}$. Consider now
different cells $\sigma_0,\sigma_1\in \Omega_2$; for $i=0,1$ let
$p_i\in \dirprj k,,.(c_{\sigma_i})$; as $c_{\sigma_i}$ does not meet
the $1$-skeleton of $X_k$, any cell 
$\tau$ of $X_{k+1}$ intersecting $\dirprj k,,m.(p_i)$ must intersect a
given cell $\tau_i$ which belongs to the set of those
subcells of $\dirprj k,k+1,.(\sigma_i)$ which meet $\dirprj k,k+1,.(c_{\sigma_i})$. Let $\tau'_i$ be a cell adjacent to $\tau_i$;
recall that $X_{k+1}$ is obtained from $X_k$
by quotienting by 
 $(\hat\pi^k\circ\Phi^k)_*\calr_0$;
and thus (\ref{eq:nagata_bound_p1}) guarantees  that either $\tau'_0$
and $\tau'_1$ are not adjacent or the sets $x(\tau'_1)$ and
$x(\tau'_0)$ are at distance at least $m^{-(k+1)}$.
Thus, applying  Corollary \ref{cor:cell_dist} and the fact that $x$ is
$1$-Lipschitz, we obtain
\begin{equation}
  \label{eq:nagata_bound_p2}
  \dirdst {{\dirprj k,,.(c_{\sigma_0})}},{{\dirprj
      k,,.(c_{\sigma_1})}},.\ge m^{-(k+1)}.
\end{equation}
\par Let $\Omega_1$ be the collection of $1$-cells of $X_k$; from
$e\in\Omega_1$ produce a subset $c_e$ as follows: let $e^{(1)}$ be the
subdivision of $e$ in $X_k^{(1)}$ and let $C_e$ be the collection of
those points $p$ which belong to a $2$-cell of $X_k^{(1)}$ which
intersects $e$; then, if $e$ is vertical, $c_e$ is obtained from $C_e$
by taking the points $p\in C_e$ satisfying:
\begin{equation}
  \label{eq:nagata_bound_p3}
  \begin{aligned}
    d_\real(x(p),x(e))&\le m^{-(k+1)}/2\\
    d_{S^1}(y_k(\tau),0)&\ge 5m_v^{-(k+1)};
  \end{aligned}
\end{equation}
if $e$ is horizontal, $c_e$ is obtained from $C_e$ by taking the
points $p\in C_e$ satisfying:
\begin{equation}
  \label{eq:nagata_bound_p4}
  \begin{aligned}
    d_\real(x(p),\{\max x(e),\min x(e)\})&\ge m^{-(k+1)}\\
    d_{S^1}(y_k(p),0)\le 3m_v^{-(k+1)}.
  \end{aligned}
\end{equation}
Let $e_0,e_1$ be distinct cells in $\Omega_1$. Then either $x(e_0)$
and $x(e_1)$ are  at distance $m^{-(k+1)}$ apart, or there are no adjacent
cells $\tau_i$ of $X_{k+1}$ such that $\tau_i\cap \dirprj
k,k+1,.(c_{e_i})\ne\emptyset$. Thus by Corollary \ref{cor:cell_dist} we
conclude that:
\begin{equation}
  \label{eq:nagata_bound_p5}
  \dirdst {{\dirprj k,k+1,.(e_0)}},{{\dirprj k,k+1,.(e_1)}},.\ge m^{-(k+1)}.
\end{equation}
\par Let $\Omega_0$ be the collection of vertices of $X_k$. For
$v\in\Omega_0$ consider the set $C_v$ consisting of those $2$-cells of
$X_k^{(1)}$ within combinatorial distance $5$ from a cell containing
$v$; then $c_v$ consists of those points $p\in C_v$ satisfying:
\begin{equation}
  \begin{aligned}
  \label{eq:nagata_bound_p6}
  d_\real(x(p),x(v))&\le m^{-(k+1)}\\
  d_{S^1}(y_k(p),0)&\le 5m_v^{-(k+1)}.
\end{aligned}
\end{equation}
Then $\diam\dirprj k,,.(c_v)\le m^{-(k+1)}$. Let $v_0$ and $v_1$ distinct
vertices. If $x(v_0)\ne x(v_1)$ then $\dirdst {{\dirprj
    k,,.(c_{v_0})}},{{\dirprj k,,.(c_{v_1})}},.\ge m^{-(k+1)}.$ Otherwise, there are no adjacent
cells $\tau_i$ of $X_{k+1}$ such that $\tau_i\cap \dirprj
k,k+1,.(c_{v_i})\ne\emptyset$. Thus Corollary \ref{cor:cell_dist}
gives:
\begin{equation}
  \label{eq:nagata_bound_p7}
  \dirdst {{\dirprj
    k,,.(c_{v_0})}},{{\dirprj k,,.(c_{v_1})}},.\ge m^{-(k+1)}.
\end{equation}
\par The families $\left\{\dirprj
  k,,.(c_\sigma)\right\}_{\sigma\in\Omega_2}$, $\left\{\dirprj
  k,,.(c_e)\right\}_{e\in\Omega_1}$ and $\left\{\dirprj
    k,,.(c_v)\right\}_{e\in\Omega_0}$ provide a good covering
  of $X_\infty$.
\end{proof}
\section{The Poincar\'e inequality}
\label{sec_poincare_inequality}
In this section we prove that $(X_\infty,\mu)$ satisfies a Poincar\'e inequality.
Following Semmes \cite{semmes}, we do this by showing that any two points
$p,q\in X_\infty$ can be joined by a ``pencil'' --- a good measured family of curves.  
To obtain such a
family, we first construct a string of horizontal galleries that connects
$p$ to $q$; we then convert this to a  measured family of curves by 
replacing each $2$-cell with
the corresponding 
(appropriately normalized) 
measured family of horizontal geodesics, and then concatenating.
\subsection{Construction of galleries}
\label{subsec:galleries}
\def\Yset#1.{\setbox1=\hbox{$#1$\unskip}Y_{\ifdim\wd1>0pt #1\fi}}
\begin{lemma}
  \label{lem:gall_joins} For each $C>0$ there is  an $L=L(C)$
  such that the following holds.
  For every $j\in \Z$, and every pair $\sigma_{0},\sigma_{1}$ of $2$-cells of $X_j$ such
  that 
  \begin{equation}
    \label{eq:gall_joins_s1}
  \dirdst
  {{\ndirprj j,,.(\sigma_{0})}},{{\ndirprj
      j,,.(\sigma_{1})}},.\le C m^{-j};
\end{equation}
there is a horizontal
  gallery $T$ in $X_j$ whose combinatorial
  length is at most $L$ and which
  starts at
  $\sigma_{0}$ and ends at $\sigma_{1}$.
\end{lemma}
\begin{proof}
  \par We first show that if
  $\sigma_{0}\cap\sigma_{1}\ne\emptyset$ there is a horizontal
  gallery $S$ of length at most $L_0$ starting  at
  $\sigma_{0}$ and ending at $\sigma_{1}$. If $\sigma_{0}$ and
  $\sigma_{1}$  share a vertical $1$-cell we can just take
  $S=\{\sigma_{0},\sigma_{1}\}$. If $\sigma_{0}$ and
  $\sigma_{1}$ share a horizontal $1$-cell, we know
  by the definition of $\calr_{j-1}$
  that there is a horizontal gallery $S$ of length  at most $C_0=C_0(m,m_v)$
  which starts at $\sigma_{0}$ and ends at $\sigma_{1}$. If
  $\sigma_{0}$ and $\sigma_{1}$  share only a
  $0$-cell, we know by
  Lemma~\ref{lem_bounded_complexity}~(\ref{item_links_connected}) that
  we can find a gallery $S$ joining them and of length at
  most $C_1=C_1(m,m_v)$. This can be turned into a horizontal gallery
  of length at most $C_0C_1$ by replacing consecutive cells which
  share a horizontal edge by a horizontal gallery of length at most
  $C_0$ connecting them. Thus one can take $L_0=C_0C_1$.
  
  We now turn to the general case. By
  Corollary~\ref{cor_dinfty_comb_dist} there is a chain $T_0$ in $X_j$
  consisting of at
  most $N(C)$ $2$-cells, which starts at $\sigma_{0}$ and ends at
  $\sigma_{1}$. For any pair of consecutive $2$-cells in $T_0$ we can use
  the discussion in the previous paragraph to construct a horizontal
  gallery connecting them. We thus obtain a horizontal gallery $T$
  of length at most $N(C)L_0$ starting at $\sigma_{0}$ and ending at
  $\sigma_{1}$.
\end{proof}

Next, using Lemma \ref{lem:gall_joins}, we will show that one
can connect any pair of points
$p,q\in X_\infty$ by stringing together  geometrically shrinking
sequences of horizontal galleries.  This is made precise in the
following definition.

\def\galleft#1,#2.{\setbox0=\hbox{$#2$\unskip}{\mathcal
    G}_{#1}\ifdim\wd0>0pt ^{(#2)}\fi}
\def\galright#1,#2.{\setbox0=\hbox{$#2$\unskip}{\mathcal
    T}_{#1}\ifdim\wd0>0pt ^{(#2)}\fi}
\def\fcleft#1,#2.{e_{#1}^{(#2)}}
\def\fcright#1,#2.{f_{#1}^{(#2)}}
\begin{defn}
  \label{defn:string_gallery}
  Let $p,q\in X_\infty$. A \textbf{string of galleries} connecting $p$
  to $q$ consists of four sequences $\{\sigma_j\}_{j\ge j_0}$,
  $\{\tau_j\}_{j\ge j_0}$, $\{\galleft j,.\}_{j> j_0}$, $\{\galright
  j,.\}_{j> j_0}$ which satisfy the following additional
  conditions for some constant $C_0$:
  \begin{description}
  \item[(G1)] $\dirdst p,q,.\approx_{C_0}m^{-j_0}$;
  \item[(G2)] $\sigma_j$ and $\tau_j$ are cells of $X_j$ and
    $\sigma_{j_0}=\tau_{j_0}$;
  \item[(G3)] We have the following control on the distances from $p$
    and $q$:
    \begin{equation}
      \begin{aligned}
        \label{eq:string_gallery_1}
        \dirdst p, {
          \ndirprj j,,.(\sigma_j)
        },.&\le C_0 m^{-j};\\
        \dirdst q, {
          \ndirprj j,,.(\tau_j)
        },.&\le C_0 m^{-j};
      \end{aligned}
    \end{equation}
  \item[(G4)] For each $j\ge j_0$ the vertical faces of $\sigma_j$ and
    $\tau_j$ are ordered: we denote those of $\sigma_j$ by $\fcleft
    j,0.$, $\fcleft j,1.$, and those of $\tau_j$ by $\fcright j,0.$
    and $\fcright j,1.$. We also require $\fcleft j_0,1.=\fcright j_0,0.$;
  \item[(G5)] $\galleft j,.$ (resp.~$\galright j,.$) is a collection
    of $m$ horizontal galleries $\{\galleft j,1.,\cdots,\galleft j,m.\}$
    (resp.~$\{\galright j,1.,\cdots,\galright j,m.\}$); let $\{\fcleft
    {j-1,1},0.,\cdots, \fcleft {j-1,m},0.\}$ (resp.~$\{\fcright
    {j-1,1},1.,\cdots, \fcright {j-1,m},1.\}$) denote the  $1$-cells in
     the first subdivision of $\fcleft j-1,0.$ (resp.~$\fcright j-1,1.$);
    then $\galleft j,i.$ (resp.~$\galright j,i.$) is a
    horizontal gallery of at most $C_0$ cells which connects $\fcleft
    j,1.$ (resp.~$\pi_{j-1}(\fcright {j-1,i},1.)$) to $\pi_{j-1}(\fcleft
    {j-1,1},0.)$ (resp.~$\fcright j,1.$).
  \end{description}
\end{defn}
\begin{lemma}
  \label{lem:string_gallery}
  Suppose $p,q\in X_\infty$ and $\dirdst p,q,.\in[m^{-j_0-1},m^{-j_0})$.
  Then there is a string of galleries
  connecting $p$ to $q$ where the constant $C_0$ does not depend on
  the pair $p,q$.
\end{lemma}
\begin{proof}
  As $\dirdst p,q,.<m^{-j_0}$ by
  Proposition~\ref{prop_distance_lower_bound} we conclude that:
  \begin{equation}
    \label{eq:string_gallery_p1}
    d_{j_0}\left(
      \ndirprj j_0,,a.(p) ,
      \ndirprj j_0,,a.(q)
      \right) < 3m^{-j_0},
    \end{equation}
    and thus $\ndirprj j_0-1,,a.(p)$ and $\ndirprj j_0-1,,a.(q)$
    intersect adjacent cells of $X_{j_0-1}$; we can therefore find a
    $C_0$ independent of $p,q$ and  a single cell
    $\sigma_{j_0}=\tau_{j_0}$ of $X_{j_0}$ such
    that~(\ref{eq:string_gallery_1}) holds with $j=j_0$. For each $j$
    choose $\sigma_j$ and $\tau_j$ such that $\ndirprj
    j,,a.(p)\cap\sigma_j\ne\emptyset$ and $\ndirprj
    j,,a.(q)\cap\tau_j\ne\emptyset$. We now choose a first and a last
    vertical edge for each $\sigma_j$ and $\tau_j$ as in
    \textbf{(G4)}. We now construct the collection of galleries
    $\{\galleft j,.\}_{j>j_0}$ and $\{\galright j,.\}_{j>j_0}$. By
    symmetry we just focus on the construction of $\galleft j,.$. By
    the choice of the cells $\{\sigma_j\}_{j\ge j_0}$ we have that:
    \begin{equation}
      \label{eq:string_gallery_p2}
      d_j\left(
        \sigma_j , \pi_{j-1}(\sigma_{j-1})
        \right) \le C_0m^{-j};
      \end{equation}
      therefore, by possibly enlarging $C_0$, for each $i\in\{1,\cdots,m\}$ we can choose by
      Lemma~\ref{lem:gall_joins} a horizontal gallery consisting of at most $C_0$ cells
  \begin{equation}
    \label{eq:meas_curve_join_p5}
    \galleft j,i.=\{\theta_1,\cdots,\theta_L\}\subset X_{j}
  \end{equation}
  which joins $\fcleft j,1.$ to $\pi_{j-1}(\fcleft {j-1,i},0.)$.
\end{proof}
\begin{corollary}
  \label{cor:mono_galleries}
  Let $p,q\in X_\infty$ be such that $\dirdst
  p,q,.\in[m^{-j_0-1},m^{-j_0})$; then there is a universal constant
  $C_1$ such that, for each $M>j_0$, there is a
  horizontal gallery $\galleft M,.=\{\sigma_0,\cdots,\sigma_L\}$
  consisting of
  $2$-cells of $X_M$ and whose length is at most
  $C_0m^{-j_0}$, where one has:
  \begin{equation}
    \label{eq:mono_galleries_s1}
    \begin{aligned}
      \dirdst{{\ndirprj M,,.(\sigma_0)}},p,.&\le C_0m^{-M};\\
      \dirdst{{\ndirprj M,,.(\sigma_L)}},q,.&\le C_0m^{-M}.
    \end{aligned}
  \end{equation}
\end{corollary}
\begin{proof}
  We use Lemma~\ref{lem:string_gallery} to take a string of galleries
  connecting $p$ to $q$ of length at most $C(C_0)m^{-j_0}$, where
  $C_0$ is the constant in Lemma~\ref{lem:string_gallery}. We now
  truncate this string to obtain a horizontal gallery in $X_M$:
  \begin{equation}
    \label{eq:mono_galleries_p1}
    \begin{split}
    T=\biggl\{
        \sigma_M, \galleft M,1., \pi_{M-1}(\sigma_{M-1}),
        \pi_{M-1}(\galleft
        M-1,1.),\cdots,\\\pi_{j_0}^M(\sigma_{j_0})=\pi_{j_0}^M(\tau_{j_0}),
        \pi_{j_0+1}^M(\galright j_0+1,1.),\cdots,\tau_M
        \biggr\};
      \end{split}
    \end{equation}
    we now inductively modify $T$ to increase the minimal generation
    of cells in $T$ to end up with a gallery consisting of only
    $2$-cells of $X_M$. At the first step, we take the end cell $\sigma$ of
    $\galleft j_0+1,1.$ in $X_{j_0+1}$, and the first cell $\tau$ of
    $\galright j_0+1,1.$ in $X_{j_0+1}$; as both cells intersect some
    vertical faces of a pair of $2$-cells in
    $\pi_{j_0}^{M-1}(\sigma_{j_0})$, we can find a horizontal gallery
    $S$ in $X_{j_0+1}$ connecting $\sigma$ to $\tau$ whose length is at most
    $C(H)m^{-j_0}$. We then replace $\pi_{j_0}^M(\sigma_{j_0})$ by
    $\pi_{j_0+1}^M(S)$. A similar process is applied now to replace
    the cells of $T$ of generation $j_0+1$ by cells of generation
    $j_0+2$. The increase in length is of a factor $C(H)m^{-j_0-1}$
    for each cell of generation $j_0+1$. Continuing inductively, one
    obtains a gallery $T$ consisting of cells of generation $M$ and
    whose length is $\le C(C_0,H,m)m^{-j_0}$.
\end{proof}

\subsection{Constructing Semmes pencils}

\begin{defn}
  \label{defn:meas_curv_fam}
  A measured family of curves $(\Gamma,I_\Gamma,\Omega_\Gamma,\nu_\Gamma)$ is a measurable
  function $\Gamma:I_\Gamma\times\Omega_\Gamma\to X$, where $I_\Gamma$ is an interval of
  $\real$, such that $\nu_\Gamma$ is a probability measure on $\Omega_\Gamma$, and
  there is an $L$ such that for each $\omega\in\Omega_\Gamma$ the map
  $t\mapsto\Gamma(t,\omega)$ is a Lipschitz
  curve in $X$ of Lipschitz constant at most $L$.
  \par We say that a measured family of curves
  $(\Gamma,I_\Gamma,\Omega_\Gamma,\nu_\Gamma)$ joins a set $S_0$ to a set $S_1$ if
  for $\nu_\Gamma$-a.e.~$\omega\in\Omega_\Gamma$ one has $\Gamma(\min
  I_\Gamma,\omega)\in S_0$ and $\Gamma(\max
  I_\Gamma,\omega)\in S_1$. Given a subinterval $I\subset I_\Gamma$ we
  denote by $\Gamma_I$ the restriction $\Gamma:I\times\Omega_\Gamma\to X$. 
  \par To a measured family of curves $(\Gamma,I_\Gamma,\Omega_\Gamma,\nu_\Gamma)$
  we associate a measure $\mu_\Gamma$ by:
  \begin{equation}
    \label{eq:meas_curv_fam_1}
    \mu_\Gamma=\Gamma|_{\#}(\lebmeas.\times\nu_\Gamma).
  \end{equation}
  \par In the following we will often simply write $\Gamma$ instead of
  $(\Gamma,I_\Gamma,\Omega_\Gamma,\nu_\Gamma)$. We finally define the support of
  $\Gamma$ by:
    \begin{equation}
    \label{eq:meas_curv_fam3}
    \spt\Gamma=\left\{p\in X_\infty: \forall r>0\, \nu_\Gamma(\{\omega:\Gamma(I_\Gamma,\omega)\cap B(p,r)\ne\emptyset\})>0\right\}.
  \end{equation}
\end{defn}
\begin{defn}
  \label{defn:Riesz}
  For a metric measure space $(X,\mu)$ recall the definition of the
  Riesz potential $\mu_p$ centered at $p$:
  \begin{equation}
    \label{eq:riesz}
    \mu_p = \frac{d(p,\cdot)}{\mu(B(p,d(p,\cdot)))}\,\mu.
  \end{equation}
\end{defn}
We will denote
the Riesz potential of $(X,\mu_\infty)$ centered on $p\in X_\infty$ by $\mu_{\infty,p}$.
To prove the Poincar\'e inequality and to establish a bound on the
analytic dimension we rely on the following theorem. 
\begin{thm}
  \label{thm:meas_curve_join}
  There is a universal constant $C$ depending only on $(m,m_v)$ such
  that the following holds.
  \begin{description}
\item[(C1)]  For $p,q\in X_\infty$, there is a measured family of curves
  $\Gamma$ joining $p$ to $q$;
\item[(C2)] The Lipschitz constant of
  $\Gamma$ is at most $C$ and $\lebmeas .(I_\Gamma)\le C\dirdst
  p,q,.$;
  \item[(C3)] The measure $\mu_\Gamma$ is controlled by the Riesz potentials
    centered at $p$ and $q$:
  \begin{equation}
    \label{eq:meas_curve_join_s1}
    \mu_\Gamma\le C (\mu_{\infty,p}\on B(p,Cd(p,q)) + \mu_{\infty,q}\on B(q,Cd(p,q))).
  \end{equation}
\end{description}
\end{thm}
\begin{proof}
  Using Lemma~\ref{lem:string_gallery}, we take a string of galleries
  $S=\{$ $\{\sigma_j\}_{j\ge j_0}$, $\{\tau_j\}_{j\ge j_0}$, $\{\galleft
  j,.\}_{j>j_0}$, $\{\galright j,.\}_{j>j_0}$ $\}$ connecting $p$ to
  $q$. The idea of the proof is based on the following observations:
  first, one can use $S$ to string together
  horizontal segments to obtain 
  Lipschitz paths
   joining $p$ to
  $q$. Secondly, the Fubini representations of the measures on the
  cells of $S$ give rise to a 
  ``natural'' transverse measure
   on the family of quasigeodesics connecting
  $p$ to $q$. To ease the
  exposition, we have divided the proof in two steps. In the first we
  discuss how to build a measured family of curves joining  $\ndirprj
  j_0,,.(\sigma_{j_0})$ to $p$.
  In the second we join
  together  two measured families of curves connecting the points $p$ and
  $q$ to $\ndirprj j_0,,.(\sigma_{j_0})=\ndirprj j_0,,.(\tau_{j_0})$. 
  \def\Oset#1.{\setbox1=\hbox{$#1$\unskip}\Omega_{\ifdim\wd1>0pt
      #1\fi}}
    \def\nmes#1.{\setbox1=\hbox{$#1$\unskip}\nu_{\ifdim\wd1>0pt
 #1\fi}}
\vskip10pt
  {\emph{Step 1: Construction of a measured family of curves 
      joining $\ndirprj j_0,,.(\sigma_{j_0})$ to $p$.}}
  \par We now construct a 
  measured family of curves
  $(\Gamma_p,I_{p},\Omega_p,\nu_p)$ joining $\ndirprj
  j_0,,.(\sigma_{j_0})$ to $p$. 
   We let:
  \begin{equation}
    \label{eq:meas_curve_join_x1}
    I_{p}=\left[0,\sum_{j_0\le s}m^{-s}\right],
  \end{equation}
  and introduce the intervals:
  \begin{equation}
    \label{eq:meas_curve_join_p3}
    \begin{aligned}
      A_i&=\left[\sum_{j_0\le s<i}m^{-s},\sum_{j_0\le s<i}m^{-s}+\frac{m^{-i}}{2}\right]\quad(i\ge j_0)\\
      B_i&=\left[\sum_{j_0\le
          s<i-1}m^{-s}+\frac{m^{-i+1}}{2},\sum_{j_0\le s<i}m^{-s}\right]\quad(i>j_0).
    \end{aligned}
  \end{equation}\def\aa#1,#2.{a^{(#1)}_{#2}}\def\bb#1,#2.{b^{(#1)}_{#2}}
  Note that by convention a sum over an empty set of indices is taken
  to be $0$; for example, $A_{j_0}=[0,m^{-j_0}/2]$ and $B_{j_0+1}=[m^{-j_0}/2,m^{-j_0}]$.
  For later convenience let $A_i=[\aa 0,i.,\aa 1,i.]$ and $B_i=[\bb
  0,i.,\bb 1,i.]$. We choose a cell $\hat\sigma_{j_0}\subset Y_{j_0}$
  such that $\hat\pi^{j_0}(\hat\sigma_{j_0})=\sigma_{j_0}$ and let
  $\hat e_{j_0}^{(0)}$ and $\hat e_{j_0}^{(1)}$ denote its vertical faces.
   Let $\Omega_p$ denote the subset of
  $\hat e_{j_0}^{(1)}$ obtained by removing points $p'$ such that
  $y(p')=km_v^{-j}$ for $(k,j)\in\zahlen^2$. Note that $\Omega_p$ has
  full $\hmeas 1.$-measure in $\hat e_{j_0}^{(1)}$. Let $\nmes p.=\hmeas
  1.\on\Oset p.$ .
  \par We define a map $\hat
  H_{j_0}:A_{j_0}\times\Omega_p\to\hat\sigma_{j_0}$ by letting $\hat
  H_{j_0}(\cdot,p')$ denote the unique constant
  velocity horizontal curve in $\hat \sigma_{j_0}$ starting at $p'$ and
  ending in $\hat e^{(0)}_{j_0}$. We let
  $\Gamma_p|A_{j_0}\times\Omega_p=\dirprj j_0,,.\circ\hat H_{j_0}$.
  Then Fubini's
  Theorem implies that:
  \begin{equation}
    \label{eq:meas_curve_join_p4}
    (\hat H_{j_0})_{\#}(\lebmeas . \times\nu_p\on A_{j_0}\times\Omega_p)=m_v^{j_0}\lebmeas 2.\on\hat\sigma_{j_0}.
  \end{equation}
  Note that the Lipschitz constant of $\hat H_{j_0}$ is at most $2$.
 \par As there is a uniform bound on the number of cells of $\Yset j_0.$
   that project to $\sigma_{j_0}$, we conclude from
  (\ref{eq:meas_curve_join_p4}) that:
  \begin{equation}
    \label{eq:meas_curve_join_p4tris}
 (\Gamma_p)|_{\#}(\lebmeas.\times\nu_p\on A_{j_0}\times\Omega_p)\approx m_v^{j_0}\mu\on\spt(\Gamma_p|A_{j_0}\times\Omega_p).
\end{equation}
\par Let $\galleft j_0+1,.=\{\galleft j_0+1,i.\}_{1\le i\le m}$ and
note that for $p'\in\Omega_p$ there is a unique $i=i(p')$ such that
the point $\hat\pi_{j_0}^{j_0+1}\circ\hat H_{j_0}(\aa 1,j_0.,p')$
belongs to the last cell of $\galleft j_0+1,i.$. We now write
$\galleft j_0+1,i.=\{\theta_1,\cdots,\theta_L\}$ and for each $1\le
k\le L$ choose a cell $\hat\theta_k\subset Y_{j_0+1}$ such that
$\hat\pi^{j_0+1}(\hat\theta_k)=\theta_k$. Intuitively, the
collection of cells $C_{j_0+1}=\{\theta_k\}_{1\le k\le L}$
 induces a \emph{broken
  horizontal path} $\hat J(\cdot,p'):B_{j_0+1}\to\Yset j_0+1.$. More
precisely, observe that $\hat\theta_L\cap\dirprj j_0+1,,m.(\dirprj j_0,,.(\hat
H_{j_0}(\aa 1,j_0.,p'))$ consists of a single point
$p'_{\theta_L}$. Then there is a unique horizontal constant
velocity path $\xi_{\theta_L}:[\bb 0,j_0+1.,\bb
0,j_0+1.+L^{-1}m^{-j_0}/2]\to \hat\theta_L$ which starts at
$p'_{\theta_L}$ and ends on the vertical $1$-cell of $\hat\theta_L$ other than the one
containing $p'_{\theta_L}$. Let $q'_{\theta_L}=\xi_{\theta_L}(\bb
0,j_0+1.+L^{-1}m^{-j_0}/2)$. Then again $\hat\theta_{L-1}\cap\dirprj
j_0+1,,m.(\dirprj j_0+1,,.(q'_{\theta_L}))$ consists of a single point
$p'_{\theta_{L-1}}$. This happens because $p'\in\Omega_p$ forces 
the points $p'_{\theta_L}$ and $q'_{\theta_L}$ to lie outside the
$0$-skeleton of the cells containing them. We then observe that there
is a unique horizontal constant
velocity path $\xi_{\theta_{L-1}}:[\bb 0,j_0+1.+L^{-1}m^{-j_0}/2,\bb
0,j_0+1.+2L^{-1}m^{-j_0}/2]\to \hat\theta_{L-1}$ which starts at
$p'_{\theta_{L-1}}$ and ends on the $1$-cell opposite to the one
containing $p'_{\theta_{L-1}}$. The construction of the paths
$\xi_{\theta_{L-2}},\cdots, \xi_{\theta_1}$ continues by backward induction,
constructing $\xi_{\theta_{k}}$ knowing $q'_{\theta_{k+1}}$ as we did
for $\theta_{L-1}$, and we
omit the details. We then let for $1\le k\le
L$:
\begin{equation}
  \label{eq:meas_curve_join_bb4}
  \hat J(\cdot,p')|(\bb 0,j_0+1.+kL^{-1}m^{-j_0}/2,\bb 0,j_0+1.+(k+1)L^{-1}m^{-j_0}/2)=\xi_{k-L},
\end{equation}
and extend $\hat J(\cdot,p')$ to $B_{j_0+1}$ so that it is continuous
from the left and $\hat J(\bb 0,j_0+1.,p')=p'_{\theta_L}$. We then let
\begin{equation}
  \label{eq:meas_curve_join_bb5}
  \Gamma_p|B_{j_0+1}\times\Omega_p = \dirprj j_0+1,,.\circ \hat J.
\end{equation}
Note that $\Gamma_p(\cdot,p')|B_{j_0+1}$ is a Lipschitz curve joining
$\dirprj j_0,,.(\hat H_{j_0}(\aa 1,j_0.,p'))$ to a point on the last face
$e^{(1)}_{j_0+1}$ of $\sigma_{j_0+1}$, and whose Lipschitz constant is at most $2L\le 2C$.
Taking into account that  for $i_1\ne i_2$ the cardinality of the
  set:
  \begin{equation}
    \label{eq:meas_curve_join_p8}
    \left\{(\theta,\theta')\in \galleft j_0+1,i_1.\times\galleft j_0+1,i_2.:
      \mu\left(\dirprj j_0+1,,.(\theta)\cap\dirprj j_0+1,,.(\theta')\right)>0\right\}
  \end{equation}
  is uniformly bounded, an argument similar to the one which
  yielded~(\ref{eq:meas_curve_join_p4tris}) gives:
   \begin{equation}
    \label{eq:meas_curve_join_p9}
(\Gamma_p)|_{\#}(\lebmeas.\times\nu_p\on B_{j_0+1}\times\Omega_p)\approx
m_v^{j_0+1}\mu\on\spt(\Gamma_p|B_{j_0+1}\times\Omega_p).
\end{equation}
The construction then continues by induction on the intervals $A_i$,
$B_i$. Note that the choice of $\Omega_p$ guarantees that, knowing
$\Gamma_p(\cdot,p')|\bigcup_{s\le i}A_s\cap\bigcup_{s=j_0+1}^{i}B_s$, there is
a unique extension $\Gamma_p|B_{i+1}\times\Omega_p$ which is obtained using
steps similar to those we used to construct
$\Gamma_p|B_{j_0+1}\times\Omega_p$. The point is again that $\dirprj
i+1,,m.(\Gamma(\aa 1,i.,p'))$ does not intersect the $0$-skeleton of
$\Yset i+1.$ and so the extension to $B_{i+1}$ is uniquely determined.
In particular, one has also the following
analogue of~(\ref{eq:meas_curve_join_p9}):
\begin{equation}
  \label{eq:meas_curve_join_bb9}
  (\Gamma_p)|_{\#}(\lebmeas. \times\nu_p\on B_{i+1}\times\Omega_p)\approx
m_v^{i}\mu\on\spt(\Gamma_p|B_{i+1}\times\Omega_p).
\end{equation}
Similarly,  knowing
$\Gamma_p(\cdot,p')|\bigcup_{s\le i}A_i\cap\bigcup_{s\le i+1}B_i$, there is
a unique extension $\Gamma_p|A_{i+1}\times\Omega_p$ which is  constructed
similarly to how we constructed
$\Gamma_p|A_{j_0}\times\Omega_p$. Moreover, one has the following
analogue of~(\ref{eq:meas_curve_join_p4tris}):
\begin{equation}
  \label{eq:meas_curve_join_bb4tris}
   (\Gamma_p)|_{\#}(\lebmeas. \times\nu_p\on A_{i+1}\times\Omega_p)\approx m_v^{i+1}\mu\on\spt(\Gamma_p|A_{i+1}\times\Omega_p).
 \end{equation}
 We finally let $\Gamma_p(\sum_{s=j_0}^\infty m^{-s},\cdot)=p$ because
 $\Gamma_p(\aa 1,i.,p')\to p$ and
 $\Gamma_p(\bb 1,i.,p')\to p$ uniformly in $p'$ as $i\nearrow\infty$. Thus
 $\Gamma_p(\cdot,p')$ is a Lipschitz curve of Lipschitz constant at
 most $2C$ which joins $\dirprj j_0,,.(p')$ to $p$. Note that
 because of the choice of the string of galleries $S$ (compare~(\ref{eq:string_gallery_1})),
  there is a universal constant
  $C_0$ such that:
  \begin{equation}
    \label{eq:meas_curve_join_p12}
    \spt\Gamma_p|(A_i\cup B_i)\times\Omega_p\subset B(p,C_0m^{-i}).
  \end{equation}
  Let $\mu_{\Gamma_p}$ denote the measure $(\Gamma_p)|_{\#}(\lebmeas. \times\nu_p\,\on[0,\sum_{s=j_0}^\infty m^{-s}]\times\Omega_p)$.
  Thus~(\ref{eq:meas_curve_join_bb4tris}) and (\ref{eq:meas_curve_join_bb9}) imply that:
  \begin{equation}
    \label{eq:meas_curve_join_p13}
    \frac{d\mu_{\Gamma_p}}{d\mu}\bigm|(X_\infty\setminus B(p,C_0m^{-j}))\lesssim\sum_{i=-\infty
      }^jm_v^i\lesssim m_v^j,
  \end{equation} which gives
  \begin{equation}
    \label{eq:meas_curve_join_p13bis}
    \mu_{\Gamma_p}\lesssim\mu_{\infty,p}\on B(p,Cm^{-j_0}),
  \end{equation}
  where $\mu_{\infty,p}$ denotes the Riesz potential centered on $p$.
  \vskip12pt
  \bigskip{\emph{Step 2: Joining two measured families of curves.}}
  We use \emph{Step 1} to build measured families of curves $\Gamma_p$ and
  $\Gamma_q$ joining
  $\pi_{j_0}^\infty(\sigma_{j_0})=\pi_{j_0}^\infty(\tau_{j_0})$ to $p$
  and $q$ respectively. Note that by condition \textbf{(G4)} in the
  properties of the string $S$, both $\Gamma_p$ and $\Gamma_q$ start
  at $\pi_{j_0}^\infty(\fcleft j_0,1.)=\pi_{j_0}^\infty(\fcright
  j_0,0.)$, where $\fcleft j_0,1.$ denotes the last face of
  $\sigma_{j_0}$ and $\fcright j_0,0.$ the first face of $\tau_{j_0}$.
  \par We now want to concatenate the reverse of $\Gamma_p$ and
  $\Gamma_q$. 
  In fact,
  $\Gamma_p$ is defined on $I_p\times\Omega_p$ and $\Gamma_q$ on
  $I_q\times\Omega_p$ where $\Omega_p$ is a full-measure subset of
  $\fcleft {j_0},1.$ and for $p'\in\Omega_p$ we have $\Gamma_p(\min
  I_p,p')=\Gamma_q(\min I_q,p')=p'$. Note also that the transverse
  measures $\nu_p$ and $\nu_q$ agree; we
  can thus  concatenate the reverse
  of $\Gamma_p$ and $\Gamma_q$ obtaining a measured family of curves
  $\Gamma_{pq}$. More precisely let:
  \begin{equation}
    \label{eq:meas_curve_join2_x2}
    I_{\Gamma_{pq}}=I_p\cup (I_q+\max I_p - \min I_q),
  \end{equation}
  and define:
  \begin{equation}
    \label{eq:meas_curve_join_x3}
    \Gamma_{pq}(t,\omega)=
    \begin{cases}
      \Gamma_p(\max I_{\Gamma_p}-t)&\text{if
        $t\in[0,\lebmeas.(I_{\Gamma_p})]$}\\
      \Gamma_q(t-\min I_{\Gamma_q}+\max I_{\Gamma_p})&\text{if $t\in[\lebmeas.(I_{\Gamma_p}),\lebmeas.(I_{\Gamma_p})+\lebmeas.(I_{\Gamma_q})]$}.
    \end{cases}
  \end{equation}
  For the measure $\nu_{pq}$ we take $\nu_p$.
  Now \textbf{(C1)} follows from the choice of the string $S$, and
  \textbf{(C2)} follows since
  $\Gamma_p$ and $\Gamma_q$ are $2C$-Lipschitz and because
  $I_{\Gamma_{pq}}$ is given by~(\ref{eq:meas_curve_join2_x2}).
  Note that:
  \begin{equation}
    \label{eq:meas_curve_join_x4}
    \mu_{\Gamma}=(\Gamma_{pq})|_{\#}(\lebmeas.\times\nu_{pq}\on
    I_{\Gamma_{pq}}\times\Omega_p)=\mu_{\Gamma_p}+\mu_{\Gamma_q},
  \end{equation}
  and so we get~(\ref{eq:meas_curve_join_s1}) in \textbf{(C3)} by
  using~(\ref{eq:meas_curve_join_p13}).
\end{proof}
\begin{thm}
  \label{thm:poinceq}
  The metric measure space $(X_\infty,\mu)$ admits a
  $(1,1)$-Poincar\'e inequality.
\end{thm}
\begin{proof}
  The $(1,1)$-Poincar\'e inequality can be proven by appealing to a
  result \cite[Thm.~1.22]{semmes}. As $(X_\infty,\mu)$ is
  Ahlfors-regular\footnote{actually by the discussion in
    \cite[Chap.~4]{juhabook} it is enough to assume $\mu$ doubling}, one has to show the existence of a universal
  constant $C$ such that the following holds: for any pair of points
  $p,q\in X_\infty$ and any pair $(u,g)$, $u$ being a real-valued
  Borel function on $X_\infty$ and $g$ an upper gradient of $u$, one
  has the estimate:
  \begin{equation}
    \label{eq:poinceq_p1}
    \left|u(p)-u(q)\right|\le C\left(\int_{B(p,Cd(p,q))}g\,d\mu_{\infty,p}+\int_{B(q,Cd(p,q))}g\,d\mu_{\infty,q}\right).
  \end{equation}
  Let $\Gamma$ be a measured family of curves joining $p$ to $q$ as in
  Theorem~\ref{thm:meas_curve_join}. Then for each $\omega\in\Omega_\Gamma$
  one has:
  \begin{equation}
    \label{eq:poinceq_p2}
    \left|u(p)-u(q)\right|\le\int_{I_\Gamma}g(\Gamma(t,\omega))\,dt;
  \end{equation}
  integrating~(\ref{eq:poinceq_p2}) in $\nu_\Gamma$ one obtains:
  \begin{equation}
    \label{eq:poinceq_p3}
    \left|u(p)-u(q)\right|\le\int g\,d\mu_\Gamma,
  \end{equation}
  and~(\ref{eq:poinceq_p1}) follows from~(\ref{eq:meas_curve_join_s1}).
\end{proof}

\section{The analytic dimension}
\label{sec:anal_dim}
In this section we prove parts (\ref{item_analytic_dimension_1}) 
and (\ref{item_horizontal_a_rep})
of Theorem \ref{thm_top_dim_n_anal_dim_1}, see Theorem 
\ref{thm:anal_dim}.

We let  $\lipalg X.$ denote the set of bounded real-valued
Lipschitz functions defined on $X$, which is a Banach algebra with norm given
by:
\begin{equation}
  \label{eq:allip_norm}
  \max\left(\|f\|_\infty,\sup_{x_1\ne x_2}\frac{|f(x_1)-f(x_2)|}{d(x_1,x_2)}\right).
\end{equation}
\begin{defn}
  \label{defn:push_forward}
  Consider a bounded linear operator
  $D:\lipalg X.\to L^\infty(\mu)$ where $\mu$ is a Radon measure on
  the metric space $X$. Suppose that $F:X\to Y$ is Lipschitz; then one
  obtains the \textbf{pushforward} $F_{\#}D:\lipalg Y.\to
  L^\infty(F_{\#}\mu)$ of $D$ as
  follows; given $u\in\lipalg Y.$, $F_{\#}Du$ is determined by the
  requirement:
  \begin{equation}
    \label{eq:push_forward_p1}
    \int_Y g\,F_{\#}Du\,dF_{\#}u=\int_Xg\circ F\,D(u\circ
    F)\,d\mu\quad(\forall g\in L^1(F_{\#}\mu)).
  \end{equation}
\end{defn}
The collection of metric measure spaces $\{(X_j,\mu_j)\}_{j\in\zahlen}$ and maps
$\{\pi_j\}_{j\in\zahlen\cup\{\infty\}}$ gives rise to a direct system of operators
$\{D_j\}_{j\in \Z}$ as follows. Let $\hat D$ be the horizontal derivative operator:
\begin{equation}
  \label{eq:hor_der}
  \hat D:\lipalg\real^2.\to L^\infty(\lebmeas 2.);
\end{equation}
we then 
let $D_j=\dirprj \#,j,.\hat D$ for $j\in \zahlen\cup\{\infty\}$. Note that
if $k\ge j$, $\hat\pi^k=\pi_j^k\circ\hat\pi^j$ and so
$\pi_{j,\#}^kD_j=D_k$.  In fact,   the operators $D_j$ are derivations
in the sense of Weaver \cite{weaver00}, i.e.~satisfy a product rule and a weak*
continuity axiom.  These properties are easy to verify, but will not be
used in the following, except in the alternative argument given in
Remark~\ref{rem:anal_dim_weaver}.
\begin{thm}
  \label{thm:anal_dim}
  $(X_\infty,\mu)$ has analytic dimension $1$ and
  $(X_\infty,x_\infty)$ is a differentiability chart. 
    Moreover, (5) in Theorem~\ref{thm_top_dim_n_anal_dim_1} holds.
\end{thm}
\begin{proof}
  Let $f$ be a Lipschitz function; as the analytic dimension is a
  local property, we will assume $f$ to be bounded. Consider an
  approximate continuity point $p$ of
  $D_\infty f$ and fix $M\in\natural$; using that $p$ is an approximate continuity point of
  $D_\infty f$, for each $\varepsilon>0$ one can choose $r_0(\varepsilon)$
  such that, for any $r\le r_0$, any $i\le
  M+\lfloor\log_m(1/r)\rfloor$ and any
  cell $\hat\sigma_i$ of $Y_{i}$ satisfying
  \begin{equation}
    \label{eq:anal_dim_p1}
    \dirprj i,,.(\hat\sigma_i)\subset B(p,2Cr),
  \end{equation}
  one has:
  \begin{equation}
    \label{eq:anal_dim_p2}
    \av_{\dirprj i,,.(\hat\sigma_i)}\left|D_\infty f(p)-D_\infty f \right|\,d\mu\le\varepsilon.
  \end{equation}
  Let $q\in X_\infty$ satisfy $d(p,q)\le r$. We let:
  \begin{equation}
    \label{eq:anal_dim_p3}
    u=f-f(p)-D_\infty f(p)(x_\infty -x_\infty(p));
  \end{equation}
  we will show that:
  \begin{equation}
    \label{eq:anal_dim_p4}
    \left|u(p)-u(q)\right|\lesssim(m^{-M}+\varepsilon) r,
  \end{equation}
  from which the Theorem follows as $M$ and $\varepsilon$ are
  arbitrary.  
  \par Let $M_r=\lfloor\log_m(1/r)\rfloor+M$; we use
  Corollary~\ref{cor:mono_galleries} to obtain a horizontal gallery
  $\galleft M_r,.=\{\sigma_1,\cdots,\sigma_{L_r}\}$ of $2$-cells of
  $X_{M_r}$ such that $\pi_{M_r}^{\infty}(\galleft M_r,.)$ ``almost
  connects'' $p$ to $q$. More precisely, one has:
  \begin{equation}
    \label{eq:anal_dim_p6}
    \begin{aligned}
      d(\ndirprj M_r,\infty,.(\sigma_1),p)&\le Crm^{-M}\\
      d(\ndirprj M_r,\infty,.(\sigma_{L_r}),q)&\le Crm^{-M};
    \end{aligned}
  \end{equation}
  and 
  \begin{equation}
    \label{eq:anal_dim_p7}
    L_rm^{-M}r\le Cr,
  \end{equation}
  where $C$ is a universal constant.
  \par For $j\in\zahlen$ let $S_j$ denote the $1$-skeleton of $Y_j$;
  then $S_j$ is $\hat\pi^j$-saturated,
  i.e.~$(\hat\pi^j)^{-1}(\hat\pi^j(S_j))=S_j$; as $S_j$ is $\lebmeas
  2.$-null, we conclude that $\hat\pi^j(S_j)$ is $\mu_j$-null. Note
  that $\hat\pi^j$ restricts to a homeomorphism mapping $Y_j\setminus
  S_j$ onto $X_j\setminus\hat\pi^j(S_j)$; in particular, $\hat\pi^j$
  maps open cells of $Y_j$ homeomorphically onto open cells of
  $X_j$. Let $v\in\lipalg X_j.$ and $\sigma$ be an open cell of $X_j$;
  using the definition of pushforward, we conclude that $D_jv$
  coincides $\mu_j$-a.e.~with the horizontal derivative $\hat D
  (v\circ\hat\pi^j)$, i.e.:
  \begin{equation}
    \label{eq:anal_dim_in1}
    D_jv=\hat D(v\circ\hat\pi^j)\circ(\hat\pi^j)^{-1}\quad(\text{on $\sigma$}).
  \end{equation}
  \par Consider two adjacent cells $\sigma_i$, $\sigma_{i+1}$ of
  $\galleft M_r,.$;
  as $\sigma_i$ and $\sigma_{i+1}$ share a vertical face, the Fubini
  representation of $\mu_{M_r}\on (\sigma_i\cup\sigma_{i+1})$ and the
  fundamental Theorem of Calculus imply
  
   \begin{multline}
    \label{eq:anal_dim_p8}
      \left|\av_{\sigma_i} (u\circ\pi_{M_r}^\infty)\,d\mu_{M_r} 
      -\av_{\sigma_{i+1}} (u\circ\pi_{M_r}^\infty)\,d\mu_{M_r}\right|\\
      \lesssim 
      m^{-M} r 
      \av_{\sigma_i\cup\sigma_{i+1}}|D_{M_r}(u\circ\pi_{M_r}^\infty)|\,d\mu_{M_r}.\\
    \end{multline}
    
    \par The goal is now to replace 
     the derivative $D_{M_r}(u\circ\ndirprj
    M_r,,.)$ with $D_\infty u$ in the right hand side
    of~\eqref{eq:anal_dim_p8}. Fix a $2$-cell $\sigma$ of $X_{M_r}$
    and let $\hat\sigma$ be the $2$-cell of $Y_{M_r}$ such that
    $\hat\pi^{M_r}(\hat\sigma)=\sigma$. The map $\hat\pi^{M_r}$
    restricts to a homeomorphism between the interiors of $\hat\sigma$
    and $\sigma$; we now define $g\in L^1(\mu_{M_r})$ by letting:
    \begin{equation}
      \label{eq:anal_dim_in2}
      g=
      \begin{cases}
        \frac{\chi_{\sigma}}{\mu_{M_r}(\sigma)}\left(
          \sgn (D_{M_r}(u\circ\ndirprj M_r,,.))
        \right)&\text{on the interior of $\sigma$}\\
        0&\text{elsewhere;}
      \end{cases}
    \end{equation}
    then we have:
    \begin{equation}
      \label{eq:anal_dim_in3}
      g\circ\hat\pi^{M_r} = \frac{\chi_{\hat\sigma}}{\lebmeas 2.(\hat\sigma)}\sgn(D_{M_r}(u\circ\ndirprj M_r,,.))\circ\hat\pi^{M_r},
    \end{equation}
    and using the definition of pushforward~\eqref{eq:push_forward_p1} we get:
    \begin{equation}
      \label{eq:anal_dim_in4}
      \av_{\sigma}|D_{M_r}(u\circ\ndirprj
      M_r,,.)|\,d\mu_{M_r}\le\av_{\hat\sigma}\left|
        \hat D(u\circ\hat\pi^\infty)
        \right|\,d\lebmeas 2..
      \end{equation}
      Let $\hat S$ be the union of the $1$-skeleta of $Y_j$ for
      $j\in\zahlen$; then $\hat S$ is $\lebmeas 2.$-null and
      $\hat\pi^\infty$-saturated so that
      $\mu_\infty(\hat\pi^\infty(\hat S))=0$. Now $\hat\pi^\infty$
      restricts to a homeomorphism between $\real^2\setminus\hat S$ and
      $X_\infty\setminus\hat\pi^\infty(\hat S)$; in particular, we can
      define $g\in L^1(\mu_\infty)$ as:
      \begin{equation}
        \label{eq:anal_dim_in5}
        g =
        \begin{cases}
          \frac{\chi_{\hat\pi^\infty(\hat\sigma)}}{\mu_\infty(\hat\pi^\infty(\hat\sigma))}\left(
            \sgn \hat D(u\circ\hat\pi^\infty)
            \right)\circ(\hat\pi^\infty)^{-1}&\text{on
              $X_\infty\setminus\hat\pi^\infty(\hat\sigma)$}\\
            0&\text{elsewhere;}
        \end{cases}
      \end{equation}
      then:
      \begin{equation}
        \label{eq:anal_dim_in6}
        g\circ\hat\pi^\infty=\frac{\chi_{\hat\sigma}}{\lebmeas 2.(\hat\sigma)}\sgn\hat D(u\circ\hat\pi^\infty)
      \end{equation}
      as an element of $L^1(\lebmeas 2.)$; therefore, using the definition
      of pushforward~\eqref{eq:push_forward_p1}, we get:
      \begin{equation}
        \label{eq:anal_dim_in7}
        \av_{\hat\sigma}\left|
        \hat D(u\circ\hat\pi^\infty)
        \right|\,d\lebmeas
        2.\le\av_{\hat\pi^\infty(\hat\sigma)}|D_\infty u|\,d\mu_\infty.
      \end{equation}
      Combining~\eqref{eq:anal_dim_in7}, \eqref{eq:anal_dim_in4},
      \eqref{eq:anal_dim_p8} we get:
      \begin{equation}
        \label{eq:anal_dim_in8}
        \left|
      \av_{\sigma_i} (u\circ\ndirprj M_r,,.)\,d\mu_{M_r} -
      \av_{\sigma_{i+1}} (u\circ\ndirprj M_r,,.)\,d\mu_{M_r}
      \right|
      \lesssim m^{-M} r
      \av_{\hat\pi^\infty(\sigma_i\cup\sigma_{i+1})}|D_\infty u|\,d\mu_\infty,
    \end{equation}
    and using~\eqref{eq:anal_dim_p2} we conclude that:
    \begin{equation}
      \label{eq:anal_dim_in9}
      \left|
      \av_{\sigma_i} (u\circ\ndirprj M_r,,.)\,d\mu_{M_r} -
      \av_{\sigma_{i+1}} (u\circ\ndirprj M_r,,.)\,d\mu_{M_r}
      \right|
      \lesssim m^{-M} r\varepsilon.
    \end{equation}
    Summing over $i$ and using~(\ref{eq:anal_dim_p7}) we obtain:
    \begin{equation}
      \label{eq:anal_dim_p9}
      \left|
      \av_{\ndirprj M_r,\infty,.(\sigma_1)} u\,d\mu_\infty -
      \av_{\ndirprj M_r,\infty,.(\sigma_{L_r})} u\,d\mu_\infty
      \right|=\left|
      \av_{\sigma_1} u\circ \ndirprj M_r,\infty,.\,d\mu_{M_r} -
      \av_{\sigma_{L_r}} u\circ \ndirprj M_r,\infty,.\,d\mu_{M_r}
      \right|\le C\varepsilon r;
    \end{equation}
    using~(\ref{eq:anal_dim_p6}) and that $u$ is Lipschitz, we obtain:
    \begin{equation}
      \label{eq:anal_dim_p10}
      \left|u(p)-u(q)\right|\le 2C\LIP(u) rm^{-M}+\varepsilon Cr,
    \end{equation}
    which gives~(\ref{eq:anal_dim_p4}). 
    
   Finally, the horizontal derivative operator
    $D_\infty$ is associated with the pushforward under $\hat\pi^\infty$ of the family
    $\hat\Gamma$ of horizontal lines of $\real^n$ equipped with the
    obvious measure; now~(\ref{eq:anal_dim_in4}) can be interpreted as
    saying that this family of lines gives a universal Alberti
    representation in the sense of Bate. Note also that the horizontal
    geodesics are really \emph{gradient curves} for $f$, in the sense
    that $f$ decreases/increases along them with \emph{optimal speed} in the
    sense of \cite{deralb}, see also Remark~\ref{rem:minimal_gradient}.
  \end{proof}
\begin{remark}
  \label{rem:anal_dim_weaver}
  There is an alternative approach to the proof that the analytic
  dimension is $1$, which uses directly the measured families of curves
  constructed in Theorem~\ref{thm:meas_curve_join}. This approach is based on Weaver derivations. By
  \cite[Thm.~5.9]{derivdiff} it suffices to show that for each
  Lipschitz function $f$ the inequality:
  \begin{equation}
    \label{eq:der_lip}
    \Lip f\lesssim |D_\infty f|
  \end{equation}
  holds $\mu_\infty$-a.e. Let $p$ be a Lebesgue point for $|D_\infty f|$. Given a
  point $q$, we consider the measured family of curves $\Gamma$ joining $p$ to $q$
  which was constructed in
  Theorem~\ref{thm:meas_curve_join}; by inspecting the construction
  of $\Gamma$ and using estimates like~\eqref{eq:anal_dim_in4}, \eqref{eq:anal_dim_in7}
  one concludes that there is a universal constant $C$ such that:
  \begin{equation}
    \label{eq:curve_bound}
    |f(p)-f(q)|\le C\int_{B(p,Cd(p,q))\cup B(q,Cd(p,q))}|D_\infty f|\,(d\mu_{\infty,p}+d\mu_{\infty,q});
  \end{equation}
recall that $\mu_{\infty,p}(B(p,Cd(p,q)))\approx\mu_{\infty,q}(B(p,Cd(p,q)))\approx
d(p,q)$. As $p$ is a Lebesgue point for $|D_\infty f|$, the fact that $\mu_\infty$
is doubling implies that:
\begin{equation}
  \label{eq:riesz_leb}
  \begin{aligned}
    \lim_{q\to p}\av_{B(p,Cd(p,q))}|D_\infty f|\,d\mu_{\infty,p}&=|D_\infty f|(p)\\
    \lim_{q\to p}\av_{B(q,Cd(p,q))}|D_\infty f|\,d\mu_{\infty,q}&=|D_\infty f|(q).
  \end{aligned}
\end{equation}
Thus, dividing~(\ref{eq:curve_bound}) by $d(p,q)$ and taking the
$\limsup$ as $q\to p$, (\ref{eq:der_lip}) follows.
\end{remark}
\begin{remark}
  \label{rem:minimal_gradient}
  Let $f$ be Lipschitz and $g_f$ denote its generalized minimal upper gradient. By
  \cite[Secs.~5,6]{cheeger} (see \cite{cks_metric_diff} for a simplified
  proof) we know that $g_f=\Lip f$
  $\mu_\infty$-a.e. However, in these examples one can obtain a direct
  argument under the following lines. As we are dealing with a PI
  space we already know that $\Lip f\le Cg_f$; moreover, as the
  Sobolev space $H^{1,p}(\mu_\infty)$ is reflexive for 
  $p>1$, one can always assume that $g_f$ is the $L^p$-limit of a
  sequence $g_k$ where $g_k$ is an upper gradient of a Lipschitz
  function $f_k$, and where $f_k\to f$ in $H^{1,p}(\mu_\infty)$ and $\Lip f_k\to \Lip
  f$ $\mu$-a.e. It thus suffices to show that
  whenever $g$ is an upper gradient for $f$, one has $g\ge \Lip f$
  $\mu$-a.e. Without loss of generality we can assume $g\in
  L^1_{\text{loc}}(\mu_\infty)$.  We secondly observe that, as $x_\infty$ is a chart
  function, we have $\Lip f= |D_\infty f|$ $\mu$-a.e. Now the Fubini representation of the measure $\lebmeas
  2.$ on $\real^2$ descends to a similar representation of the measure $\mu_\infty$
  in terms of the horizontal geodesics of $X_\infty$. If $p$ is a
  Lebesgue point for $g$ and $|D_\infty f|$, for any $\varepsilon>0$ we can find a
  horizontal line $\gamma$ with $\gamma(0)$ a Lebesgue point for
  $g\circ\gamma$ and $D_\infty f\circ\gamma$, $d(\gamma(0),p)\le\varepsilon$, and
  $g(\gamma(0))\in[g(p)-\varepsilon,g(p)+\varepsilon]$, and
  $|D_\infty f|(\gamma(0))\in[|D_\infty f|(\gamma(0))-\varepsilon,
  |D_\infty f|(\gamma(0))+\varepsilon]$. Applying Lebesgue differentation to
  $g\circ\gamma$ and $D_\infty f\circ\gamma$ at $0$ one gets
  $g(\gamma(0))\ge |D_\infty f|(\gamma(0))$.
\end{remark}

\section{The $n=2$ case
of Theorem \ref{thm_top_dim_n_anal_dim_1} concluded}
\label{sec_n_equals_2_case_concluded}
We now verify (\ref{item_ditto_for_limits}) 
of Theorem \ref{thm_top_dim_n_anal_dim_1} in the $n=2$
case.

Let $\{(\la_kX_\infty,p_k)\}$ and $(Z,z)$ be as in the statement of 
Theorem \ref{thm_top_dim_n_anal_dim_1}.   For every $k$, the map
$\hat\pi^\infty:(\R^n,\la_kd_\al)\ra X_\infty$ is a rescaling of a
David-Semmes regular map, so it is a David-Semmes regular map with
uniform constants.  Therefore there is an $N$ such that for every
$k$, there exists points $\hat p_{k,1},\ldots,\hat p_{k,\ell}$
($\ell\le N$),
such that $(\hat\pi^\infty)^{-1}(B(p_k,r))\subset \cup_i B(p_{k,i},Cr)$
for all $r\in (0,\infty)$.  After passing to a subsequence we may assume
that $\ell$ is constant, and that for each $1\leq i\leq \ell$, the pointed maps
$\hat\pi^\infty: (\R^n,\la_k\hat d_\infty^Y,\hat p_{k,i})\ra (\la_kX_\infty,p_k)$
Gromov-Hausdorff converge as $k\ra\infty$ to a Lipschitz map
$\phi_i:(W_i,w_i)\ra (Z,z)$.  Then the $\phi_i$'s are David-Semmes regular,
and $\cup_i\im(\phi_i)=Z$.

As $\phi_i$ is a light map,  the topological dimension of $(Z,z)$ is at least
$n$.  Note that the properties of satisfying a $(1,1)$-Poincar\'e inequality, 
being Ahlfors $Q$-regular, and having Assouad-Nagata dimension $\leq n$, with uniform constants,
all pass to Gromov-Hausdorff limits.  Therefore (2) and (3) hold.

Part  (\ref{item_analytic_dimension_1})  can be verified in several 
different ways.  

One approach is to implement either of the arguments of
Section~\ref{sec:anal_dim} by passing the ingredients --- the
horizontal galleries of Corollary~\ref{cor:mono_galleries}, the
measured families of curves $\Gamma$
of Theorem~\ref{thm:meas_curve_join} and the derivation $D_\infty$ --- to
the limit space $Z$. More specifically, in adapting Remark~\ref{rem:anal_dim_weaver},
one can associate to $\Gamma$
and $D_\infty$ normal $1$-currents in the sense of Lang and use
compactness of normal currents~\cite[Thm.~5.4]{lang_loc_curr}.

A second approach is to exploit the symmetry of $X_\infty$.  
Using the self-similarity of $X_\infty$
induced by the affine transformation
$\Phi:\R^2\ra \R^2$ we may assume, without loss of
generality, that the scale factor $\la_k$ is $1$ for all $k$. 
Also, note that the action
$\Z^2\acts\R^2$ preserves the equivalence relation $\calr_\infty$
on the open set $\{p\in\R^2\mid x(p)\not\in\Z\}$, and induces
an action on $\{p\in X_\infty\mid x_\infty(p)\not\in\Z\}$ which is a local
isometry.  In particular, for every point $p\in X_\infty$ with
$x_\infty(p)\not\in\Z$, if $r=\dist(x_\infty(p),\Z)$, then
the ball $B(p,r)\subset X_\infty$ is 
measure-preserving
isometric to the ball $B(\hat\pi^\infty(\hat q),r)\subset X_\infty$ for some 
 $\hat q\in (0,1)^2$.
This property passes to the Gromov-Hausdorff limit, allowing one to see that
$\h^Q$-a.e. point of $Z$ lies in a ball that is isometric to a ball in $X_\infty$
itself, and in particular the analytic dimension is $1$.

A third approach involves rescaling the direct system of cell complexes
$\{X_j\}$ and passing to a pointed limit, which is another pointed 
direct system with similar properties.  See Section 11 --- especially
Subsection \ref{subsec_limits_admissible}
 --- for more details.

\section{The $n>2$ case of Theorem \ref{thm_top_dim_n_anal_dim_1}}
\label{sec_n_dim_case}

In this section we will discuss the higher-dimensional version of the 
examples treated in Sections 
\ref{sec_cell_structure}-\ref{sec_n_equals_2_case}.

Pick $n>2$.  We will imitate the $n=2$ construction, but where the last
$(n-1)$ coordinates of $\R^n$ will play the r\^ole of the $y$-coordinate in
$\R^2$.  Thus the notation $(x,y)$ will henceforth mean that $y\in\R^{n-1}$.
The words ``horizontal'' and ``vertical'' will have obvious interpretations
in this new setup.

Let $m=1+3(n-1)$, and pick $m_v\gg m$.  Let $\Phi:\R^n\ra \R^n$ be the 
linear transformation $\Phi(x,y)=(m^{-1}x,m_v^{-1}y)$.  Let
$Y_0$ be the cell structure on  $\R^n$ coming from the tiling
by unit cubes, and for $j\in\Z$ let $Y_j$  be the image of $Y_0$ under $\Phi^j$.

Following the $n=2$ construction,
we  define an equivalence relation $\calr$ on $\R^n$ that is generated by the
identifications of certain pairs of $(n-1)$-cells of $Y_1$ by vertical translation.  
For each
$k\in\Z$,  $d\in\{2,\ldots,n\}$, we perform gluings within the three
vertical hyperplanes
$$
\{(x_1,\ldots,x_n)\in\R^n\mid x_1=k+m^{-1}(1+3(d-2)+(i-1))\}_{i\in\{1,2,3\}}\,,
$$ 
so as to enable horizontal galleries to ``jump'' in the 
$x_d$-coordinate direction.  To define these gluings, for each $k,\ell\in \Z$, 
$2\leq d\leq n$, $1\leq i\leq 3$, we let

 \begin{equation}
a_{k,\ell,d,i}=
    \left\{(x_1,\ldots,x_n)\;\;\middle|\;\;
    \begin{aligned}
    x_1=m^{-1}(1+3(d-2)+(i-1))\,,\;\\
    x_d\in [(3\ell+i-1)m_v^{-1},(3\ell+i)m_v^{-1}]
    \end{aligned}
    \right\},
  \end{equation}

 \begin{equation}
a_{k,\ell,d,i}'=
    \left\{(x_1,\ldots,x_n)\;\;\middle|\;\;
    \begin{aligned}
    x_1=m^{-1}(1+3(d-2)+(i-1))\,,\;\\
    x_d\in [(3\ell+i)m_v^{-1},(3\ell+i+1)m_v^{-1}]
    \end{aligned}
    \right\},
  \end{equation}
and we identify $a_{k,\ell,d,i}$ with $a_{k,\ell,d,i}'$ by the vertical translation
$p\mapsto p+m_v^{-1}e_d$.

The definition of $\calr_\infty$  and the
pseododistance $\hat d_\infty$ on the quotient $X_\infty$ remain the same as
before.  Then $(X_\infty,\hat d_\infty)$ is the metric space of Theorem  
\ref{thm_top_dim_n_anal_dim_1} for general $n$.  

The verification of the
assertions in Theorem~\ref{thm_top_dim_n_anal_dim_1} for $(X_\infty,\hat d_\infty)$ proceeds along the 
same steps, with appropriate modifications, a few of which we indicate here:

\begin{itemize}
\item A gallery is a chain of cells where two consecutive cells share an
$(n-1)$-face. 
\item (cf. Lemma \ref{lem_pi_j_structure}(\ref{item_distance_2}))
If $p_0,p_1\in X_j$ are distinct points with
$\pi_j(p_0)=\pi_j(p_1)$, then $p_i=\hat\pi^j(\hat p_i)$ for a
unique point $\hat p_i\in\R^n$, and the pair $\hat p_0, \hat p_1$
is contained in $\ell\cap (\hat\si_0\cup\hat\si_1)$,
where $\ell\subset\R^n$ is vertical line $\ell\subset \R^n$ parallel to one
the $d^{th}$-coordinate axis for some $2\leq d\leq n$, 
and $\hat\si_0,\hat\si_1$ is a vertical gallery the $(n-1)$-skeleton
of  $Y_{j+1}$ (i.e. a pair of vertical $(n-1)$-cells that share an $(n-2)$-cell).
\item If $\si$ is an $n$-cell of $X_j$ (or $Y_j$), then it has two vertical
$(n-1)$-faces $\tau_0, \tau_1$, where $x_j(\tau_1)=x_j(\tau_0)+m^{-j}$; these
replace the vertical $1$-cells in the $n=2$ case.
\item The cosets of 
$\calr_j$ lie in orbits of $m^{-j}\Z\times m_v^{-j}\Z^{n-1}$, so the
maps $\hat y_j$ and $y_j$ take values in $\R^{n-1}/m_v^{-j}\Z^{n-1}$ instead of 
$\R/m_v^{-j}\Z$.
\item For the proof that $(X_\infty,\hat d_\infty)$ has Assouad-Nagata dimension $\leq n$,
one builds $(n+1)$ ``good families'' of subsets by working with cells of dimension
$0$ through $n$.
\end{itemize}
In addition to the minor points above, there are changes to the proof of 
Proposition \ref{prop_distance_lower_bound} that require more care.
In Steps 1 and 2 of the proof,  rather than using
the distance to $0\in S^1(m_v^{-j})$, one uses the distance to the
$(n-2)$-skeleton in the torus $T^{n-1}(m_v^{-j})=\R^{n-1}/m_v^{-j}\Z^{n-1}$.  
However, the conclusion of
Step 2 only says that $\pi_i^{-1}(\bar\si_{i_3})$ lies at controlled distance
from the $(n-2)$-skeleton of $T^{n-1}(m_v^{-j})$.  To proceed, one has to 
inductively exclude
the possibility that $\pi_i^{-1}(\bar\si_{i_3})$ lies far from the $(n-3)$-skeleton
of $T^{n-1}(m_v^{-j})$, etc.   This is done using a variation on Steps 1 and 2, 
where one uses the distance from the $(n-k)$-skeleton of $T^{n-1}(m_v^{-j})$, 
for $k\geq 3$, and a variation of Lemma \ref{lem_edge_star}.

\section{Generalizations}
\label{sec_generalizations}

In this section we consider generalizations of the examples given in previous 
sections.  It is not our intention to be exhaustive --- it is clear that 
one can further generalize the contruction in different ways.
Our main purpose is  
to illustrate that the same overall scheme of proof applies to a much broader 
class of examples, and to clarify the logical structure of the proofs by 
identifying the essential properties needed.

\subsection{Admissible systems}
Before giving the precise definition, we begin with some observations about the 
examples discussed in earlier sections.  These are direct limits of direct
systems
$$
\ldots\stackrel{\pi_{-1}}{\lra}
X_0\stackrel{\pi_0}{\lra}X_1\stackrel{\pi_1}{\lra}\ldots
\stackrel{\pi_{j-1}}{\lra}X_j\stackrel{\pi_j}{\lra}\ldots
$$
where the $X_j$'s are cell complexes.  Crucial to the analysis is the distinction
between horizontal and vertical directions, and the fact that the first coordinate
function on $\R^n$ descends to a compatible family of functions.   We will
axiomatize this by requiring cells to have distinguished characteristic maps
that induce the horizontal/vertical structure.

Fix integers $n \geq 2$  and $2\leq m\leq m_v$.  Let  
$\{Y_j\}_{j\in\Z}$ be a family of tilings
of $\R^n$, where $Y_j$ is a tiling by translates of the parallelopiped 
$[0,m^{-j}]\times [0,m_v^{-j}]^{n-1}$, and for all $j\in \Z$, the tiling
$Y_{j+1}$ is subdivision of $Y_j$.
Let $G$ be the group of isometries $g:\R^n\ra \R^n$ of the form
$g=\id_{\R}\times h$ for some isometry $h:\R^{n-1}\ra \R^{n-1}$.   
Let $G_j\subset G$ be the subgroup that preserves
the cell structure of $Y_j$; hence for all $g\in G_j$, the linear part
$L(g)$ preserves the set of coordinate vectors $\{\pm e_i\}_{2\leq i\leq n}$.

\begin{definition}
\label{def_admissible_system}
An {\bf admissible direct system} is a tuple
consisting of:
\begin{itemize}
\item A direct system of cell complexes indexed by
the integers
$$
\ldots\stackrel{\pi_{-1}}{\lra}X_0\stackrel{\pi_0}{\lra}\ldots \ldots
\stackrel{\pi_{j-1}}{\lra} X_j\stackrel{\pi_j}{\lra}
\ldots
$$
\item
A  collection $\{x_j:X_j\ra \R\}$ of continuous maps.
\item For each cell $\si$ of $X_j$, a collection $\Phi_\si$
of distinguished characteristic
maps $\phi:\hat\si\ra \si$, where $\hat\si$ is a cell of $Y_j$. 
\item A Radon measure $\mu_j$ on $X_j$ for all $j\in \Z$.
\end{itemize}
The tuple is required to satisfy the following conditions for some 
constants $\De$, $H$:

\begin{description}
\item[(Ax1)] $X_j$ is connected, is a union of its closed $n$-cells, and all links have cardinality at most $\De$.
\item[(Ax2)] Compatibility of distinguished characteristic maps:
	\begin{description}
	\item[(Ax2a)] (Compatibility with $G$) For any cell $\si$ of $X_j$, any
	two elements $\phi_0:\hat\si_0\ra \si$, $\phi_1:\hat\si_1\ra \si$ of 
	$\Phi_\si$ agree up to 
	precomposition with some $g\in G_j$ 	
	such that
	$g(\hat\si_0)= \hat\si_1$.
	In particular, the notions of vertical and horizontal cells in $Y_j$ 
	descend to well-defined notions for cells in $X_j$.	
	\item[(Ax2b)] (Compatibility with face restrictions) The restriction of any 
	distinguished characteristic map $\phi:\hat\si\ra \si$ to a face of 
	$\hat\si$ is a distinguished characteristic map.
	\item[(Ax2c)] (Compatibility with subdivision)  If 
	$\si$ is a cell of $X_j$
	and $\phi:\hat\si\ra \si$
	belongs to $\Phi_\si$, then for any cell $\hat\si'$ of $Y_{j+1}$ 
	contained in $\hat\si$, the composition $\pi_j\circ\phi\restr_{\hat\si'}:\hat\si'\ra 
	\si'\subset X_{j+1}$ is a distinguished characteristic map of some cell
	$\si'$ of $X_{j+1}$.  Moreover all distinguished
	characteristic maps of $X_{j+1}$ may be obtained in this way.
	\item[(Ax2d)] (Compatibility with $x_j$) For every cell $\si$ of $X_j$ and
	every $\phi:\hat\si\ra\si$ in $\Phi_\si$, $x_j\circ\phi=x$, where $x:\R^n\ra\R$
	is the first coordinate function.        
	\end{description}
\item[(Ax3)] $\pi_j:X_j^{(1)}\ra X_{j+1}$ is a surjective  cellular map, where
$X_j^{(1)}$ denotes the subdivision of $X_j$ defined by  restricting
distinguished characteristic maps $\phi:\hat\si\ra\si$ to cells $\hat\si'\in Y_{j+1}$.
\item[(Ax4)] (Fibers have controlled vertical diameter)
For every $j$, and every 
$0$-cell $\bar v$ of $X_{j+1}$, any two elements $v_0,v_1$
of the inverse image $\pi_j^{-1}(\bar v)$  are contained in a vertical edge path
of $X_j^{(1)}$ of combinatorial length at most $H$.
\item[(Ax5)]
\label{item_gallery_accessibility}
 (Gallery accessibility)
For every $j$, and any two adjacent
$n$-cells $\si_0, \si_1$ of $X_j$, there is a
horizontal gallery from $\si_0$ to $\si_1$ of combinatorial length
at most $\De$.
\item[(Ax6)] Compatibility of measures:
\label{item_measure_compatibility}
	\begin{description}
	\item [(Ax6a)]
	\label{item_pushforward_compatible}
	$(\pi_j)_{\#}\mu_j=\mu_{j+1}$ for all $j\in\Z$.
	\item[(Ax6b)] (Lebesgue measure on cells)
	For every $j$, every $n$-cell $\si$ of $X_j$, and every $\phi_\si\in\Phi_\si$,
	the restriction $\mu_j\on\si$ agrees with the pushforward of
        the Lebesgue measure,
	$(\phi_\si)_{\#}\lebmeas n.$, up to a constant $w_\si$.
	\item[(Ax6c)] (Doubling) For any two  adjacent $n$-cells $\si,\si$ of $X_j$, the weights
	$w_\si$, $w_{\si'}$ agree to within a factor of at most $\De$.
	\end{description}
      \end{description}
\end{definition}

\begin{remark}
As with the examples from Sections \ref{sec_overview} and \ref{sec_n_dim_case},
one way of constructing an admissible direct system is by defining an increasing sequence
$\ldots\subset\calr_{j-1}\subset\calr_j\subset\calr_{j+1}\subset\ldots$
of equivalence relations on $\R^n$, such that $\calr_j$ respects the cell structure
on $Y_j$ and the quotients $X_j=\R^n/\calr_j$ satisfy
Definition \ref{def_admissible_system}.
\end{remark}

For any admissible system
we define a pseododistance $\hat d_\infty$ on the direct limit
$X_\infty$ as follows. We let $\hat d_\infty$ be the largest pseudodistance on $X_\infty$ such that, for every $j\in \Z$
and every cell $\hat\si$ of $X_j$, the projection to $X_\infty$ 
has $\hat d_\infty$-diameter at most $m^{-j}$.  
In general this need not define a metric (see the next example), so we
form a metric space 
$(\bar X_\infty,\bar d_\infty)$, the \textbf{limit space} of the
admissible system, by collapsing
subsets of $X_\infty$ of zero $\hat d_\infty$-diameter to points.
We denote by $\nbdirprj j,,.$ the composition
$X_j\stackrel{\pi_j^\infty}{\ra}X_\infty\ra \bar X_\infty$.
The pushforward of $\mu_j$ under $\nbdirprj j,,.$
is independent of $j$ by Axiom \textbf{(Ax6)}, and it thus defines the natural 
measure $\mu_\infty$ on $\bar X_\infty$. 
\begin{example}
To illustrate how the pseudodistance $\hat d_\infty$ can fail to be a distance,
we construct a direct system $\{X_j\}$ that satisfies all the conditions of 
Definition \ref{def_admissible_system} except \textbf{(Ax5)}, 
such that the psuedodistance
$\hat d_\infty$ is not a distance.  It is not hard to modify this example to obtain
an admissible system with the same property.

Let $n=2$, $m=3$, $m_v=6$.  For every $k,\ell\in\Z^2$, we define a pair
of  $2$-cells  
$$
a_{k,\ell}=\left[\frac{1}{3}+k,\frac{2}{3}+k\right]\times
\left[\frac{1}{6}+\ell,\frac{2}{6}+\ell\right]\,,
$$
$$
a_{k,\ell}'=\left[\frac{1}{3}+k,\frac{2}{3}+k\right]\times
\left[\frac{4}{6}+\ell,\frac{5}{6}+\ell\right]\,.
$$
Now let $\calr$ be the equivalence relation on $\R^2$ obtained by identifying $a_{k,\ell}$
with $a_{k,\ell}'$ by the vertical translation $p\mapsto p+\frac12 e_2$, for every
$k,\ell$.  We now define $\calr_j$, $X_j$, $\hat\pi^j$, etc.~as in 
Section \ref{sec_cell_structure}.
 Then $\{X_j=\R^2/\calr_j\}$
is the desired direct system.  Note that there is a Cantor set in the vertical line
$\{x=\frac{1}{2}\}$ that maps to a subset of $X_\infty$ with zero $\hat d_\infty$-diameter.
On the other hand if $\hat p\in\R^2$, $p=\hat\pi^\infty(p)$, then
the set $(\hat\pi^\infty)^{-1}(p)$ is countable, being a countable union of the sets
$(\hat\pi^j)^{-1}(\hat\pi^j(\hat p))$, each of which is finite.

\end{example}

Most  of Theorem \ref{thm_top_dim_n_anal_dim_1} generalizes
verbatim to admissible
systems:

\begin{theorem}
\label{thm_admissible_system}
For every $n,m,H$, there is an $\ul{m_v}=\ul{m_v}(n,m,H)$ such that if
$m_v\geq \ul{m_v}$, then:
\begin{enumerate}
\item $(\bar X_\infty,\bar d_\infty,\mu_\infty)$ is a complete
doubling metric measure 
space satisfying
a $(1,1)$-Poincar\'e inequality.
\item $\bar X_\infty$ has topological and Assouad-Nagata dimension $n$.
\item $(\bar X_\infty,\bar d_\infty,\mu_\infty)$ has analytic dimension $1$.
\item (1)-(3) also hold for any  pointed measured Gromov-Hausdorff  limit
of any sequence of rescalings of $(\bar X_\infty,\bar d_\infty,\mu_\infty)$.
\end{enumerate}
\end{theorem}

Part (1) of Theorem \ref{thm_top_dim_n_anal_dim_1} does not generalize
directly to admissible systems, because, in particular, 
$(\bar X_\infty,\bar d_\infty)$ is not Ahlfors-regular in general.
To formulate a modified statement, we introduce another metric
on the $X_j$'s, which plays the role of the metric $d_\al$ on $\R^n$.

\begin{definition}
For all $j\in\Z$, let $d_j$ be the 
largest pseudodistance on $X_j$ such that $d_j\leq \hat d_j$, 
and every cell of $X_j^{(k)}$ has $d_j$-diameter at most $m^{-(j+k)}$. 
\end{definition}

\begin{theorem}
\label{thm_lipschitz_light}
For every $j\in\Z$:
\begin{enumerate}
\item $(X_j,d_j,\mu_j)$ is a doubling metric measure space.
\item $\bar\pi_j^\infty:(X_j,d_j)\ra (\bar X_\infty,\bar d_\infty)$
is a Lipschitz light map (see Definition \ref{defn:lipschitz-light}); 
in particular, the point inverses
of $\bar\pi_j^\infty$ are uniformly totally disconnected (i.e.~have
Assouad-Nagata dimension $0$).
\item Let $Q=1+\frac{(n-1)\log m_v}{\log m}$. If for every $j$, the inverse image under $\pi_j:X_j\ra X_{j+1}$
of every open $n$-cell of $X_{j+1}$ is a single open $n$-cell of $X_j^{(1)}$,
then $(X_j,d_j)$ is Ahlfors $Q$-regular, and 
$\bar\pi_j^\infty:(X_j,d_j)\ra (\bar X_\infty,\bar d_\infty)$ is a David-Semmes
regular map.
\end{enumerate}
\end{theorem}
We recall here the notion of Lipschitz light map~\cite[Def.~1.14]{cheeger_inverse_l1}.
  \begin{defn}
    \label{defn:lipschitz-light}
    A Lipschitz map $f:X\to Y$ is \textbf{Lipschitz light} if there is
    some $C>0$ such that for every bounded subset $W\subset Y$ the
    $\diam(W)$-components of $f^{-1}(W)$ have diameter at most
    $C\diam(W)$. Note that a Lipschitz light map is also a continuous
    light map.
  \end{defn}

\subsection{The proof of
  Theorem~\ref{thm_admissible_system}}
Overall, the proof of parts (1)-(3) of
Theorem~\ref{thm_admissible_system} follows
closely that of Theorem~\ref{thm_top_dim_n_anal_dim_1}. However,
some modifications are needed and these are explained in this
Subsection.
In Subsection \ref{subsec_limits_admissible} we 
give the proof of item (4), i.e.~the stability under taking weak
tangents.  
Note that even for the specific
examples considered in the previous sections a cleaner treatment is
obtained in the more general context of admissible systems.

The proof of
Proposition~\ref{prop_distance_lower_bound} remains valid using a map
$y_j$ which takes values in the quotient $\tilde T_j$ of $\R^{n-1}\simeq \{0\}\times\R^{n-1}
\subset\R^n$ by the action of
$G_j$, which is a quotient of the torus 
$T^{n-1}=\R^{n-1}/m_v^{-j}\Z^{n-1}$ by the finite group of orthogonal transformations
of $\R^{n-1}$ that preserves the subset $\{\pm e_d\}_{2\leq d \leq
  n}$. Note that in the case $n\ge 3$ one must also modify \emph{Steps
1 and 2} of that argument by using the distances from the
$\{0,1,\cdots,n-2\}$-skeleta of $\tilde T_j$ (compare
Section~\ref{sec_n_dim_case}).
 
  The proof that  the
Assouad-Nagata dimension of $\bar X_\infty$ is at most $n$ can be
carried out as in Theorem~\ref{thm:nagata_bound} by working with the
cells of $X_j$ of dimensions $1,\cdots,n$.  The fact that 
the topological dimension (and hence the
Assouad-Nagata dimension) is at least $n$ can be deduced by adapting
either of the arguments in 
Lemma \ref{lem_david_semmes_regular_map}.

 To prove the Poincar\'e inequality, one can
essentially follow the argument of Theorem~\ref{thm:poinceq}.

We now turn to the proof of assertion~(3) in
Theorem~\ref{thm_admissible_system}, i.e.~that $(\bar X_\infty,\bar
d_\infty,\mu_\infty)$ has analytic dimension $1$.
We first define a direct system of derivations 
$\{D_j:\Lip_b(X_j,d_j)\ra L^\infty(X_j,\mu_j)\}$
as follows.
For every  bounded $d_j$-Lipschitz function $u_j:X_j\ra \R$ we let
$D_ju_j$  be the function in  $L^\infty(\mu_j)$ such that, for every 
$n$-cell $\si$ of $X_j$, we have  
$$
D_j(u_j\circ\phi_\si)=\frac{\D(u_j\circ\phi_\si)}{\D x}
$$
for every distinguished characteristic map $\phi_\si\in\Phi_\si$; this is well-defined
because of \textbf{(Ax2a)}.  Note that the family $\{D_j\}$ is compatible
with the projections $\{\pi_j^k\}$, and by pushforward we get a well-defined derivation
$D_\infty:\Lip_b(\bar X_\infty)\ra L^\infty(\bar X_\infty,\mu_\infty)$.
With this setup, the first part of the proof 
of Theorem~\ref{thm:anal_dim} (or using Weaver
derivations as sketched in Remark~\ref{rem:anal_dim_weaver})
carries over.
The justification after (\ref{eq:anal_dim_p8}) requires modifications for 
general admissible systems, however, as it is based on the fact that for
$\mu_\infty$-a.e. point, the point inverse $(\pi_j^\infty)^{-1}(p)$ contains
only one point.  Instead, one may use the fact that for a Lipschitz
function $f:\bar X_\infty\ra \R$, the function $D_j(f\circ\bar\pi_j^\infty)$
is a.e. constant on the fibers of $\bar\pi_j^\infty$, which is 
established in  Lemmas~\ref{lem:generalized_metric_structure}, \ref{lem_compatible_families},
and \ref{lem:gen_anal_dim}.

\begin{lemma}
  \label{lem:generalized_metric_structure}
Fix $p_\infty\in\bar X_\infty$ and $k\ge j$; let
$S=\{\sigma_0,\cdots,\sigma_L\}$ be a chain of $n$-cells of
$X_j^{(k-j)}$ such that each cell of $S$ intersects $\nbdirprj
j,,a.(p_\infty)$. Assume that:
\begin{equation}
  \label{eq:generalized_metric_structure_s3}
  p_j\in \sigma_0\cap\nbdirprj j,,a.(p_\infty);
\end{equation}
if $m_v$ is sufficiently large there is a universal constant $C_0$
such that:
\begin{equation}
  \label{eq:generalized_metric_structure_s4}
  d_j(p_j,\sigma_L)\le 3C_0m^{-k}.
\end{equation}
Moreover, let $k\ge j$ and denote by $\sigma_k$ the $n$-cell of
$X_k^{(1)}$ containing $\pi_j^k(p_j)$. Assume that any vertical
gallery in $X_k^{(1)}$ from $\sigma_k$ to the horizontal
$(n-1)$-skeleton of $X_k$ has length at least $10H$. Then $\nbdirprj
k,,a.(p_\infty)$ is entirely contained in the open $n$-cell of $X_k$
that contains $\pi_j^k(p_j)$.
\end{lemma}
\begin{proof}
 Let $S=\{\sigma_0,\cdots,\sigma_L\}$ be a chain of $n$-cells of
  $X_j^{(k-j)}$ such that each cell of $S$ intersects $\nbdirprj
  j,,a.(p_\infty)$. Fix $p_j\in \sigma_0\cap\nbdirprj
  j,,a.(p_\infty)$. By
  Corollary \ref{cor:cell_dist}, $\ndirprj j,k,.(S)$
  lies in the star of a cell of $X_k$; by Axiom \textbf{(Ax4)} we
  conclude that
  $\ndirprj j,k-1,.(S)$ lies in a combinatorial ball of
  radius $C_0$ in $X_{k-1}^{(1)}$ which is centered on an $n$-cell of
  $X_{k-1}^{(1)}$. Consider now an $n$-cell $\sigma$ of
  $X_j^{(k-1-j)}$; by Axiom \textbf{(Ax2c)} the map $\ndirprj j,k-1,.:X_j^{(k-j)}\to
  X_{k-1}^{(1)}$ is a combinatorial isomorphism when restricted to
  $\sigma$. Thus, for $0\le i\le L$ we have:
  \begin{equation}
    \label{eq:lipschitz_light_p7}
    \begin{aligned}
      d_{\real}(x_j(\sigma_0),x_j(\sigma_i))&\le 2m^{-k}\\
      d_{\tilde T_j}(y_j(\sigma_0),y_j(\sigma_i))&\le 2C_0m_v^{-k};
    \end{aligned}
  \end{equation}
  therefore, if $m_v$ is sufficiently large compared to $C_0$ and $m$
  we must have:
  \begin{equation}
    \label{eq:lipschitz_light_p8}
    d_j(\sigma_0,\sigma_i)\le 3C_0m^{-k}.
  \end{equation}
  \par Let now $\sigma_k$ denote the $n$-cell of $X_k^{(1)}$
  containing $\pi_j^k(p_j)$ and assume that any vertical gallery in
  $X_k^{(1)}$ from $\sigma_k$ to the horizontal $(n-1)$-skeleton of
  $X_k^{(1)}$ has length at least $10H$. Let $q\in\nbdirprj
  k,,a.(p_\infty)$ and choose a cell $\sigma'_k$ of $X_k^{(1)}$ which
  contains $q$ and intersects $\ndirprj k,,a.(p_\infty)$. As both
  cells $\pi_k(\sigma_k)$ and $\pi_k(\sigma'_k)$ meet $\nbdirprj
  k+1,,a.(p_\infty)$, by Corollary \ref{cor:cell_dist}
  they must be adjacent. However, as $\sigma'_k$ intersects $\ndirprj
  k,,a.(p_\infty)$, by Axiom \textbf{(Ax4)} there is a vertical
  gallery in $X_k^{(1)}$ of length at most $H$ from $\sigma_k$ to a
  cell $\tau$ adjacent to $\sigma'_k$. Thus $q$ lies in the open cell
  of $X_k$ containing $\pi_j^k(p_j)$.
\end{proof}

\begin{lemma}
\label{lem_compatible_families}
Let  $\{g_j:X_j\ra \R\}_{j\in\Z}$ be a family of bounded measurable functions,
and suppose that $\{g_j\}$ 
is compatible with projection in the sense that, for all $j\leq k$,
one has $g_k\circ\pi_j^k=g_j$   $\mu_j$-a.e. Moreover, let $g_\infty\in L^\infty(\bar
  X_\infty,\mu_\infty)$ be the function satisfying:
  \begin{equation}
    \label{eq:compatible_families_s1}
    \nbdirprj {j\#},,.(g_j\mu_j)=g_\infty\mu_\infty.
  \end{equation}
  Then for every $j\in \Z$, 
and $\mu_\infty$-a.e.~$p\in \bar X_\infty$, if $\tilde\mu_j(p)$ is the disintegration of
$\mu_j$  with respect to $\bar\pi_j^\infty$, then $g_j$ is
$\tilde\mu_j(p)$-a.e. equal to $g_\infty(p)$ on 
 $(\bar\pi_j^\infty)^{-1}(p)$.
\end{lemma}
\begin{proof}
For $\eta\in[1,\lfloor\frac{m_v}{3H}\rfloor]$ and $k\ge j$ let
$S_{\eta,k}$ be the set of points $q\in X_j$ such that:
\begin{enumerate}
\renewcommand{\labelenumi}{(\alph{enumi})}
\item $q$ does not belong to the $(n-1)$-skeleton of $X_j^{(k-j)}$ for any 
$k\geq j$.
\item 
\label{item_far_from_boundary} If $\si$ is the
$n$-cell of $X_k^{(1)}$ containing  $\pi_j^k(q)$, then any  vertical
gallery in $X_k^{(1)}$ from $\si$ to the horizontal $(n-1)$-skeleton of $X_k$
has length at least  $\eta H$.
\end{enumerate}
If $q\in S_{\eta,k}$ and $\eta\ge 10$, by
Lemma~\ref{lem:generalized_metric_structure} the fibre $\nbdirprj
k,,a.(\bar\pi_j^\infty(q)) $ is entirely contained in the open
$n$-cell of $X_k$ which contains $\pi_j^k(q)$.
\par We now look at the set of points $S$ in $X_j$ such
that~(b) occurs infinitely often:
\begin{equation}
  \label{eq:compatible_families_p1}
  S_{\eta}=\bigcap_{k\ge 1}\bigcup_{k'\ge k} S_{\eta,k'};
\end{equation}
then $\mu_j(S^c_\eta)=0$. Consider now the subset
$\Omega\subset\bar X_\infty$:
\begin{equation}
  \label{eq:compatible_families_p2}
  \Omega=\left\{
    p\in\bar X_\infty: \nbdirprj j,,a.(p)\cap S^c_{10}\ne\emptyset
    \right\};
  \end{equation}
  then by Lemma~\ref{lem:generalized_metric_structure} $\nbdirprj
  j,,a.(\Omega)\subset S_9^c$ and so $\mu_\infty(\Omega)=0$. Thus for
  $\mu_\infty$-a.e.~$p\in\bar X_\infty$ we can assume that $\nbdirprj
  j,,a.(p)\subset S_{10}$.
  \par As the Lebesgue Differentiation Theorem holds in $X_j$, for
  $\mu_\infty$-a.e.~$p\in\bar X_\infty$ there is a subset
  $T_p\subset(\bar\pi_j^\infty)^{-1}(p)$ of full $\tilde\mu_j(p)$ measure
such that  $T_p\subset  S_{10}$ and every $q\in T_p$ is an approximate
continuity point of $g_j$.  Let $q,q'\in T_p$, and $\eps>0$.  For any
$k_0\geq j$ we can find $k\ge k_0$ such that $q\in S_{10,k}$; let $\hat\si$, $\hat\si'$ be the
$n$-cells of $X_j^{(k-j)}$ which contain $q$, $q'$ respectively.  Provided
$k_0$ is sufficiently large, we will have
\begin{equation}
  \label{eq:compatible_families_p3}
\max\left(
\left|\av_{\hat\si} g_j\;d\mu_j-g_j(q)\right|, \left|\av_{\hat\si'} g_j\;d\mu_j-g_j(q)\right|
\right)<\eps\,.
\end{equation}
As $q\in S_{10,k}$, we have
$\pi_j^k(\hat\si)=\pi_j^k(\hat\si')$, and then
\begin{equation}
  \label{eq:compatible_families_p4}
\av_{\hat\si}g_j\;d\mu_j=\av_{\pi_j^k(\hat\si)}g_k\;d\mu_k
=\av_{\pi_j^k(\hat\si')}g_k\;d\mu_k=\av_{\hat\si'}g_j\;d\mu_j\,.
\end{equation}
Therefore $|g_j(q)-g_j(q')|<2\eps$, and as $\eps$ was arbitrary, we
have $g_j(q)=g_j(q')$. Now the lemma follows
because~(\ref{eq:compatible_families_s1})
and the Disintegration Theorem imply
that for $\mu_\infty$-a.e.~$p\in\bar
X_\infty$ $g_\infty(p)$ is the $\tilde\mu_j(p)$ average of $g_j$.
\end{proof}
\begin{lemma}
  \label{lem:gen_anal_dim}
For $j\in \Z$ and $p\in \bar X_\infty$, we let $\tilde\mu_j(p)$ denote the disintegration of 
$\mu_j$ with respect to $\bar\pi_j^\infty$.
Let $u:\bar X_\infty\ra \R$ be Lipschitz, and $u_j=u\circ
\bar\pi_j^\infty$. Let $D_\infty u\in L^\infty(\bar
X_\infty,\mu_\infty)$ be the function satisfying:
\begin{equation}
  \label{eq:eq:gen_anal_dim_s1}
  \nbdirprj {j\#},,.(D_ju_j\cdot\mu_j)=D_\infty u\cdot\mu_\infty.
\end{equation}
Then for
$\mu_\infty$-a.e. $p\in X_\infty$, one has:
\begin{equation}
  \label{eq:gen_anal_dim_s2}
  D_ju_j=D_\infty u(p)\quad(\text{$\tilde\mu_j(p)$-a.e.})
\end{equation}
\end{lemma}
\begin{proof}
  We apply Lemma~\ref{lem_compatible_families}; thus it suffices to
  show that for $k\ge j$ one has:
  \begin{equation}
    \label{eq:gen_anal_dim_p1}
    D_ku_k\circ\pi_j^k=D_ju_j\quad(\text{$\mu_j$-a.e.})
  \end{equation}
  Let $S$ be the set of points $q\in X_j$ which do not belong to the
  $(n-1)$-skeleton of $X_j^{(m-j)}$ for any $m\ge j$. Note that
  \begin{equation}
    \label{eq:gen_anal_dim_p2}
    \mu_j(S)=\mu_k(\pi_j^k(S))=0.
  \end{equation}
  Now, for $\mu_j$-a.e.~$q$, $D_ju_j(q)$ equals
  $(u_j\circ\gamma_q)'(0)$ where $\gamma_q$ is a unit-speed horizontal
  segment with $\gamma_q(0)=q$ and along which $x_j\circ\gamma_q$ is
  non-decreasing. As the fibres of $\pi_j^k$ are finite, we also have
  that for $\mu_j$-a.e.~$q$, $D_ku_k(\pi_j^k(q))$ equals
  $(u_k\circ\gamma_{\pi_j^k(q)})'(0)$ where $\gamma_{\pi_j^k(q)}$ is a unit-speed horizontal
  segment with $\gamma_{\pi_j^k(q)}(0)=\pi_j^k(q)$, and along which $x_k\circ\gamma_{\pi_j^k(q)}$ is
  non-decreasing. If $q\in S$ for some $\delta>0$ one has:
  \begin{equation}
    \label{eq:gen_anal_dim_p3}
    \pi_j^k(\gamma_q(t))=\gamma_{\pi_j^k(q)}(t)\quad(|t|\le\delta);
  \end{equation}
  as $u_j=u_k\circ\pi_j^k$ we conclude that~\eqref{eq:gen_anal_dim_p1} holds.
\end{proof}

\subsection{Proof of Theorem~\ref{thm_lipschitz_light}}
\label{subsec_lipschitz_light_proof}

\begin{proof}[Proof of Theorem~\ref{thm_lipschitz_light}]
  We first prove (1). Note that Axioms~\textbf{(Ax6)}, \textbf{(Ax1)}
  imply the existence of a $C=C(\Delta,n)$ (independent of $j$) such
  that:
  \begin{equation}
    \label{eq:lipschitz_light_p1}
    \mu_j\left(B_{X_j}(p_j,r)\right)\ge
    C^{-1}\mu_j\left(B_{X_j}(p_j,2r)\right)\quad(\forall r\le
    4m^{-j+1}, p_j\in X_j).
   \end{equation}
  Assume that $r>m^{-j}$, and let $k$ be such that $m^{-k}<r\le
  m^{-k+1}$; and choose a point $p_{k-1}\in\ndirprj k-1,j,a.(p_j)$. As
  $\ndirprj k-1,j,.$ is $1$-Lipschitz:
  \begin{equation}
    \label{eq:lipschitz_light_p2}
    B_{X_{k-1}}(p_{k-1},r)\subset\ndirprj k-1,j,a.\left(B_{X_j}(p_j,r)\right);
  \end{equation}
  on the other hand, Proposition~\ref{prop_distance_lower_bound}
  implies that:
  \begin{multline}
    \label{eq:lipschitz_light_p3}
    \ndirprj k-1,j,a.\left(B_{X_j}(p_j,2r)\right)\subset\nbdirprj
    k-1,,a.\left(B_{\bar X_\infty}(\nbdirprj j,,.(p_j),2r)\right)\\\subset B_{X_{k-1}}(p_{k-1},4m^{-k+1});
  \end{multline}
  thus, combining~(\ref{eq:lipschitz_light_p2}),
  (\ref{eq:lipschitz_light_p3}) one gets:
  \begin{equation}
    \label{eq:lipschitz_light_p4}
    \begin{split}
      \mu_j\left(B_j(p_j,r)\right)&\ge\mu_{k-1}\left(B_{X_{k-1}}(p_{k-1},r)\right)\\&\ge
      C^{-2-\log_2m}\mu_{k-1}\left(B_{X_{k-1}}(p_{k-1},4m^{-k+1})\right)\\
      &\ge C^{-2-\log_2m}\mu_j\left(B_j(p_j,2r)\right).
    \end{split}
  \end{equation}
  We now prove (2), i.e.~that $\nbdirprj j,,.$ is a Lipschitz-light
  map. We will show that there is a universal constant $C$ such that
  for each $(k,j,p_\infty)\in\zahlen^2\times \bar X_\infty$, $\nbdirprj
  j,,a.(p_\infty)$ can be covered by a family of sets
  $\{\Omega_\alpha\}_\alpha$ such that:
  \begin{equation}
    \label{eq:lipschitz_light_p5}
    \diam\Omega_\alpha\le Cm^{-k},
  \end{equation}
  and
  \begin{equation}
    \label{eq:lipschitz_light_p6}
    d(\Omega_\alpha,\Omega_\beta)\ge m^{-k}\quad(\alpha\ne\beta).
  \end{equation}
  Note that by Corollary~\ref{cor:cell_dist} any two points of
  $\nbdirprj j,,a.(p_\infty)$ must belong to adjacent cells of $X_j$, so
  the case of interest is $k>j$.
  \par For each $n$-cell $\sigma$ of $X_j^{(k-j)}$ intersecting
  $\nbdirprj j,,a.(p_\infty)$ let $\Omega_\sigma$ denote the set of
  points of $\nbdirprj j,,a.(p_\infty)$ that can be connected to
  $\sigma$ using a chain $S=\{\sigma_0,\cdots,\sigma_L\}$ of cells of
  $X_j^{(k-j)}$, such that each $\sigma_i$ intersects $\nbdirprj
  j,,a.(p_\infty)$.
  From~\eqref{eq:generalized_metric_structure_s4} in Lemma~\ref{lem:generalized_metric_structure} we conclude that:
  \begin{equation}
    \label{eq:lipschitz_light_p9}
    \diam\Omega_\sigma\le 6C_0m^{-k}.
  \end{equation}
  On the other hand, if $\Omega_\sigma\ne\Omega_{\sigma'}$, then
  $\Omega_\sigma$ and $\Omega_{\sigma'}$ do not intersect adjacent cells
  of $X_j^{(k-j)}$ and hence:
  \begin{equation}
    \label{eq:lipschitz_light_p10}
    d_j(\Omega_\sigma,\Omega_{\sigma'})\ge m^{-k}.
  \end{equation}
  \par We now turn to the proof of (3). We first observe that if
  $\sigma$ is an $n$-cell of $X_j$, one has:
  \begin{equation}
    \label{eq:lipschitz_light_p11}
    \mu_j(\sigma)=w_\sigma m^{-jQ},
  \end{equation}
  where $w_\sigma$ is the weight in Axiom \textbf{(Ax6b)}. Let
  $\sigma_0$, $\sigma_1$ be $n$-cells of $X_j$. If
  \begin{equation}
    \label{eq:lipschitz_light_p12}
    d_j(\sigma_0,\sigma_1)\le m^{-j},
  \end{equation}
  the cells $\sigma_0$, $\sigma_1$ are adjacent and hence
  \begin{equation}
    \label{eq:lipschitz_light_p13}
    \frac{w_{\sigma_0}}{w_{\sigma_1}}\le\Delta
  \end{equation}
  by Axiom \textbf{(Ax6c)}. Assume now that
  $d_j(\sigma_0,\sigma_1)>m^{-j}$. As $X_j$ is connected
  (\textbf{(Ax1)}) there is an $N\ge 2$ such that:
  \begin{equation}
    \label{eq:lipschitz_light_p14}
    (N-1)m^{-j}<d_j(\sigma_0,\sigma_1)\le N m^{-j};
  \end{equation}
  let $k<0$ be such that $m^{-k-1}<N\le m^{-k}$. Using
  Proposition~\ref{prop_distance_lower_bound} we see that:
  \begin{equation}
    \label{eq:lipschitz_light_p15}
    d_k\left(\ndirprj j+k,j,a.(\sigma_0),\ndirprj
      j+k,j,a.(\sigma_1)\right)\le 3m^{-j-k}.
    \end{equation}
    Now by Axiom \textbf{(Ax3)}
    \begin{equation}
      \label{eq:lipschitz_light_p16}
      \ndirprj j+k,j,.:X_{j+k}^{(-k)}\to X_j
    \end{equation}
    is a surjective cellular map; moreover, we are assuming that
    $\ndirprj j+k,j,.$ is injective on the complement of the
    $(n-1)$-skeleton of $X_{j+k}^{(-k)}$. Thus, if
    $\mathring{\sigma_i}$ denotes the interior of $\sigma_i$, there is
    a unique $n$-cell $\tilde\sigma_i$ of $X_{j+k}$ such that:
    \begin{equation}
      \label{eq:lipschitz_light_p17}
      \ndirprj
      j+k,j,a.(\mathring{\sigma_i})\cap\tilde\sigma_i\ne\emptyset.
    \end{equation}
    By~(\ref{eq:lipschitz_light_p15}) $\tilde\sigma_0$,
    $\tilde\sigma_1$ are at combinatorial distance at most $3$ and
    hence:
    \begin{equation}
      \label{eq:lipschitz_light_p18}
      \frac{w_{\tilde\sigma_0}}{w_{\tilde\sigma_1}}\le\Delta^3;
    \end{equation}
    but $w_{\sigma_i}=w_{\tilde\sigma_i}$ and so the weights
    $w_{\sigma_0}$ and $w_{\sigma_1}$ are comparable up to a uniformly
    bounded multiplicative factor. We therefore find a universal
    constant $C$ such that, for each $j\in\zahlen$, each $p_j\in X_j$
    and each $r\le 3m^{-j+1}$ one has:
    \begin{equation}
      \label{eq:lipschitz_light_p19}
      \mu_j\left( B_{X_j}(p_j,r) \right)\approx_C r^{-jQ}.
    \end{equation}
    Suppose now that $m^{-k}<r\le m^{-k+1}$ and let $p_k\in\ndirprj
    k,j,a.(p_j)$; then, arguing as in the proof of (1), we obtain:
    \begin{equation}
      \label{eq:lipschitz_light_p20}
      B_{X_{k-1}}(p_{k-1},m^{-k})\subset\ndirprj k-1,j,a.\left(
        B_{X_j}(p_j,r) \right)\subset B_{X_{k-1}}(p_{k-1},3m^{-k+1});
    \end{equation}
    therefore, by enlarging $C$, we have
    that~(\ref{eq:lipschitz_light_p19}) holds also for
    $r>3m^{-j+1}$. This proves that each $(X_j,d_j,\mu_j)$ is
    Ahlfors-regular, where the constant in the Ahlfors-regularity
    condition is independent of $j$.
    \par We now show that $\nbdirprj j,,.$ is David-Semmes regular
    (where the constants again do not depend on $j$). Let $k<0$; as
    the map in~(\ref{eq:lipschitz_light_p16}) is injective on the
    complement of the $(n-1)$-skeleton of $X_{j+k}^{(-k)}$, and as by
    \textbf{(Ax1)} there is a uniform bound $\Delta$ on the
    cardinality of each link of $X_j$, we conclude that there is a
    universal constant $C=C(\Delta)$ such that for each $p_j\in X_j$
    one has that $\ndirprj j+k,j,a.(p_j)$ has cardinality at most
    $C$. As $k$ and $j$ are arbitary, we conclude that for
    $p_\infty\in X_\infty$ also $\ndirprj j,,a.(p_\infty)$ has
    cardinality at most $C$. Fix now $\bar p_\infty\in\bar X_\infty$
    and let $\{\Omega_\alpha\}$ be a family of subsets of $X_j$
    which cover $\nbdirprj j,,a.(\bar p_\infty)$, and which were obtained
    in the proof of (2). In constructing the $\{\Omega_\alpha\}_\alpha$ there
    was the freedom to choose a scale $m^{-k}$, which in this case we
    take to be $m^{-j}$, so that (\ref{eq:lipschitz_light_p5}),
    (\ref{eq:lipschitz_light_p6}) hold with $k=j$. For each
    $\Omega_\alpha$ there is a $p_{\infty,\sigma}\in[\bar p_{\infty}]$
    such that:
    \begin{equation}
      \label{eq:lipschitz_light_p21}
      \ndirprj j,,a.(p_{\infty,\sigma})\cap\Omega_\alpha\ne\emptyset.
    \end{equation}
    Fix now one $p_\infty\in[\bar p_{\infty}]$. Then by
    Corollary~\ref{cor:cell_dist} the sets $\ndirprj j,,a.(p_\infty)$
    and $\ndirprj j,,a.(p_{\infty,\sigma})$ must intersect adjacent
    cells of $X_j$ and so $\ndirprj
    j,,a.(p_{\infty})\cap\Omega_\alpha\ne\emptyset$. As distinct
    $\Omega_\alpha$'s are disjoint, we conclude that the cardinality
    of the set $\{\Omega_\alpha\}$ is at most $C$. Let now $\bar
    q_\infty\in B_{\bar X_\infty}(\bar p_{\infty},m^{-j})$; then
    $\nbdirprj j,,a.(\bar q_{\infty})$ is contained in a
    $3m^{-j}$-neighbourhood of $\bigcup_\alpha\Omega_\alpha$; as each
    set $\Omega_\alpha$ intersects $\ndirprj
    j,,a.(p_{\infty})$ we conclude that:
    \begin{equation}
      \label{eq:lipschitz_light_p22}
      \nbdirprj j,,a.(\bar q_{\infty})\subset\bigcup_{p\in\ndirprj j,,a.(p_\infty)}B(p,6(C_0+1)m^{-j});
    \end{equation}
    thus the David-Semmes regularity condition holds with constants
    $3(C_0+1)$ and $C$.
\end{proof}

\subsection{Preservation of admissibility under limits}
\label{subsec_limits_admissible}

\mbox{}

We now turn to the proof of (4) in
Theorem~\ref{thm_admissible_system}, i.e.~the stability of assertions
(1)--(3) under the operation of taking weak tangents. 
This is an immediate consequence of  the following lemma and (1)-(3) of
Theorem~\ref{thm_admissible_system}.

\begin{lemma}
\label{lem_weak_tangent_admissible}
Any weak tangent of an admissible system is, modulo rescaling, 
measure-preserving isometric to the  
limit space of some admissible
system.
\end{lemma}
\begin{proof}
Let $(Y,\nu,p)$ be a weak tangent of $(\bar X_\infty,\mu_\infty)$.
Thus there exists   a sequence of basepoints $\{p_\al\}$ in $\bar X_\infty$, as well 
as sequences $\{\la_\al\}$, $\{\la_\al'\}$ of scale factors, such that
the sequence $\{(\la_\al \bar X_\infty, \la_\al'\mu_\infty,p_\al)\}$
converges in the pointed measured Gromov-Hausdorff topology to the pointed
(doubling) metric measure space $(Y,\nu,p)$.   

For each $\al$, we have $\la_\al=a_\al m^{j_\al}$ for unique elements
$j_\al\in \Z$ and 
$a_\al\in [1,m)$.  After passing to a subsequence, we may assume that
$a_\al$ converges to some  $a_\infty\in [1,m]$.  Without loss of generality,
we may replace $\la_\al$
with $m^{j_\al}$, since the resulting sequence will converge to the same limit,
modulo rescaling by $a_\infty^{-1}$.  

Now observe that for every $\al$, the rescaled metric measure space
$(\la_\al \bar X_\infty,\la_\al'\mu_\infty)$ is measure-preserving isometric
to the limit space of an admissible system obtained from 
$\{(X_j,\mu_j)\}$  by shifting the indices by 
$j_\al$, and rescaling the measures;  moreover this new admissible
system satisfies Definition \ref{def_admissible_system} 
where the constants $m,m_v, H, \De$ are independent of $\al$.
Thus Lemma \ref{lem_weak_tangent_admissible}
is reduced to the following lemma.
\end{proof}

\begin{lemma}
Suppose $\{(X_{j,\al},\mu_{j,\al})\}_{\al\in\N}$ is a sequence of admissible systems with
uniform constants, and $q_{j,\al}\in X_{j,\al}$ are projection compatible
basepoints chosen such that
\begin{equation}
\label{eqn_ball_measure_controlled}
0<
\liminf_{\al} \mu_{0,\al}(B(q_{0,\al},1))\leq \limsup_{\al} \mu_{0,\al}(B(q_{0,\al},1))
<\infty\,.
\end{equation}
Then after passing to a subsequence, the sequences of pointed admissible
systems converge in a natural sense to a pointed admissible system
$(X_{j,\infty},\mu_{j,\infty},q_{j,\infty})$
whose  limit space
$(\bar X_{\infty,\infty},\mu_{\infty,\infty},q_{\infty,\infty})$ is measure-preserving
isometric to the pointed measured Gromov-Hausdorff limit of 
the sequence of pointed limit spaces 
$\{(\bar X_{\infty,\al},\mu_{\infty,\al},q_{\infty,\al})\}$.
\end{lemma}
\begin{proof}
This is a consequence of standard
finiteness/compactness arguments applied to the 
controlled-geometry complexes of the admissible systems, so we will be 
brief.

Pick $j\in\Z$.  Since the parameter $\De$ is independent of $\al$, for any $N$, there
are only finitely many possibilities for the combinatorial $N$-ball
centered at $q_{j,\al}$, up to a homeomorphism preserving the collection of distinguished
characteristic maps.  Therefore, after passing to a subsequence, there is a pointed
cell-complex $(X_{j,\infty},q_{j,\infty})$ with a collection of distinguished
characteristic maps, such that for all $N$ and large $\al$, there is  a 
homeomorphism $\Psi_{j,\al,N}$ from the combinatorial $N$-ball in $X_{j,\al}$
centered at $q_{j,\al}$ to the combinatorial $N$-ball in
$X_{j,\infty}$ centered at $q_{j,\infty}$, such that $\Psi_{j,\al,N}$
respects distinguished characteristic maps, 
$\Psi_{j,\al,N}(q_{j,\al})\ra q_{j,\infty}$ as $\al\ra\infty$, and
for all $N'\geq N$ the maps $\Psi_{j,\al,N}$, $\Psi_{j,\al,N'}$ are compatible
on the $N$-balls for large $\al$.

Using (\ref{eqn_ball_measure_controlled}), we get that the 
$\mu_{j,\al}$-measure of any $n$-cell of $X_{j,\al}$ containing $q_{j,\al}$
is controlled.  Therefore, passing to a subsequence again, there is a family
of measures $\{\mu_{j,\infty}\}$ and a compatible system of projection maps
$\{\pi_{j,\infty}^k:X_{j,\infty}\ra X_{k,\infty}\}$ which define an admissible
system, such that the maps $\{\Psi_{j,\al,N}\}$ are asymptotically
measure-preserving and 
compatible with projection. 

It follows from Proposition \ref{prop_distance_lower_bound} that for all
$R$, the $\hat d_{j,\al}$-ball  $B(q_{j,\al},R)\subset X_{j,\al}$ is a  Gromov-Hausdorff approximation to within
error $\lesssim m^{-j}$ of the $\bar d_{\infty,\al}$-ball 
$B(q_{\infty,\al},R)\subset \bar X_{\infty,\al}$,
so the maps $\{\Psi_{j,\al,N}\}$ induce the 
pointed measured Gromov-Hausdorff convergence  
$$
(\bar X_{\infty,\al},\bar d_{\infty,\al},\mu_{\infty,\al},q_{\infty,\al})
\lra 
(\bar X_{\infty,\infty},\bar d_{\infty,\infty},\mu_{\infty,\infty},q_{\infty,\infty})\,.
$$

\end{proof}

\bibliography{piex2}

\bibliographystyle{alpha}

\end{document}